\documentclass[a4paper, reqno]{amsart}

\usepackage{amsthm,amsfonts,amssymb,amsmath,amsxtra}
\usepackage[all]{xy}
\SelectTips{cm}{}
\usepackage{tikz}
\usepackage{verbatim}

\usepackage[margin=1.25in]{geometry}
\usepackage{mathrsfs}

\RequirePackage{xspace}
\RequirePackage{etoolbox}
\RequirePackage{varwidth}
\RequirePackage{enumitem}
\RequirePackage{tensor}
\RequirePackage{mathtools}
\RequirePackage{longtable}
\RequirePackage{multirow}
\RequirePackage{makecell}
\RequirePackage{bigints}
\RequirePackage{mathdots}

\usepackage{hyperref} %[backref=page]

\newcommand{\fka}{\ensuremath{\mathfrak{a}}\xspace}

\newcommand{\fkc}{\ensuremath{\mathfrak{c}}\xspace}

\newcommand{\fkg}{\ensuremath{\mathfrak{g}}\xspace}

\newcommand{\fkk}{\ensuremath{\mathfrak{k}}\xspace}

\newcommand{\fks}{\ensuremath{\mathfrak{s}}\xspace}

\newcommand{\BC}{\ensuremath{\mathbb{C}}\xspace}

\newcommand{\BE}{\ensuremath{\mathbb{E}}\xspace}
\newcommand{\BF}{\ensuremath{\mathbb{F}}\xspace}

\newcommand{\BL}{\ensuremath{\mathbb{L}}\xspace}
\newcommand{\BM}{\ensuremath{\mathbb{M}}\xspace}
\newcommand{\BN}{\ensuremath{\mathbb{N}}\xspace}

\newcommand{\BP}{\ensuremath{\mathbb{P}}\xspace}
\newcommand{\BQ}{\ensuremath{\mathbb{Q}}\xspace}

\newcommand{\BV}{\ensuremath{\mathbb{V}}\xspace}

\newcommand{\BX}{\ensuremath{\mathbb{X}}\xspace}

\newcommand{\BZ}{\ensuremath{\mathbb{Z}}\xspace}

\newcommand{\CE}{\ensuremath{\mathcal{E}}\xspace}
\newcommand{\CF}{\ensuremath{\mathcal{F}}\xspace}

\newcommand{\CL}{\ensuremath{\mathcal{L}}\xspace}
\newcommand{\CM}{\ensuremath{\mathcal{M}}\xspace}
\newcommand{\CN}{\ensuremath{\mathcal{N}}\xspace}
\newcommand{\CO}{\ensuremath{\mathcal{O}}\xspace}
\newcommand{\CP}{\ensuremath{\mathcal{P}}\xspace}

\newcommand{\CS}{\ensuremath{\mathcal{S}}\xspace}

\newcommand{\CZ}{\ensuremath{\mathcal{Z}}\xspace}

\newcommand{\RM}{\ensuremath{\mathrm{M}}\xspace}
\newcommand{\RN}{\ensuremath{\mathrm{N}}\xspace}

\newcommand{\RU}{\ensuremath{\mathrm{U}}\xspace}

\newcommand{\nat}{{\natural}}

\newcommand{\ad}{\mathrm{ad}}

\DeclareMathOperator{\Aut}{Aut}

\DeclareMathOperator{\charac}{char}

\DeclareMathOperator{\colspan}{colspan}

\newcommand{\corr}{\mathrm{corr}}

\newcommand{\del}{\operatorname{\partial Orb}}

\DeclareMathOperator{\diag}{diag}

\newcommand{\Dr}{\mathrm{Dr}}

\DeclareMathOperator{\End}{End}

\newcommand{\Fil}{\ensuremath{\mathrm{Fil}}\xspace}

\DeclareMathOperator{\Frob}{Frob}

\DeclareMathOperator{\Gal}{Gal}

\newcommand{\GL}{\mathrm{GL}}
\newcommand{\gl}{\frak{gl}}

\DeclareMathOperator{\Gr}{Gr}

\newcommand{\GU}{\mathrm{GU}}

\DeclareMathOperator{\Hom}{Hom}

\newcommand{\id}{\ensuremath{\mathrm{id}}\xspace}

\DeclareMathOperator{\im}{im}

\renewcommand{\i}{^{-1}}

\newcommand{\inj}{\hookrightarrow}
\DeclareMathOperator{\Int}{\ensuremath{\mathrm{Int}}\xspace}

\DeclareMathOperator{\length}{length}
\DeclareMathOperator{\Lie}{Lie}
\newcommand{\lInt}{\ensuremath{\text{$\ell$-}\mathrm{Int}}\xspace}
\newcommand{\loc}{\ensuremath{\mathrm{loc}}\xspace}

\newcommand{\naive}{\ensuremath{\mathrm{naive}}\xspace}

\newcommand{\OGr}{\mathrm{OGr}}
\DeclareMathOperator{\Orb}{Orb}
\DeclareMathOperator{\ord}{ord}

\DeclareMathOperator{\Ros}{Ros}

\newcommand{\red}{\ensuremath{\mathrm{red}}\xspace}

\DeclareMathOperator{\Res}{Res}
\newcommand{\rs}{\ensuremath{\mathrm{rs}}\xspace}

\newcommand{\SL}{{\mathrm{SL}}}
\DeclareMathOperator{\spann}{span}
\DeclareMathOperator{\Spec}{Spec}
\DeclareMathOperator{\Spf}{Spf}

\newcommand{\SO}{{\mathrm{SO}}}

\newcommand{\SU}{{\mathrm{SU}}}

\DeclareMathOperator{\sgn}{sgn}

\DeclareMathOperator{\tr}{tr}
\newcommand{\triv}{\mathrm{triv}}

\newcommand{\U}{\mathrm{U}}
\DeclareMathOperator{\uAut}{\underline{Aut}}

\DeclareMathOperator{\vol}{vol}

\DeclareMathOperator{\WT}{WT}

\newcommand{\wt}{\widetilde}
\newcommand{\wh}{\widehat}

\newcommand{\ov}{\overline}
\newcommand{\incl}{\hookrightarrow}
\newcommand{\lra}{\longrightarrow}

\newcommand{\bs}{\backslash}
\newcommand{\la}{\langle}
\newcommand{\ra}{\rangle}
\newcommand{\lv}{\lvert}
\newcommand{\rv}{\rvert}
\newcommand{\sidewedge}[1]{\sideset{}{#1}\bigwedge}

\newcommand{\ep}{\varepsilon}

\newtheorem{theorem}{Theorem}
\newtheorem{proposition}[theorem]{Proposition}
\newtheorem{lemma}[theorem]{Lemma}
\newtheorem{conjecture}[theorem]{Conjecture}
\newtheorem{`conjecture'}[theorem]{``Conjecture''}
\newtheorem{corollary}[theorem]{Corollary}

\theoremstyle{definition}
\newtheorem{definition}[theorem]{Definition}
\newtheorem{example}[theorem]{Example}

\newtheorem{remark}[theorem]{Remark}

\newenvironment{altenumerate}
   {\begin{list}
      {(\theenumi) }
      {\usecounter{enumi}
       \setlength{\labelwidth}{0pt}
       \setlength{\labelsep}{0pt}
       \setlength{\leftmargin}{0pt}
       \setlength{\itemsep}{\the\smallskipamount}
       \renewcommand{\theenumi}{\roman{enumi}}
      }}
   {\end{list}}
\newenvironment{altitemize}
   {\begin{list}
      {$\bullet$}
      {\setlength{\labelwidth}{0pt}
	   \setlength{\itemindent}{5pt}
       \setlength{\labelsep}{5pt}
       \setlength{\leftmargin}{0pt}
       \setlength{\itemsep}{\the\smallskipamount}
      }}
   {\end{list}}

\numberwithin{equation}{section}
\numberwithin{theorem}{section}

\newcommand{\aform}{\ensuremath{\langle\text{~,~}\rangle}\xspace}
\newcommand{\sform}{\ensuremath{(\text{~,~})}\xspace}

\newcounter{filler}

\setitemize[0]{leftmargin=*,itemsep=\the\smallskipamount}
\setenumerate[0]{leftmargin=*,itemsep=\the\smallskipamount}

\renewcommand{\to}{%
   \ifbool{@display}{\longrightarrow}{\rightarrow}%
   }
\let\shortmapsto\mapsto
\renewcommand{\mapsto}{%
   \ifbool{@display}{\longmapsto}{\shortmapsto}%
   }
\newlength{\olen}
\newlength{\ulen}
\newlength{\xlen}
\newcommand{\xra}[2][]{%
   \ifbool{@display}%
      {\settowidth{\olen}{$\overset{#2}{\longrightarrow}$}%
       \settowidth{\ulen}{$\underset{#1}{\longrightarrow}$}%
       \settowidth{\xlen}{$\xrightarrow[#1]{#2}$}%
       \ifdimgreater{\olen}{\xlen}%
          {\underset{#1}{\overset{#2}{\longrightarrow}}}%
          {\ifdimgreater{\ulen}{\xlen}%
             {\underset{#1}{\overset{#2}{\longrightarrow}}}
             {\xrightarrow[#1]{#2}}}}%
      {\xrightarrow[#1]{#2}}
   }
\makeatother
\newcommand{\xyra}[2][]{%
   \settowidth{\xlen}{$\xrightarrow[#1]{#2}$}%
   \ifbool{@display}%
      {\settowidth{\olen}{$\overset{#2}{\longrightarrow}$}%
       \settowidth{\ulen}{$\underset{#1}{\longrightarrow}$}%
       \ifdimgreater{\olen}{\xlen}%
          {\mathrel{\xymatrix@M=.12ex@C=3.2ex{\ar[r]^-{#2}_-{#1} &}}}%
          {\ifdimgreater{\ulen}{\xlen}%
             {\mathrel{\xymatrix@M=.12ex@C=3.2ex{\ar[r]^-{#2}_-{#1} &}}}
             {\mathrel{\xymatrix@M=.12ex@C=\the\xlen{\ar[r]^-{#2}_-{#1} &}}}}}%
      {\mathrel{\xymatrix@M=.12ex@C=\the\xlen{\ar[r]^-{#2}_-{#1} &}}}%
   }
\makeatletter
\newcommand{\xla}[2][]{%
   \ifbool{@display}%
      {\settowidth{\olen}{$\overset{#2}{\longleftarrow}$}%
       \settowidth{\ulen}{$\underset{#1}{\longleftarrow}$}%
       \settowidth{\xlen}{$\xleftarrow[#1]{#2}$}%
       \ifdimgreater{\olen}{\xlen}%
          {\underset{#1}{\overset{#2}{\longleftarrow}}}%
          {\ifdimgreater{\ulen}{\xlen}%
             {\underset{#1}{\overset{#2}{\longleftarrow}}}
             {\xleftarrow[#1]{#2}}}}%
      {\xleftarrow[#1]{#2}}
   }
\newcommand{\isoarrow}{%
   \ifbool{@display}{\overset{\sim}{\longrightarrow}}{\xrightarrow\sim}%
   }
\renewcommand{\lra}{%
   \ifbool{@display}{\longleftrightarrow}{\leftrightarrow}%
   }
\newcommand{\undertilde}{\raisebox{0.4ex}{\smash[t]{$\scriptstyle\sim$}}}

\begin{document}

\title[Regular formal moduli spaces and arithmetic transfer conjectures]{Regular formal moduli spaces and \\ arithmetic transfer conjectures}
\author{M. Rapoport}
\address{Mathematisches Institut der Universit\"at Bonn, Endenicher Allee 60, 53115 Bonn, Germany}
\email{rapoport@math.uni-bonn.de}
\author{B. Smithling}
\address{Johns Hopkins University, Department of Mathematics, 3400 N.\ Charles St.,\ Baltimore, MD  21218, USA}
\email{bds@math.jhu.edu}
\author{W. Zhang}
\address{Columbia University, Department of Mathematics, 2990 Broadway, New York, NY 10027, USA}
\email{wzhang@math.columbia.edu}

% \date{\today}

\begin{abstract}
We define various formal moduli spaces of $p$-divisible groups which are regular, and morphisms between them. We formulate arithmetic transfer conjectures, which are variants of the arithmetic fundamental lemma conjecture of \cite{Z12} in the presence of ramification. These conjectures include the AT conjecture of \cite{RSZ}. We prove these conjectures in low-dimensional cases. 

\end{abstract}

\maketitle

\tableofcontents

\section{Introduction}\label{Intro section}

In \cite{Z12} the third-named author proposed a relative trace formula approach to the arithmetic Gan--Gross--Prasad conjecture. In this context, he formulated the \emph{arithmetic fundamental lemma} (AFL) conjecture, cf.~\cite{Z12,RTZ}. The AFL conjecturally relates the special value of the derivative an orbital integral to an arithmetic intersection number on a \emph{Rapoport--Zink formal moduli space of $p$-divisible groups} (RZ space) attached to a unitary group. It is essential here that one is dealing with a situation that is unramified in every possible sense (the quadratic extension $F/F_0$ defining the unitary group is unramified, and the hermitian space is split; the function appearing in the derivative of the orbital integral is the characteristic function of a hyperspecial maximal compact subgroup, etc.). The AFL is proved for low ranks of the unitary group ($n=2$ and $3$) in \cite{Z12}, and for arbitrary rank $n$ and \emph{minuscule} group elements in \cite{RTZ}. 
A simplified proof for $n=3$ is given by Mihatsch in \cite{M-AFL}.
At present, the general case of the AFL seems out of reach, even though Yun has obtained interesting results concerning the function field analog \cite{Y12}. 

The present series of papers is devoted to investigating how the statement of the AFL has to be modified when the various unramifiedness hypotheses are dropped. In the context of the \emph{fundamental lemma} (FL) conjecture of Jacquet--Rallis, this question leads naturally to their \emph{smooth transfer} (ST) conjecture, cf.  \cite{JR}. In the arithmetic context, this question naturally leads  to the problem of formulating  \emph{arithmetic transfer} (AT) conjectures.  One goal of the present paper is to formulate such AT conjectures in as many cases as possible and to prove these conjectures in low dimension.

The search for such AT conjectures motivates the problem of defining RZ spaces with good properties and to construct morphisms between them beyond the unramified case.  The construction of such spaces and morphisms is the second goal of this paper.   The ground work for the construction of such RZ spaces has been laid in earlier papers on \emph{local models}. To be more specific, let us consider the case when the quadratic extension $F/F_0$ is ramified. In our previous paper \cite{RSZ}, we considered the problem of defining RZ spaces attached to a unitary group of signature $(1, n-1)$.  When $n$ is even, based on work on local models of Pappas and the first author \cite{PR}, we constructed such RZ spaces which are formally smooth. We termed this phenomenon \emph{exotic smoothness}, since smoothness is unexpected in the presence of ramification.  In the present paper, we complete the picture by again constructing such RZ spaces which are formally smooth when $n$ is odd, this time based on work on local models  of the second author \cite{S} (in our previous paper \cite{RSZ}, we only constructed an open subspace of these RZ spaces when $n$ is odd).

We stress that for applications to AT conjectures, it seems essential to have a functor description of the relevant RZ spaces. Correspondingly, it is essential to have a functor description of the relevant local models; the Pappas--Zhu definition through a closure operation in a mixed characteristic Beilinson--Gaitsgory degeneration of an affine Grassmannian \cite{PZ} (even though much more general) is not useful in this context. In fact, to define the morphisms between RZ spaces required for the AT conjecture in the case that $F/F_0$ is ramified and $n$ is even, we have to dig even deeper  into the theory of local models, and this constitutes the most difficult part of the present paper. The case when $F/F_0$ is unramified is much easier and is based on the work of Drinfeld \cite{D} and G\"ortz \cite{Goertz} on local models for $\GL_n$.  

Let us now describe the contents of the paper in more detail. 
\smallskip

Let $p$ be an odd prime number, and let $F_0$ be a finite extension of $\BQ_p$, with residue field $k$. Let $F/F_0$ be a quadratic field extension. We denote by $a\mapsto \ov a$ the non-trivial automorphism of $F/F_0$, and by $\eta=\eta_{F/F_0}$ the corresponding quadratic character on $F_0^\times$. Let $e:=(0,\dotsc,0,1)\in F_0^n$, and let $\GL_{n-1}\incl \GL_n$ be the natural embedding that identifies $\GL_{n-1}$ with the subgroup fixing $e$ under left multiplication, and fixing the transposed vector $\tensor[^t]{e}{}$ under right multiplication. Let 
\[
   S_n := \{\, s\in \Res_{F/F_0}\GL_n \mid s\ov s=1 \,\},
\]
which is acted on by $\GL_{n-1}$ by conjugation.
On the other hand, let $W_0$ and $W_1$ be the respective split and non-split $F/F_0$-hermitian spaces of dimension $n$. For $i \in \{0,1\}$, fix anisotropic vectors $u_i\in W_i$ of the same length, and denote by $W_i^\flat \subset W_i$ the orthogonal complement of the line spanned by $u_i$. The unitary group $\U(W_i^\flat)$ acts by conjugation on $\U(W_i)$.

Let us recall the matching relation between regular semi-simple elements of $S_n(F_0)$ and of $\U(W_0)(F_0)$ and $\U(W_1)(F_0)$. Here an element of  $S_n(F_0)$, resp.~of $\U(W_i)(F_0)$, is called \emph{regular semi-simple} (rs) if its orbit under $\GL_{n-1}$, resp.~$\U(W_i^\flat)$, is Zariski-closed of maximal dimension. We denote by $S_n(F_0)_\rs$ and $\U(W_i)(F_0)_\rs$ the subsets of regular semi-simple elements. 
For each $i$, choose a basis of $W_i$ by first choosing a basis of $W_i^\flat$ and then appending $u_i$ to it. This identifies $\U(W^\flat_i)(F_0)$ with a subgroup of $\GL_{n-1}(F)$ and  $\U(W_i)(F_0)$ with a subgroup of $\GL_{n}(F)$. An element $\gamma\in S_n(F_0)_\rs$ is then said to \emph{match} an element $g\in \U(W_i)(F_0)_\rs$ if these elements are conjugate under $\GL_{n-1}(F)$ when considered as elements in $\GL_n(F)$. The matching relation induces a bijection 
\[
   \big[S_n(F_0)_\rs\big] \simeq \big[ \U(W_0)(F_0)_\rs\big]\amalg \big[\U(W_1)(F_0)_\rs\big],
\]
cf.~\cite[\S2]{Z12}, where the brackets indicate the sets of orbits under $\GL_{n-1}(F_0)$, $\U(W_0^\flat)(F_0)$, and $\U(W_1^\flat)(F_0)$, respectively.

Dual to the matching of elements is the transfer of functions, which is defined through orbital integrals. For a function $f'\in C_c^\infty(S_n(F_0))$, an element $\gamma\in S_n(F_0)_\rs$, and a complex parameter $s\in\BC$, we define the weighted orbital integral
\[
   \Orb(\gamma,f', s) := \int_{\GL_{n-1}(F_0)}f'(h^{-1}\gamma h)\lvert\det h \rvert^s\eta(\det h) \, dh,
\]
as well as its special value 
\[
   \Orb(\gamma,f') := \Orb(\gamma,f', 0) .
\]
For later use in the arithmetic situation, we also introduce the special value of its derivative,
\[
   \del(\gamma,f') := \frac d{ds} \Big|_{s=0} \Orb(\gamma,f', s).
\]
Here the Haar measure on $\GL_{n-1}(F_0)$ is normalized so that $\vol(\GL_{n-1}(O_{F_0}))=1$. 
For a function $f_i\in C_c^\infty(\U(W_i)(F_0))$ and an element $g\in  \U(W_i)(F_0)_\rs$, we define the orbital integral
\[
   \Orb(g,f_i) := \int_{\U(W^\flat_i)(F_0)}f_i(h^{-1} g h)\, dh .
\]
Then the function $f'\in C_c^\infty(S_n(F_0))$ is said to \emph{transfer} to the pair of functions $(f_0,f_1)$ in $ C_c^\infty(\U(W_0)(F_0))\times  C_c^\infty(\U(W_1)(F_0))$
if, whenever $\gamma\in S_n(F_0)_\rs$ matches $g\in \U(W_i)(F_0)_\rs$,
\[
   \omega(\gamma)\Orb(\gamma,f')=\Orb(g,f_i).
\]
Here 
\[
   \omega\colon S_n(F_0)_\rs\to \BC^\times
\]
is a fixed transfer factor \cite[p.~988]{Z14}, and the Haar measures on $\U(W_i^\flat)(F_0)$ are fixed. For the particular transfer factor we will take in this paper, see \S\ref{trans factor}. 

The FL conjecture asserts a specific transfer relation in a completely unramified situation. Namely, assume that $F/F_0$ is unramified and that the special vectors $u_i$ have norm one, and normalize the Haar measure on $\U(W_0^\flat)(F_0)$ by giving a hyperspecial maximal compact subgroup volume one.  Then the FL conjecture asserts that, with respect to the ``natural'' transfer factor (see \eqref{sign} below), the characteristic function $\mathbf{1}_{S_n(O_{F_0})}$  transfers to $(\mathbf{1}_{K_0},0)$, where $K_0 \subset \U(W_0)(F_0)$ denotes a hyperspecial maximal open subgroup. 

By contrast, when $F/F_0$ is ramified, such natural choices do not exist. The ST conjecture asserts that for any $f'$, a transfer $(f_0,f_1)$ exists (non-uniquely), and that any pair $(f_0,f_1)$ arises as a transfer from some (non-unique) $f'$.  It is known to hold by \cite{Z14}. 

\smallskip

The arithmetic situation is analogous. For the AFL conjecture we assume, just as in the FL conjecture, that $F/F_0$ is unramified and that the special vectors $u_i$ have norm one. We take the same Haar measure on $\U(W^\flat_0)(F_0)$, and the same transfer factor $\omega$. Then the AFL conjecture asserts that 
\begin{equation}\label{IntroAFL}
   \omega(\gamma)\del \bigl(\gamma, \mathbf{1}_{S_n(O_{F_0})}\bigr) = -\Int(g)\cdot\log q
\end{equation}
whenever $\gamma\in S_n(F_0)_\rs$ matches $g\in \U(W_1)(F_0)_\rs$ (note that the FL conjecture asserts that $\Orb(\gamma, \mathbf{1}_{S_n(O_{F_0})}) = 0$ for such $\gamma$). Here $q$ denotes the number of elements in the residue field of $F_0$. 

The term $\Int(g)$ that appears in \eqref{IntroAFL} is an intersection number on an unramified RZ space. Let us recall its definition, cf.\ \cite{RSZ}.  Let $\breve F$ be the completion of a maximal unramified extension of $F$, and $O_{\breve F}$ its ring of integers.  For each $\Spf O_{\breve F}$-scheme $S$, we consider triples $(X,\iota,\lambda)$ over $S$, where $X$ is a formal $p$-divisible $O_{F_0}$-module of relative height $2n$ and dimension $n$, where $\iota \colon O_F\to \End(X)$ is an action of $O_F$ extending the $O_{F_0}$-action and satisfying the Kottwitz condition of signature $(1,n-1)$ on $\Lie(X)$ (cf.~\cite[\S 2]{KR-U1}), and where $\lambda$ is a principal polarization on $X$ whose Rosati involution induces the automorphism $a\mapsto \ov a$ on $O_F$ via $\iota$.  Over $S = \Spec \ov k$, there is a unique such triple $(\BX_n,\iota_{\BX_n},\lambda_{\BX_n})$ up to $O_F$-linear quasi-isogeny compatible with the polarizations.  We then denote by $\CN_n=\CN_{F/F_0,n}$ the formal scheme over $\Spf O_{\breve{F}}$ which represents the functor that associates to each $S$ the set of isomorphism classes of tuples $(X,\iota,\lambda,\rho)$, where the first three entries are as just given, and where $\rho\colon X\times_S \ov S\to \BX_n\times_{\Spec\ov k} \ov S$ is a \emph{framing} of the restriction of $X$ to the special fiber $\ov S$ of $S$, compatible with $\iota$ and $\lambda$ in a certain sense, cf.~\cite[\S2]{RTZ}. Then $\CN_n$ is formally smooth of relative formal dimension $n-1$ over $\Spf O_{\breve F}$. The automorphism group (in a certain sense) of the framing object $\BX_n$ can be identified with $\U(W_1)(F_0)$; it acts on $\CN_n$ by changing the framing. Let $\CE$ be the canonical lifting of the formal $O_F$-module of relative height one and dimension one over $\Spf O_{\breve F}$, with its canonical $O_F$-action $\iota_\CE$ and its natural polarization $\lambda_\CE$. There is a natural closed embedding of $\CN_{n-1}$ into $\CN_n$,
\[
   \delta_\CN\colon
	\xymatrix@R=0ex{
	   \CN_{n-1} \ar[r]  &  \CN_n\\
		(Y, \iota, \lambda) \ar@{|->}[r]  &  (Y \times \CE, \iota\times\ov{\iota}_\CE, \lambda\times\lambda_\CE).
	}
\]
Here $\ov{\iota}_\CE$ denotes the precomposition of $\iota_\CE$ with the nontrivial Galois automorphism on $O_F$. 
Let $\CN_{n-1,n} := \CN_{n-1}\times_{\Spf O_{\breve F}} \CN_n$, and let
$$
   \Delta\subset \CN_{n-1,n}
$$
denote the graph of $\delta_\CN$. Then $\Int(g)$ is defined as the intersection number of $\Delta$ with its translate under the automorphism $1\times g$ of $\CN_{n-1,n}$,
\[
   \Int(g) := \chi\bigl(\CN_{n-1,n}, \CO_\Delta\otimes^\BL \CO_{(1 \times g)\Delta} \bigr) .
\]
We also define \emph{homogeneous} and \emph{Lie algebra} versions of this intersection number, cf.\ \eqref{defintprod} and \eqref{lInt}.

It should be true in the situation of the AFL that for \emph{any} $f'\in C_c^\infty(S_n(F_0))$ with transfer  $(\mathbf{1}_{K_0},0)$, there exists a function $f'_{\corr}\in C_c^\infty(S_n(F_0))$ such that
\begin{equation*}
   \omega(\gamma)\del (\gamma, f') = -\Int(g)\cdot\log q+\omega(\gamma)\Orb (\gamma,f'_{\corr})
\end{equation*}
whenever $\gamma\in S_n(F_0)_\rs$ matches $g\in \U(W_1)(F_0)_\rs$.
This would follow from the AFL and a con\-jectural \emph{density principle} on weighted orbital integrals, cf.\ \cite[Conj.\ 5.16, Lem.\ 5.18]{RSZ}. 
This remark 
indicates
% shows
how
one should
% to
formulate AT conjectures.

\smallskip

The problem of formulating an AT conjecture arises when one drops the unramifiedness assumptions made in the AFL conjecture. We want to replace $\CN_n$ by variants where the quadratic extension $F/F_0$ is allowed to be ramified, and where separately the polarization is allowed to be non-principal.  However, the scope of the conjecture is limited by the fact that we want to keep the definition of $\Int(g)$ as an Euler--Poincar\'e characteristic of a derived tensor product. This forces on us the condition that the analog of the product $\CN_{n-1}\times_{\Spf O_{\breve F}} \CN_n$ is \emph{regular}. We do not know a systematic way of singling out such cases; in the present paper, we discuss those we have found. The principle underlying our examples is that when $F/F_0$ is unramified, we take the polarization to be \emph{principal} or \emph{almost principal} (or equivalently, \emph{$\varpi$-modular} or \emph{almost $\varpi$-modular}, where $\varpi$ denotes a uniformizer in $F_0$; cf.\ Remarks \ref{pi-modular unram} and \ref{almost pi-modular unram}). 
When $F/F_0$ is ramified, we take the polarization to be \emph{$\pi$-modular} or \emph{almost $\pi$-modular}, where $\pi$ denotes a uniformizer in $F$.

Let us now review case-by-case the variants of the AFL conjecture that we propose in this paper. 

We start with the case when $F/F_0$ is unramified.  Again let $\varpi$ denote a uniformizer in $F_0$.  To define a variant of $\CN_n$, we impose that the polarization $\lambda$ is \emph{almost principal}, i.e.\ that 
$$
\ker \lambda \subset X[\iota(\varpi)]\ \text{\emph{is of rank $q^2$.}}
$$
The corresponding formal scheme $\wt{\CN}_n$ has semi-stable reduction over $\Spf O_{\breve {F}}$ of relative formal dimension $n-1$ (cf.\ Theorem \ref{wtCN_n semistable}). For $n=2$, there is an isomorphism \cite{KR-alt}
\begin{equation*}
   \wt\CN_2 \simeq \wh\Omega^2_{F_0}\times_{\Spf O_{F_0}}\Spf O_{\breve {F}} ,
\end{equation*}
where $\wh\Omega^2_{F_0}$ denotes the formal scheme version of the Drinfeld halfspace corresponding to the local field $F_0$. Note that, in contrast to $\CN_n$, the automorphism group of the framing object for $\wt\CN_n$ identifies with $\U(W_0)(F_0)$, which therefore acts on $\wt\CN_n$.

 We define a closed embedding of formal schemes
\begin{equation*}
   \wt\delta_\CN \colon
	\xymatrix@R=0ex{
	   \CN_{n-1} \ar[r]  &  \wt\CN_n\\
		(X, \iota, \lambda) \ar@{|->}[r]
		   &  \bigl(X \times \ov\CE, \iota \times \ov\iota_{\CE},
                 \lambda \times \varpi\lambda_{\CE} \bigr) .
   }
\end{equation*}
We therefore obtain, as in the case of the AFL, a closed embedding 
\begin{equation*}
   \wt\Delta_\CN \colon \CN_{n-1} \xra{(\id_{\CN_{n-1}},\wt\delta_\CN)} \wt\CN_{n-1,n} := \CN_{n-1}\times_{\Spf O_{\breve F}}\wt\CN_n,
\end{equation*}
whose image we denote by
\[
   \wt\Delta := \wt\Delta_\CN (\CN_{n-1}).
\]
Here we point out that $\wt{\CN}_{n-1, n}$ is regular of formal dimension $2(n-1)$. Hence for $g \in \U(W_0)(F_0)_\rs$ we may set
\[
   \Int(g) := \chi\bigl(\wt\CN_{n-1,n}, \CO_{\wt\Delta} \otimes^\BL \CO_{(1 \times g)\wt\Delta} \bigr).
\]
 
In this case, we take the special vectors $u_i \in W_i$ to have norm $\varpi$. It follows that $W_1^\flat$ is split, and therefore contains self-dual lattices. We normalize the Haar measure on $\U(W_1^\flat)(F_0)$ by giving the stabilizer of such a self-dual lattice volume one. Consider the characteristic function $\mathbf{1}_{K'}\in C_c^\infty(S_n(F_0))$, where
  \begin{align*}
   K' := S_n(O_{F_0})\cap K_0(\varpi) ,
\end{align*}
where, in turn, $K_0(\varpi)$ denotes the subgroup of matrices in $\GL_n(O_F)$ which are congruent modulo $\varpi$ to a lower triangular block matrix with respect to the decomposition $O_F^n=O_F^{n-1}\oplus O_F e$. 
We conjecture that  $(-1)^{n-1}\mathbf{1}_{K'}$ transfers to $(0, \mathbf{1}_{K_1})$, where the open compact subgroup $K_1$ of $\U(W_1)(F_0)$ is the stabilizer of an almost self-dual lattice in $W_1$, cf.\ Conjecture \ref{FL almost self-dual}. We prove that this conjecture holds, provided that $q\geq n$ and that the Lie algebra version of the FL conjecture holds (cf.\ Theorem \ref{unram variant FL}). Hence this transfer relation holds for $p \gg n$ by \cite{Go,Y}, and for any odd $p$ when $n=2$ or $n=3$ by \cite{Z12}.  The 
AT conjecture  in the present situation is now as follows. 

\begin{conjecture}\label{Intro-ConjATunr} Let $F/F_0$ be unramified. 
\begin{altenumerate}
\renewcommand{\theenumi}{\alph{enumi}}
\item
For \emph{any} $f'\in C_c^\infty(S_n(F_0))$ with transfer  $(0, \mathbf{1}_{K_1})$, there exists a function $f'_{\corr}\in C_c^\infty(S_n(F_0))$ such that
\begin{equation*}
   \omega_S(\gamma)\del (\gamma, f') = -\Int(g)\cdot\log q+\omega_S(\gamma)\Orb (\gamma,f'_{\corr})
\end{equation*}
for any $\gamma\in S_n(F_0)_\rs$ matching an element $g\in \U(W_0)(F_0)_\rs$.
\item\label{Intro-ConjATunr AFL}
Suppose that $\gamma \in S(F_0)_\rs$ matches $g \in \U(W_0)(F_0)_\rs$. Then
\begin{equation}\label{IntroATunr}
   \omega_S (\gamma)\del\bigl(\gamma, (-1)^{n-1}\mathbf{1}_{K'}\bigr) = - \Int(g)\cdot \log q .
\end{equation}
\end{altenumerate}
\end{conjecture}

We also give a \emph{homogeneous version} of this conjecture, as well as a \emph{Lie algebra version}, cf.\ Conjecture \ref{conjunramified}.
Our main result on this conjecture is the following result (cf.\ \S\ref{results atc afl}).

\begin{theorem}\label{Intro-Thmunr}
Conjecture \ref{Intro-ConjATunr}(\ref{Intro-ConjATunr AFL}) holds true in the non-degenerate case if the AFL conjecture in its Lie algebra version (cf.\ Conjecture \ref{AFLconj}(\ref{AFLconj lie})) holds, provided that $q\geq n$. 

Let $F_0 = \BQ_p$. Then   Conjecture \ref{Intro-ConjATunr} holds true  for $n=2$ and $n=3$. 
\end{theorem}

Here by the ``non-degenerate case'' we mean that the intersection $\wt\Delta \cap (1\times g)\wt\Delta$ is an artinian scheme.  Let us also remark on the restriction to $F_0 = \BQ_p$ at the end of the theorem. Indeed, strictly speaking the RZ spaces $\CN_n$ and $\wt \CN_n$ introduced above only fall under the framework of \cite{RZ} when $F_0 = \BQ_p$.  Therefore, when working in contexts where formal schemes are present, we will tacitly take $q = p$ and $F_0 = \BQ_p$ throughout this paper (including the setting of ramified $F/F_0$ to be discussed below).  On the other hand, this restriction on $F_0$ should be removable due to recent work of Ahsendorf--Cheng--Zink \cite{ACZ} and Mihatsch \cite{M-Th}. In anticipation of this, we will continue to use the general notation $q$ and $F_0$, and we often refer to $O_{F_0}$-relative polarizations, relative Dieudonn\'e modules, etc.

Now  assume  that $F/F_0$ is \emph{ramified}, again with $\pi \in F$ a uniformizer. In this case we modify the definition of the formal moduli space $\CN_{F/F_0,n}=\CN_n$ by imposing  on the polarization $\lambda$ the condition
\[
   \ker(\lambda) \subset X[\iota(\pi)]
	\text{\emph{ is of rank }} q^{2\lfloor n/2 \rfloor}.
\]
In order to obtain formal schemes with reasonable properties, we impose additional conditions on the action of $O_{F}$ induced by $\iota$  on the Lie algebra of the $p$-divisible groups involved (the \emph{Pappas wedge condition} and (variants of) the \emph{spin condition}). When $n$ is odd, there is no longer a unique framing object over $\ov k$; however, the two possible choices lead to isomorphic moduli spaces, and we will therefore ignore this issue. In both the even and the odd case,  $\CN_n$ is formally smooth of relative formal dimension $n-1$ over $\Spf O_{\breve F}$ and  essentially proper (\emph{exotic smoothness}). In the case $n=2$, the formal scheme $\CN_2$ can be identified with the base change from $\Spf O_{\breve F_0}$ to $\Spf O_{\breve F}$ of the disjoint sum of two copies of the Lubin--Tate deformation space $\CM$, cf.\ Example \ref{CN_2}. 

In the case when $n$ is odd, the morphism $\delta_\CN\colon \CN_{n-1}\to\CN_n$ can be defined exactly as in the AFL case,  since then $2\lfloor\frac{n-1}{2}\rfloor = 2\lfloor \frac{n}{2}\rfloor$. As in the AFL, the group $\U(W_1)(F_0)$ acts on $\CN_n$, and we define $\Int(g)$ as before for $g\in \U(W_1)(F_0)_\rs$.  We take the special vectors to have norm one. We recall from \cite{RSZ} our AT conjecture in this case (cf.\ \S\ref{s:ATCodd}). 

\begin{conjecture}\label{Intro-Conjramodd}
Let $F/F_0$ be ramified, and let $n\geq 3$ be odd. 
\begin{altenumerate}
\renewcommand{\theenumi}{\alph{enumi}}
\item
For any $f'\in C^\infty_c(S_n(F_0))$  with transfer $(\mathbf{1}_{K_0}, 0)$, there exists a function  $f'_\corr\in C^\infty_c(S_n(F_0))$ such that 
$$
2 \omega(\gamma)\del(\gamma,f') = -\Int(g)\cdot\log q + \omega(\gamma)\Orb(\gamma,f'_\corr)
$$ for any $\gamma\in S_n(F_0)_\rs$ matching an element $g\in \U(W_1)(F_0)_\rs$.
\item
There exists a function $f'\in C^\infty_c(S_n(F_0))$ with transfer $(\mathbf{1}_{K_0}, 0)$ such that 
\begin{equation}\label{IntroATramodd}
   2 \omega(\gamma)\del(\gamma,f')=-\Int(g)\cdot\log q 
\end{equation}
for any $\gamma\in S_n(F_0)_\rs$ matching an element $g\in \U(W_1)(F_0)_\rs$.
\end{altenumerate}
\end{conjecture}

Here $K_0$ denotes the maximal compact subgroup stabilizing an \emph{almost $\pi$-modular lattice} $\Lambda_0$ in $W_0$. The Haar measure on $\U(W^\flat_0)(F_0)$ is defined by $\vol K^\flat_0=1$ for a special maximal compact subgroup $K^\flat_0$ of $\U(W^\flat_0)(F_0)$. 

In \cite{RSZ}, we also formulate homogeneous and Lie algebra versions of this conjecture, comp.\ \S\ref{s:ATCodd}; and we prove this conjecture when $F_0 = \BQ_p$ and $n=3$.

In the case when $n$ is even, it is much less obvious how to define the morphism $\delta_\CN\colon \CN_{n-1}\to\CN_n$.  Indeed, the definition of $\delta_\CN$ in the AFL and odd ramified setting, transposed to the present case, produces $p$-divisible groups with the wrong polarization type to give points on $\CN_n$.  More precisely, consider the moduli problem of triples $(X,\iota,\lambda)$ as in the definition of $\CN_n$, except where the polarization $\lambda$ satisfies
\[
   \ker(\lambda) \subset X[\iota(\pi)]
	\text{\emph{ is of rank }} q^{n-2}.
\]
As in the case of $\CN_n$ when $n$ is odd, up to isogeny there are two such triples over $\ov k$, only one of which is isogenous ($O_F$-linearly and compatibly with the polarizations) to the framing object for $\CN_n$. Taking this as the framing object, we obtain a formal moduli space $\CP_n$ as before. The formula defining $\delta_\CN$ in the AFL and odd ramified settings then defines a morphism
\[
   \wt\delta_\CN\colon \CN_{n-1} \to \CP_n.
\]
We obtain a morphism from $\CP_n$ to $\CN_n$ --- in fact, two of them --- as follows. Let $\CP_n'$ denote the moduli space over $\Spf O_{\breve F}$ of tuples $(X,\iota,\lambda,\rho,X',\iota',\lambda',\rho',\phi)$,
where first four entries are a point on $\CN_n$, the second four entries are a point on $\CP_n$, and $\phi\colon X \to X'$ is an isogeny of degree $q$ lifting a fixed isogeny between the framing objects of $\CN_n$ and $\CP_n$.  Then $\CP_n'$ is again an RZ space, and there are evident projections
\begin{equation*}
   \CP_n' \to \CN_n 
   \quad\text{and}\quad
	\CP_n' \to \CP_n.
\end{equation*}
We show in Proposition \ref{even n decomp lem} that $\CN_n$ naturally decomposes into a disjoint union $\CN_n^+ \amalg \CN_n^-$ of open and closed formal subschemes (this generalizes the case $n = 2$ mentioned above).  Pulling back along the projection $\CP_n' \to \CN_n$, we obtain a decomposition $\CP_n' = (\CP_n')^+ \amalg (\CP_n')^-$.  The key geometric result is now that this decomposition presents $\CP_n'$ as a trivial double cover of $\CP_n$, which in turn gives rise to two embeddings of $\CN_{n-1}$ into $\CN_n$, cf.\ Theorem \ref{wtCN'pm isom} and Proposition \ref{delta^pm closed emb}.

\begin{theorem}\label{intro CP' isom}
The projection $\CP_n' \to \CP_n$ maps each of the summands $(\CP_n')^\pm$ isomorphically to $\CP_n$.  Denoting the inverse by $\psi^\pm \colon \CP_n \isoarrow (\CP_n')^\pm$, the composite
\[
   \delta_\CN^\pm \colon \CN_{n-1} \xra{\wt\delta_\CN} \CP_n \xra{\psi^\pm} (\CP_n')^\pm \to \CN_n^\pm
\]
is a closed embedding.
\end{theorem}

In other words, the first assertion in the theorem says that, loosely speaking, given any point $(X',\iota',\lambda',\rho')$ on $\CP_n$ over a connected base, there are exactly two ways to extend it to an isogeny chain $X \xra\phi X'$ such that the composite $X \xra\phi X' \xra{\lambda'} (X')^\vee \xra{\phi^\vee} X^\vee$ is a $\pi$-modular polarization $\lambda$ on $X$.  This can be viewed as an analog of the linear algebra fact that in a split, even-dimensional $F/F_0$-hermitian space, every $O_F$-lattice $\Lambda'$ such that $\Lambda' \subset (\Lambda')^\vee \subset^2 \pi\i\Lambda'$ contains exactly two $\pi$-modular lattices $\Lambda$ (and, furthermore, this remark essentially \emph{is} the proof that $\CP_n' \to \CP_n$ is a double cover on $\ov k$-points, via the interpretation of both sides in terms of Dieudonn\'e modules).  The fact that the ``simple'' way of defining $\delta_\CN$ from before does not work in the even ramified case can also be understood in terms of the group theory of lattice stabilizers; see Remark \ref{group theoretic}.

Using Theorem \ref{intro CP' isom}, we define $\Delta$ to be the union of the graph subschemes of $\delta^+_\CN$ and  $\delta^-_\CN$ inside $\CN_{n-1, n}=\CN_{n-1}\times_{\Spf O_{\breve F}}\CN_n$.
As in the case of odd $n$, the group $\U(W_1)(F_0)$ acts on $\CN_n$, and we then define $\Int(g)$ as before, for $g\in \U(W_1)(F_0)_\rs$. We take the special vectors $u_i$ of norm $-1$.  We normalize the Haar measure on $\U(W_0^\flat)(F_0)$ by giving the stabilizer $K_0^\flat$ of an almost $\pi$-modular lattice $\Lambda_0^\flat$ volume one. Consider the two $\pi$-modular lattices $\Lambda_0^\pm$ contained in $\Lambda_0^\flat\oplus O_F u_0$, and denote by $K_0^\pm$ their stabilizers in $\U(W_0)(F_0)$. We consider $K_0^\flat$ as a subgroup of $\U(W_0)(F_0)$. 
\begin{conjecture}\label{Intro-Conjrameven}
Let $F/F_0$ be ramified, and let $n\geq 2$ be even. 
\begin{altenumerate}
\renewcommand{\theenumi}{\alph{enumi}}
\item
For any $f'\in C^\infty_c(S_n(F_0))$  with transfer $(\mathbf{1}_{K_0^\flat K_0^+}+\mathbf{1}_{K_0^\flat K_0^-}, 0)$, there exists a function  $f'_\corr\in C^\infty_c(S_n(F_0))$ such that 
$$
2 \omega(\gamma)\del(\gamma,f') = -\Int(g)\cdot\log q + \omega(\gamma)\Orb(\gamma,f'_\corr)
$$ for any $\gamma\in S_n(F_0)_\rs$ matching an element $g\in \U(W_1)(F_0)_\rs$.
\item
There exists a function $f'\in C^\infty_c(S_n(F_0))$ with transfer $(\mathbf{1}_{K_0^\flat K_0^+}+\mathbf{1}_{K_0^\flat K_0^-}, 0)$ such that 
\begin{equation}\label{IntroATrameven}
   2 \omega(\gamma)\del(\gamma,f')=-\Int(g)\cdot\log q 
\end{equation}
for any $\gamma\in S_n(F_0)_\rs$ matching an element $g\in \U(W_1)(F_0)_\rs$.
\end{altenumerate}
\end{conjecture}
 We also formulate homogeneous and Lie algebra versions of this conjecture (cf.\ Conjecture \ref{conjram v3}), and we prove them all in the case of minimal rank (cf.\ \S\ref{even ram proof}).
 
\begin{theorem}
Let $F_0 = \BQ_p$.  Then Conjecture \ref{Intro-Conjrameven} holds true for   $n = 2$.
\end{theorem}

Note that there is a discrepancy of a factor of $2$ in the statements of the AFL conjecture \eqref{IntroAFL} and the AT conjecture \eqref{IntroATunr}  on the one hand, and of the AT conjectures \eqref{IntroATramodd} and \eqref{IntroATrameven} on the other.  This is a genuine difference between the unramified and ramified cases, which finds its justification in the global comparison between the height pairing and the derivative of a relative trace formula, cf.\ \cite{RSZglob}. 

There are two further cases of AT statements, which we formulate in \S\ref{s:ATC ram n=2}. In both of these, $F/F_0$ is ramified and $n=2$; and, in fact, we prove them in \S\ref{s:ramifiedn=2}. In these cases we impose that the polarization $\lambda$ in the moduli problem is \emph{principal}. Then there are two possibilities for the framing objects, and the corresponding formal schemes $\wt\CN_2^{(0)}$ and $\wt\CN_2^{(1)}$ (which are defined over $\Spf O_{\breve F_0}$) are genuinely different. In fact, there are natural identifications 
\[
   \wt\CN_2^{(0)}\simeq \wh{\Omega}^2_{F_0}\times_{\Spf O_{F_0}}\Spf O_{\breve F_0}
   \quad\text{and}\quad
	\wt\CN_2^{(1)}\simeq \CM_{\Gamma_0(\varpi)} .
\]
Here $\wh\Omega^2_{F_0}$ again denotes the formal scheme version of the Drinfeld halfspace corresponding to the local field $F_0$, and $\CM_{\Gamma_0(\varpi)}$ denotes the formal deformation space of a $\varpi$-isogeny between Lubin--Tate formal groups (after base change to $\Spf O_{\breve F}$, the latter case recovers the formal scheme $\CP_2$ from above). We refer to the body of the text for the precise AT statements in these cases (they fall formally somewhat outside the framework of the general conjectures we make in Conjectures \ref{Intro-ConjATunr}, \ref{Intro-Conjramodd}, and \ref{Intro-Conjrameven}). 

\smallskip
We now comment on the proofs of our results.

In the situation of Theorem \ref{Intro-Thmunr}, the field extension $F/F_0$ is unramified, while the polarization in the moduli problem for $\wt\CN_n$ is taken to be almost principal, reflecting the fact the compact open subgroup $K_1$ is  the stabilizer of an almost self-dual lattice in $W_1$. The key observation of our proof is that, although the automorphism groups of the framing objects for $\CN_n$ and $\wt\CN_n$ are not directly related, their Lie algebras are. Therefore we first use the Cayley transform to reduce the group version of the AT conjecture to the corresponding Lie algebra version, and then to the Lie algebra version of the AFL conjecture. 

In the situation of Conjecture \ref{Intro-Conjrameven} when $n=2$, the field extension $F/F_0$ is ramified, and the moduli space $\CN_2$ can be identified with the base change from $\Spf O_{\breve F_0}$ to $\Spf O_{\breve F}$ of the disjoint union of two copies of the Lubin--Tate deformation space $\CM$. On the geometric side, Gross's theory of the canonical lifting allows us to compute explicitly the intersection number $\Int(g)$. On the analytic side, we develop a germ expansion of the derivative of the orbital integrals around every semi-simple (but irregular) point. The irregular orbital integrals of the test function on the unitary side completely determine the derivative of the orbital integrals, up to an orbital integral function. 

The proofs of the two sporadic cases of AT theorems stated in  \S\ref{s:ATC ram n=2} are parallel to those of Conjecture \ref{Intro-Conjrameven} for $n=2$. 

Let us also comment on the proof of Theorem \ref{intro CP' isom}, which holds for any even $n$.  The proof of the first assertion consists of two steps, the first of which is to show, via a straightforward Dieudonn\'e module argument, that the morphism in question induces a bijection on the underlying point sets of the formal schemes.  The remaining step, which is more difficult, is to show the the corresponding morphism on local models is an isomorphism, and hence the two sides have the same deformation theories.  Here the ``refined'' spin condition introduced in \cite{S} plays an important role, insofar as it is necessary for flatness of the local models, cf.\ Remarks \ref{Yu flatness} and \ref{not flat}.  We also remark that this isomorphism between local models can be viewed as a reflection of the group-theoretic fact that, for an $O_F$-lattice $\Lambda'$ in a split, even-dimensional $F/F_0$-hermitian space such that $\Lambda' \subset (\Lambda')^\vee \subset^2 \pi\i\Lambda'$, and for $\Lambda$ one of the two $\pi$-modular lattices contained in $\Lambda'$, the connected stabilizer of $\Lambda'$ in the unitary group is the same as the common stabilizer group of $\Lambda$ and $\Lambda'$.
We prove the second assertion in Theorem \ref{intro CP' isom} in a similar way, via a Dieudonn\'e module step followed by a local model step.

\smallskip
We finally give an overview of the contents of this paper, which consists of four parts. 

In Part 1, we give the group-theoretic setup (in its homogeneous, its inhomogeneous, and its Lie algebra versions), and we review the statements of the FL conjecture and the AFL conjecture. 

In Part 2, we introduce the RZ spaces that will appear in the various AT conjectures. Here the auxiliary spaces in \S\ref{aux spaces} are introduced because they are needed to construct the morphisms $\delta^\pm_\CN$ mentioned above. 

In Part 3, we formulate case by case the AT conjectures mentioned above. 

In Part 4, we prove the AT conjectures in some instances. Besides the general result of Theorem \ref{Intro-Thmunr}, and after the result of \cite{RSZ} concerning the case when $F/F_0$ is ramified and $n=3$, these instances all occur for $n=2$.

\subsection*{Acknowledgements} We are grateful to U.~G\"ortz, X.~He, and S.~Yu for helpful discussions.   We also acknowledge the hospitality of the ESI (Vienna) and the MFO (Oberwolfach), where part of this work was carried out.  We finally thank the referee for his/her remarks on the text.

M.R. is supported by a grant from  the Deutsche Forschungsgemeinschaft through the grant SFB/TR 45. B.S. is supported by a Simons Foundation grant \#359425 and an NSA grant H98230-16-1-0024. W.Z. is supported by NSF
grants DMS \#1301848 and \#1601144, and by a Sloan research fellowship.

\subsection*{Notation}
Throughout the paper we fix an odd prime number $p$.  We let $F_0$ be a finite extension of $\BQ_p$, and we denote by $O_{F_0}$ its ring of integers, by $\varpi$ a uniformizer, and by $k$ its residue field. We set $q := \# k$ (a power of $p$), and we fix an algebraic closure $\ov k$ of $k$.  We write $v$ for the normalized (i.e.~$\varpi$-adic) valuation on $F_0$.  
We let $F$ be a quadratic extension of $F_0$, with ring of integers $O_F$ and residue field $k_F$.
When $F/F_0$ is unramified, we take $\pi := \varpi$ as a uniformizer for $F$; and when $F/F_0$ is ramified, since $p \neq 2$, we may and will choose a uniformizer $\pi$ for $F$ such that (after possibly changing $\varpi$) $\pi^2=\varpi$.  We denote by $a\mapsto \ov a$ the nontrivial automorphism of $F/F_0$, and by  $\eta = \eta_{F/F_0}$ the corresponding  quadratic character on $F_0^\times$.  
We denote the group of norm one elements in $F^\times$ by
\[
   F^1 := \{\, a \in F \mid a \ov a = 1 \,\}.
\]
We denote by $\breve F_0$ the completion of a maximal unramified extension of $F_0$, and by $\breve F$ the analogous object for $F$; thus $\breve F/\breve F_0$ is an extension of degree $1$ or $2$ according as $F/F_0$ is unramified or ramified.  Given a scheme $S$ over $\Spec O_{\breve F}$ (or more commonly for us, over $\Spf O_{\breve F}$), we denote its special fiber by
\[
   \ov S := S \times_{\Spec O_{\breve F}} \Spec \ov k.
\]

As stated above, when working in an algebro-geometric context where formal schemes are present, we always understand that $q = p$ and $F_0 = \BQ_p$.

A polarization on a $p$-divisible group $X$ is an anti-symmetric isogeny $X \to X^\vee$, where $X^\vee$ denotes the dual.  For any polarization $\lambda$ on $X$, we denote the Rosati involution on $\End^\circ(X)$ by
\[
   \Ros_\lambda\colon f \mapsto \lambda^{-1} \circ f^\vee \circ \lambda;
\]
here and elsewhere the superscript $\circ$ denotes the operation $-\otimes_\BZ \BQ$.  For any quasi-isogeny $\rho \colon X \to Y$ and polarization $\lambda$ on $Y$, we define the pullback quasi-polarization
\[
   \rho^* (\lambda) := \rho^\vee \circ \lambda \circ \rho
\]
on $X$.

We denote by $\BE$ ``the'' formal $O_{F_0}$-module of relative height $2$ and dimension $1$ over $\Spec \ov k$.  We set
\[
   O_D := \End_{O_{F_0}}(\BE)
	\quad\text{and}\quad
	D := O_D \otimes_{O_{F_0}} F_0.
\]
Thus $D$ is ``the'' quaternion division algebra over $F_0$, and $O_D$ is its maximal order.  We make $\BE$ into a formal $O_F$-module of relative height $1$ by fixing any $F_0$-embedding
\begin{equation*}
   \iota_\BE\colon F \to D.
\end{equation*}
When $F/F_0$ is unramified, the requirement that $O_F$ acts on $\Lie\BE$ via the structure map dictates an embedding of $k_F$ into $\ov k$, and hence an embedding $F \to \breve F_0$, via which we take $\breve F = \breve F_0$. We denote the main involution on $D$ by $c \mapsto \ov c$, and the reduced norm by $\RN$.  We also write $\RN$ for the norm map $F^\times \to F_0^\times$; of course, all of this notation is compatible with the embedding $\iota_\BE\colon F \to D$.  
When $F/F_0$ is ramified, we write $D^-$ for the $-1$-eigenspace in $D$ of the conjugation action of $\iota_\BE(\pi)$, so that $D = F \oplus D^-$.
  We fix an $O_{F_0}$-linear principal polarization
(any two of which differ by an $O_{F_0}^\times$-multiple)
\[
   \lambda_\BE\colon \BE \isoarrow \BE^\vee.
\] 
The Rosati involution $\Ros_{\lambda_\BE}$ induces the main involution on $D$, and hence the nontrivial Galois involution on $F$ via $\iota_\BE$.

We denote by $\CE$ the canonical lift of $\BE$ over $\Spf O_{\breve F}$ with respect to $\iota_\BE$, equipped with its $O_F$-action $\iota_\CE$, $O_F$-linear framing isomorphism $\rho_\CE\colon \CE_{\ov k} \isoarrow \BE$, and principal polarization $\lambda_\CE$ lifting $\rho_\CE^*(\lambda_\BE)$.  We denote by $\ov\BE$ the same object as $\BE$, except where the $O_F$-action $\iota_{\ov\BE}$ is equal to the precomposition of $\iota_\BE$ by the nontrivial automorphism of $F/F_0$; and ditto for $\ov\CE$ and $\iota_{\ov\CE}$ in relation to $\CE$ and $\iota_\CE$.  On the level of $O_{F_0}$-modules, we set
\[
   \lambda_{\ov\BE} := \lambda_\BE,
	\quad
	\lambda_{\ov\CE} := \lambda_\CE,
	\quad\text{and}\quad
	\rho_{\ov\CE} := \rho_\CE.
\]
Note that $\ov\BE$ is \emph{not} a formal $O_{F_0}$-module when $F/F_0$ is unramified, provided we keep the same map $O_F \to \ov k$ dictated above; and $\ov\CE$ is never a formal $O_{F_0}$-module.

Given a $p$-divisible group $X$ over $\Spec \ov k$ equipped with an $O_F$-action $\iota$, we define
\begin{equation}\label{defV}
	\BV(X) := \Hom_{O_F}^\circ\bigl(\ov \BE, X\bigr).
\end{equation}
When $X$ is equipped with a polarization $\lambda$ such that $\Ros_\lambda$ induces the nontrivial Galois involution on $O_F$ (via $\iota$), $\BV(X)$ carries a natural $F/F_0$-hermitian form, cf.~e.g.~\cite[Def.~3.1]{KR-U1}: for $x,y \in \BV(X)$, the composition
\[
   \ov\BE \xra x X \xra\lambda X^\vee \xra{y^\vee} \ov\BE^\vee \xra{\lambda_{\ov\BE}^{-1}} \ov\BE
\]
lies in $\End_{O_F}^\circ(\ov\BE)$, and hence identifies via $\iota_{\ov\BE}$ with an element in $F$, which we define to be the pairing of $x$ and $y$.

Given modules $M$ and $N$ over a ring $R$, we write $M \subset^r N$ to indicate that $M$ is an $R$-submodule of $N$ of finite colength $r$.  Typically $R$ will be $O_F$ or $O_{\breve F}$, and the quotient $N/M$ will be a vector space over the residue field of dimension $r$.  When $\Lambda$ is an $O_F$-lattice in an $F/F_0$-hermitian space, we denote the dual lattice with respect to the hermitian form by $\Lambda^\vee$, and we call $\Lambda$ a \emph{vertex lattice of type $r$} if $\Lambda \subset^r \Lambda^\vee \subset \pi^{-1} \Lambda$.  Note that this terminology differs slightly from e.g.~\cite{KR-U1,RTW}.  Of course, a \emph{vertex lattice} is a vertex lattice of type $r$ for some $r$.  Let us single out the following special cases.  A \emph{self-dual} lattice is, of course, a vertex lattice of type $0$.  An \emph{almost self-dual} lattice is a vertex lattice of type $1$.  At the other extreme, a vertex lattice $\Lambda$ is \emph{$\pi$-modular} if $\Lambda^\vee = \pi^{-1} \Lambda$, and \emph{almost $\pi$-modular} if $\Lambda \subset \Lambda^\vee \subset^1 \pi^{-1}\Lambda$.

Given a variety $V$ over $\Spec F_0$, we denote by $C_c^\infty(V)$ the set of locally constant, compactly supported functions on the space $V(F_0)$ relative to the $\varpi$-adic topology. 

We write $1_n$ for the $n \times n$ identity matrix.  We use a subscript $S$ to denote base change to a scheme (or other object) $S$, and when $S = \Spec A$, we often use a subscript $A$ instead.

\part{Setup and background}

In this first part of the paper we describe the group-theoretic setup involved in our various AT conjectures, and we review the AFL conjecture of the third author \cite{Z12}.

\section{Group-theoretic setup and definitions}\label{setup}

We begin by recalling the setup of \cite{RSZ}, except we make no assumption on the ramification of $F/F_0$.  Let $n \geq 2$ be an integer.

\subsection{Homogeneous group setting}\label{setup homog}
Let
\[
   e := (0,\dotsc,0,1) \in F^n,
\]
and consider the embedding of algebraic groups over $F$,
\begin{equation}\label{GL_n-1 emb}
	\begin{gathered}
   \xymatrix@R=0ex{
	   \GL_{n-1} \ar[r]  &  \GL_n\\
		\gamma_0 \ar@{|->}[r]  &  \diag(\gamma_0,1)
	}
	\end{gathered};
\end{equation}
this identifies $\GL_{n-1}$ with the subgroup of points $\gamma$ in $\GL_n$ such that $\gamma e = \tensor[^t]{\gamma}{} e = e$.  Next define
the
algebraic groups over $F_0$,
\begin{align*}
   G' &:= \Res_{F/F_0}(\GL_{n-1} \times \GL_n),\\
	H_1' &:= \Res_{F/F_0} \GL_{n-1},\\
	H_2' &:= \GL_{n-1}\times\GL_n.
\end{align*}
We embed $H_1'$ in $G'$ by taking the graph of the map \eqref{GL_n-1 emb}, and we embed $H_2'$ in $G'$ via the evident natural map. Let
\[
   H_{1,2}' := H_1' \times H_2'.
\]
We consider the natural right action of $H_{1,2}'$ on $G'$,
\[
   \gamma \cdot (h_1,h_2) = h_1^{-1}\gamma h_2.
\]
We say that an element $\gamma \in G'(F_0)$ is \emph{regular semi-simple} if it is regular semi-simple for this action, i.e.~its orbit under $H_{1,2}'$ is Zariski-closed in $G'$, and its stabilizer in $H_{1,2}'$ is of minimal dimension.  In the situation at hand, it is equivalent that $\gamma$ have closed orbit and trivial stabilizer, which follows from \cite[Th.~6.1]{RS}. We denote by $G'(F_0)_\rs$ the set of regular semi-simple elements in $G'(F_0)$.

Next let $W$ be an $F/F_0$-hermitian space of dimension $n$.  Up to isometry there are two possibilities for $W$, the split and non-split cases, and we write $\chi(W) = 1$ or $\chi(W) = -1$ accordingly.  These satisfy the formula
\begin{equation}\label{chi formula}
   \chi(W) = \eta\bigl((-1)^{n(n-1)/2}\det W\bigr),
\end{equation}
where we recall that $\eta = \eta_{F/F_0}$ is the quadratic character on $F_0^\times$ attached to $F/F_0$, and we set $\det W := \det J \bmod \RN F^\times$ for any hermitian matrix $J$ representing the form on $W$.  We fix an anisotropic vector
\[
   u \in W,
\]
which we call the \emph{special vector}.  Let
\[
   W^\flat := u^\perp \subset W.
\]
Then there is an orthogonal decomposition $W = W^\flat \oplus Fu$. Setting $\epsilon := (u,u)$, it follows from \eqref{chi formula} that
\begin{equation}\label{chi decomp formula}
   \chi(W) = \chi(W^\flat)\eta\bigl((-1)^{n-1}\epsilon\bigr).
\end{equation}
We define the algebraic groups over $F_0$,
\begin{equation}\label{unitary group defs}
\begin{aligned}
	G &:= \U(W),\\
	H &:= \U(W^\flat),\\
	G_W &:= H \times G,\\
	H_W &:= H \times H.
\end{aligned}
\end{equation}
We embed $H$ in $G$ in the natural way as the stabilizer of the special vector $u$, and we embed $H$ in $G_W$ as the graph of this embedding.  We then consider the natural right action of $H_W$ on $G_W$,
\[
   g \cdot (h_1,h_2) = h_1^{-1}gh_2.
\]
We say that an element $g \in G_W(F_0)$ is \emph{regular semi-simple} if it is regular semi-simple for this action, and we denote the set of such elements by $G_W(F_0)_\rs$.

Now choose an $F$-basis for $W^\flat$, and complete it to a basis for $W$ by appending $u$ (and thus identifying $u$ with $e \in F^n$). This determines closed embeddings
\[
   H \inj H_1',
	\quad
	G \inj \Res_{F/F_0} \GL_n,
	\quad\text{and}\quad
	G_W \inj G',
\]
where the third embedding is the product of the first two.  We call the maps obtained in this way \emph{special embeddings}.

\begin{definition}
An element $\gamma \in G'(F_0)_\rs$ and an element $g \in G_W(F_0)_\rs$ \emph{match} if they are in the same $H_{1,2}'(F_0)$-orbit when $g$ is regarded as an element in $G'(F_0)$ via any special embedding.
\end{definition}

The notion of matching is independent of the choice of special embedding. Now let $W_0$ and $W_1$ respectively denote the split and non-split hermitian spaces of dimension $n$, and take the special vectors in each to have the same norm (which is always possible since $n \geq 2$).  Then the basic group-theoretic fact of concern to us is that the matching relation induces a bijection on regular semi-simple orbits (which follows from the analogous bijection \eqref{matching bij inhomog} below in the inhomogeneous setting),
\begin{equation}\label{matching bij homog}
   G'(F_0)_\rs/H_{1,2}'(F_0) \simeq \bigl(G_{W_0}(F_0)_\rs/H_{W_0}(F_0)\bigr) \amalg \bigl(G_{W_1}(F_0)_\rs/H_{W_1}(F_0)\bigr).
\end{equation}

The matching bijection \eqref{matching bij homog} gives rise to the notion of \emph{transfer of smooth functions} with respect to the following orbital integrals. For $\gamma\in G'(F_0)_\rs$, a function $f'\in C_c^\infty(G')$, and a complex parameter $s\in \BC$, we define the weighted orbital integral
\begin{equation*}
   \Orb(\gamma, f', s) := 
	   \int_{H_{1,2}'(F_0)} f'(h_1^{-1}\gamma h_2) \lv\det h_1\rv^s \eta(h_2)\, dh_1\, dh_2,
\end{equation*}
where $\lv\phantom{a}\rv$ denotes the normalized absolute value on $F$, where we set
\[
   \eta(h_2) := \eta(\det h_2')^n \eta(\det h_2'')^{n-1}
	\quad\text{for}\quad
	h_2 = (h_2', h_2'')\in H_2'({F_0}) = \GL_{n-1}(F_0) \times \GL_n(F_0),
\] 
and where we use fixed Haar measures on $H_1'(F_0)$ and $H_2'(F_0)$ and the product Haar measure on $H_{1,2}'(F_0) = H_1'(F_0) \times H_2'(F_0)$.
We further define the special values
\begin{equation*}
   \Orb(\gamma, f') := \Orb(\gamma, f', 0)
	\quad\text{and}\quad
	\del(\gamma, f') := \frac{d}{ds} \Big|_{s=0} \Orb({\gamma},  f',s) . 
\end{equation*}
The integral defining $\Orb(\gamma,f',s)$ is absolutely convergent, and $\Orb(\gamma,f')$ has the transformation property
\[
   \Orb(h_1^{-1}\gamma h_2,f') = \eta(h_2)\Orb(\gamma,f') 
	\quad\text{for}\quad 
	(h_1, h_2)\in H_{1, 2}'(F_0)=H_1'(F_0)\times H_2'(F_0) .
\]
For $W$ an $n$-dimensional hermitian space as above, an element $g \in G_W(F_0)_\rs$, and a function $f \in C_c^\infty(G_W)$, we similarly define the orbital integral
\begin{equation*}
   \Orb(g, f) := \int_{H_W(F_0)} f(h_1^{-1} g h_2)\, dh_1\, dh_2 .
\end{equation*}
Here the Haar measure on $H_W(F_0) = H(F_0) \times H(F_0)$ is a product of identical Haar measures on $H(F_0)$.

Finally, recall that a \emph{transfer factor} for $G'$ is a function $\omega\colon G'({F_0})_\rs\to \BC^\times$ such that
\[
   \omega(h_1^{-1}\gamma h_2)=\eta(h_2)\omega(\gamma)
	\quad\text{for all}\quad
	(h_1,h_2)\in H_1'(F_0) \times H_2'(F_0).
\] 
We will specify different transfer factors depending on the context later in the paper.  We can now state the definition of smooth transfer in the present context; we again denote by $W_0$ and $W_1$ the respective split and non-split hermitian spaces of dimension $n$.

\begin{definition}\label{transfer homog}
A function $f'\in C_c^\infty(G')$ and a pair of functions $(f_0,f_1) \in C_c^\infty(G_{W_0}) \times C_c^\infty(G_{W_1})$ are \emph{transfers} of each other (for the fixed choices of Haar measures, a fixed choice of transfer factor, and fixed choices of special vectors $u_i \in W_i$) if for each $i \in \{0,1\}$ and each $g\in G_{W_i}(F_0)_\rs$,
\[
 \Orb(g,f_i)=\omega(\gamma) \Orb(\gamma,f')
\]
whenever $\gamma\in G'(F_0)_\rs $ matches $g$.
\end{definition}

\subsection{Inhomogeneous group setting}\label{setup inhomog}
In this subsection we give an ``inhomogeneous'' analog of the previous subsection, whose notation we retain.  The role of $G'$ in the inhomogeneous setting is played by the symmetric space
\begin{equation}\label{S_n}
   S := S_n := \bigl\{\, \gamma \in \Res_{F/F_0}\GL_n \bigm| \gamma \ov\gamma = 1_n \,\bigr\}.
\end{equation}
Note that in some later parts of the paper (especially in Part \ref{spaces} in the context of RZ spaces) we will also use the symbol $S$ to denote an arbitrary test scheme; context should always make the meaning clear.
The role of $H_{1,2}'$ in the inhomogeneous setting is played by the group over $F_0$
\[
   H' := \GL_{n-1},
\]
which acts naturally on the right on $S$ by conjugation (via the map \eqref{GL_n-1 emb}).  The homogeneous and inhomogeneous settings are related via the maps
\begin{equation}\label{G' -> Res GL_n}
	\begin{gathered}
 \xymatrix@R=0ex{
	   G' \ar[r]  &  \Res_{F/F_0} \GL_n\\
		(\gamma_1,\gamma_2) \ar@{|->}[r]  &  \gamma_1^{-1}\gamma_2
	}
	\end{gathered}
\end{equation}
(defined again using \eqref{GL_n-1 emb}) and
\begin{equation}\label{Res GL_n -> S}
   \begin{gathered}
     r\colon \xymatrix@R=0ex{
	   \Res_{F/F_0} \GL_n \ar[r]  &  S\\
		\gamma \ar@{|->}[r]  &  \gamma \ov\gamma^{-1}
	}
	\end{gathered}.
\end{equation}
These maps induce respective isomorphisms of varieties
\[
   H_1' \bs G' \isoarrow \Res_{F/F_0} \GL_n
	\quad\text{and}\quad
	(\Res_{F/F_0} \GL_n) / \GL_n \isoarrow S,
\]
and a bijection on $F_0$-rational points
\[
   G'(F_0)/H_{1,2}'(F_0) \isoarrow S(F_0)/H'(F_0).
\]

On the unitary side, let $W$ be an $n$-dimensional $F/F_0$-hermitian space, and choose a special vector $u \in W$, as in the previous subsection.  Then the role of $G_W$ is played by $G$, and the role of $H_W$ is played by $H$, cf.~\eqref{unitary group defs}.  The natural map
\begin{equation}\label{G_W -> G}
	\begin{gathered}
   \xymatrix@R=0ex{
	   G_W \ar[r]  &  G\\
		(g_1,g_2) \ar@{|->}[r]  &  g_1^{-1}g_2
	}
	\end{gathered}
\end{equation}
induces an isomorphism of varieties
\[
   H \bs G_W \isoarrow G,
\]
where $H$ acts on $G_W$ via its diagonal embedding; and a bijection on $F_0$-rational points
\[
   G_W(F_0)/H_W(F_0) \isoarrow G(F_0)/H(F_0),
\]
where $H$ acts on $G$ by conjugation.

Now let us say that an element $\gamma \in \RM_n(F)$ is \emph{regular semi-simple} if it is regular semi-simple for the conjugation action of $\GL_{n-1,F}$ on $\RM_{n,F}$ with respect to the embedding \eqref{GL_n-1 emb}.  It is equivalent that $\gamma$ have Zariski-closed orbit and trivial stabilizer; or that the sets of vectors $\{\gamma^i e\}_{i=0}^{n-1}$ and $\{\tensor[^t]{\gamma}{^i} e \}_{i=0}^{n-1}$ are linearly independent over $F$ \cite[Th.~6.1]{RS}.  We say that an element $\gamma \in S(F_0)$ is \emph{regular semi-simple} if it is regular semi-simple for the conjugation action of $H'$ on $S$.  It is equivalent that $\gamma$ be regular semi-simple as an element in $\RM_n(F)$ in the sense just given.  The notions of regular semi-simplicity in the homogeneous and inhomogeneous settings are compatible in the sense that an element $\gamma \in G'(F_0)$ is regular semi-simple if and only if its image in $S(F_0)$ under the composite of the maps \eqref{G' -> Res GL_n} and \eqref{Res GL_n -> S} is.  Similarly, an element $g \in G(F_0)$ is \emph{regular semi-simple} if it is regular semi-simple for the conjugation action of $H$ on $G$; or equivalently if it is regular semi-simple as an element in $\RM_n(F)$ under one, hence any, special embedding $G \inj \Res_{F/F_0} \GL_n$.  An element $g \in G_W(F_0)$ is regular semi-simple if and only if its image in $G(F_0)$ under the map \eqref{G_W -> G} is.  We denote by $S(F_0)_\rs$ and $G(F_0)_\rs$ the sets of regular semi-simple elements in $S(F_0)$ and $G(F_0)$, respectively.  In the inhomogeneous setting the notion of matching takes the following form.

\begin{definition}
An element $\gamma \in S(F_0)_\rs$ and an element $g \in G(F_0)_\rs$ \emph{match} if they are in the same $\GL_{n-1}(F)$-orbit when $g$ is regarded as an element in $\GL_n(F)$ via any special embedding.
\end{definition}

The matching relation is again independent of the choice of special embedding, and it induces a bijection on regular semi-simple orbits \cite[\S2]{Z12},
\begin{equation}\label{matching bij inhomog}
	S(F_0)_\rs/H'(F_0) \simeq \bigl( G_0(F_0)_\rs/H_0(F_0)\bigr) \amalg \bigl( G_1(F_0)_\rs/H_1(F_0)\bigr),
\end{equation}
where we write
\[
   G_0 := G
	\quad\text{and}\quad
	H_0 := H
\]
when $W = W_0$ is split, and
\[
   G_1 := G
	\quad\text{and}\quad
	H_1 := H
\]
when $W = W_1$ is non-split.  (Note that the quasi-splitness of $H_0$ and $H_1$ is then governed by the formula \eqref{chi decomp formula}.) Here as before we take the special vectors in $W_0$ and $W_1$ to have the same norm.

The formalism of orbital integrals and smooth transfer carries over readily from the homogeneous setting to the inhomogeneous setting.  For $\gamma \in S(F_0)_\rs$, $f' \in C_c^\infty(S)$, and $s \in \BC$, we define
\begin{equation}\label{Orb(gamma,f',s)}
   \Orb(\gamma,f',s) := \int_{H'(F_0)} f'(h^{-1}\gamma h) \lv \det h \rv^s \eta(h)\, dh,
\end{equation}
where $\lv\phantom{a}\rv$ denotes the normalized absolute value on $F_0$, where we set
\[
   \eta(h) := \eta(\det h),
\]
and where we use a fixed Haar measure on $H'(F_0) = \GL_{n-1}(F_0)$.  We define the special values
\[
   \Orb(\gamma, f') := \Orb(\gamma, f', 0)
	\quad\text{and}\quad
	\del(\gamma, f') := \frac{d}{ds} \Big|_{s=0} \Orb({\gamma},  f',s) . 
\]
As in the homogeneous setting, the integral defining $\Orb(\gamma,f',s)$ is absolutely convergent, and when $s = 0$ it transforms as
\[
   \Orb(h^{-1}\gamma h, f') = \eta(h) \Orb(\gamma,f')
	\quad\text{for all}\quad
	h \in H'(F_0).
\]
On the unitary side, for $g \in G(F_0)_\rs$ and $f \in C_c^\infty(G)$, we define
\[
   \Orb(g,f) := \int_{H(F_0)} f(h^{-1}gh) \, dh,
\]
where we use the same fixed Haar measure on $H(F_0)$ as in the previous subsection.  Finally, a \emph{transfer factor} for $S$ is a function $\omega\colon S(F_0)_\rs \to \BC^\times$ such that
\[
   \omega(h^{-1}\gamma h) = \eta(h) \omega(\gamma)
	\quad\text{for all}\quad
	h \in H'(F_0).
\]
With this we arrive at the inhomogeneous version of smooth transfer.

\begin{definition}
A function $f' \in C_c^\infty(S)$ and a pair of functions $(f_0,f_1) \in C_c^\infty(G_0) \times C_c^\infty(G_1)$ are \emph{transfers} of each other (for the fixed choices of Haar measures, a fixed choice of transfer factor, and fixed choices of special vectors) if for each $i \in \{0,1\}$ and each $g \in G_i(F_0)_\rs$,
\[
   \Orb(g,f_i) = \omega(\gamma) \Orb(\gamma,f')
\]
whenever $\gamma \in S(F_0)_\rs$ matches $g$.
\end{definition}

\subsection{Lie algebra setting}\label{setup lie}
In this subsection we give a ``Lie algebra'' analog of the inhomogeneous group setup.  The role of $S$ is played by its tangent space at the identity matrix,
\[
   \fks := \fks_n := \bigl\{\, y \in \Res_{F/F_0} \RM_n \bigm| y + \ov y = 0 \,\bigr\}.
\]
The group action we consider is the natural right action of $H'$ on \fks by conjugation.  For $W$ an $n$-dimensional hermitian space, the role of $G$ is played by its Lie algebra
\[
   \fkg := \Lie G.
\]
Upon choosing a special vector in $W$, we then consider the right adjoint action of $H$ on \fkg.

We say that an element in $y \in \fks(F_0)$ is \emph{regular semi-simple} if it is regular semi-simple for the action of $H'$, and we denote the set of such elements by $\fks(F_0)_\rs$.  It is equivalent that $y$ have closed $H'$-orbit and trivial stabilizer; or that $y$ be regular semi-simple as an element in $\RM_n(F)$ in the sense of the previous subsection.  We say that an element $x \in \fkg(F_0)$ is \emph{regular semi-simple} if it is regular semi-simple for the action of $H$, and we denote the set of such elements by $\fkg(F_0)_\rs$.  As in \S\ref{setup homog}, the choice of a basis for $W^\flat$, extended by the special vector to a basis for $W$, determines a closed embedding
\[
   \fkg \inj \Res_{F/F_0}\RM_n,
\]
which we again call a \emph{special embedding}.  For $x \in \fkg(F_0)$ to be regular semi-simple, it is again equivalent that $x$ have closed $H$-orbit and trivial stabilizer; or that $x$ be regular semi-simple as an element in $\RM_n(F)$ under one, hence any, special embedding.

\begin{definition}
An element $y \in \fks(F_0)_\rs$ and an element $x \in \fkg(F_0)_\rs$ \emph{match} if they are in the same $\GL_{n-1}(F)$-orbit when $x$ is regarded as an element in $\RM_n(F)$ via any special embedding.
\end{definition}

As before, the matching relation is independent of the choice of special embedding, and it induces a bijection on regular semi-simple orbits \cite[\S5]{JR},
\[
   \fks(F_0)_\rs/H'(F_0) \simeq \bigl(\fkg_0(F_0)_\rs / H_0(F_0)\bigr) \amalg \bigl(\fkg_1(F_0)_\rs / H_1(F_0)\bigr),
\]
where as in the inhomogeneous group setting we use the subscripts $0$ and $1$ on $\fkg$ and $H$ according as $W$ is split or non-split, and we take the special vectors to have the same norm.

The formalism of orbital integrals and smooth transfer again carries over in a straightforward way to the present setting.  For $y \in \fks(F_0)_\rs$, $\phi' \in C_C^\infty(\fks)$, and $s \in \BC$, we define
\[
   \Orb(y,\phi',s) := \int_{H'(F_0)} \phi'(h^{-1}yh) \lv\det h \rv^s \eta(h) \, dh,
\]
as well as the special values
\[
   \Orb(y,\phi') := \Orb(y,\phi,0)
	\quad\text{and}\quad
	\del(y,\phi') := \frac{d}{ds} \Big|_{s=0} \Orb(y, \phi',s) . 
\]
The notation in the integral defining $\Orb(y,\phi',s)$ is as in \eqref{Orb(gamma,f',s)}. This integral is again absolutely convergent and transforms when $s = 0$ as
\[
   \Orb(h^{-1}yh,\phi') = \eta(h)\Orb(y,\phi')
	\quad\text{for all}\quad
	h \in H'(F_0).
\]
On the unitary side, for $x \in \fkg(F_0)_\rs$ and $\phi \in C_c^\infty(\fkg)$, we define
\[
   \Orb(x,\phi) := \int_{H(F_0)} \phi(h^{-1}xh) \, dh,
\]
where we use the same Haar measure on $H(F_0)$ as before.  A \emph{transfer factor} for \fks is a function $\omega\colon \fks(F_0)_\rs \to \BC^\times$ such that
\[
   \omega(h^{-1}\gamma h) = \eta(h)\omega(\gamma)
	\quad\text{for all}\quad
	h \in H'(F_0).
\]

\begin{definition}
A function $\phi' \in C_c^\infty(\fks)$ and a pair of functions $(\phi_0,\phi_1) \in C_c^\infty(\fkg_0) \times C_c^\infty(\fkg_1)$ are \emph{transfers} of each other (for the fixed choices of Haar measures, a fixed choice of transfer factor, and fixed choices of special vectors) if for each $i \in \{0,1\}$ and each $x\in \fkg_i(F_0)_\rs$,
\[
   \Orb(x,\phi_i) = \omega(y) \Orb(y,\phi')
\]
whenever $y\in \fks(F_0)_\rs $ matches $x$.
\end{definition}

\subsection{Transfer factors}\label{trans factor}
We now fix transfer factors for use throughout the rest of the paper, which are slight variants of the ones in \cite[\S2.4]{Z14}. First fix an extension $\wt\eta$ of the quadratic character $\eta$ from $F_0^\times$ to $F^\times$ (not necessarily of order $2$). If $F$ is unramified, then we take the natural extension $\wt\eta(x)=(-1)^{v(x)}$. For $S$ we take the transfer factor
\[
   \omega_S(\gamma) := \wt\eta\bigl(\det(\gamma)^{-\lfloor n/2\rfloor } \det(\gamma^ie)_{0 \leq i \leq n-1}\bigr), \quad \gamma \in S(F_0)_\rs.
\]
For $G'$ we take the transfer factor
\begin{equation*}
   \omega_{G'}(\gamma) := 
	\begin{cases}
		\omega_S\bigl(r(\gamma_1^{-1}\gamma_2)\bigr), & n \text{ odd}; \\
      \wt\eta(\gamma_1^{-1}\gamma_2) \omega_S\bigl(r(\gamma_1^{-1}\gamma_2)\bigr), & n \text{ even}, 
   \end{cases}
	\quad
	\gamma= (\gamma_1,\gamma_2)\in  G'(F_0)_\rs,
\end{equation*}
where $r$ is defined in \eqref{Res GL_n -> S}, and where for any $\gamma_0 \in \GL_n(F)$ we set
\[
   \wt\eta(\gamma_0) := \wt\eta(\det \gamma_0).
\]
For \fks we take the transfer factor
\begin{equation}\label{fks transfer factor ram}
   \omega_\fks(y) := \wt\eta\bigl(\det(y^ie)_{0 \leq i \leq n-1}\bigr), \quad y \in \fks(F_0)_\rs.
\end{equation}

\section{Review of the FL conjecture}\label{s:FL}
To set the stage, in this section we review the FL conjecture in its homogeneous, inhomogeneous, and Lie algebra versions, cf.~\cite{JR, Z12, RTZ}.  Let $F/F_0$ be unramified and $n \geq 2$.  As in the previous section, let $W_0$ and $W_1$ respectively be the split and non-split $F/F_0$-hermitian spaces of dimension $n$. Assume furthermore that the special vectors $u_i \in W_i$ have common norm which is a unit in $O_{F_0}$.  Then by \eqref{chi decomp formula} the orthogonal complement $W_i^\flat$ of $u_i$ in $W_i$ is again split for $i=0$ and non-split for $i=1$.  As in the previous section, we write $G_i = \U(W_i)$, $\fkg_i = \Lie G_i$, and $H_i = \U(W_i^\flat)$. Fix a self-dual $O_F$-lattice
\[
   \Lambda_0^\flat \subset W_0^\flat,
\]
which exists and is unique up to $H_0(F_0)$-conjugacy since $W_0^\flat$ is split.  Let
\[
   \Lambda_0 := \Lambda_0^\flat \oplus O_Fu_0 \subset W_0,
\]
which is again self-dual.  We denote by
\[
   K_0^\flat \subset H_0(F_0)
\]
the stabilizer of $\Lambda_0^\flat$, and by
\[
   K_0 \subset G_0(F_0)
	\quad\text{and}\quad
	\fkk_0 \subset \fkg_0(F_0)
\]
the respective stabilizers of $\Lambda_0$.  Then $K_0^\flat$ and $K_0$ are both hyperspecial maximal parahoric subgroups.

We normalize the Haar measures on the groups
\[
   H'(F_0) = \GL_{n-1}(F_0), \quad \GL_n(F_0),\quad H_1'(F_0) = \GL_{n-1}(F),\quad\text{and}\quad H_0(F_0)
\]
by assigning each of the respective subgroups
\[
   \GL_{n-1}(O_{F_0}),\quad \GL_n(O_{F_0}),\quad \GL_{n-1}(O_F), \quad\text{and}\quad K_0^\flat
\]
measure one.  We then take the product measure on the groups $H_2'(F_0) = \GL_{n-1}(F_0) \times \GL_n(F_0)$, $H_{1,2}'(F_0) = H_1'(F_0) \times H_2'(F_0)$, and $H_{W_0}(F_0) = H_0(F_0) \times H_0(F_0)$.  The Haar measures on $H_1(F_0)$ and $H_{W_1}(F_0)$ will not be important for us.  The transfer factors are defined in \S\ref{trans factor}. They take the following simple form on $S(F_0)_\rs$ and $\fkg(F_0)_\rs$, 
\begin{equation}\label{sign}
\begin{aligned}
   \omega_S(\gamma) &= (-1)^{v(\det(\gamma^i e)_{0 \leq i \leq n-1})}, \\
   \omega_\fks(y) &= (-1)^{v(\det(y^i e)_{0 \leq i \leq n-1})}.
\end{aligned}
\end{equation}
With respect to these normalizations, the FL conjecture is the following statement.

\begin{conjecture}[Fundamental lemma]\label{FLconj}\hfill
\begin{altenumerate}
\renewcommand{\theenumi}{\alph{enumi}}
\item
\textup{(Homogeneous version)} The characteristic function $\mathbf{1}_{G'(O_{F_0})} \in C_c^\infty(G')$ transfers to the pair of functions $(\mathbf{1}_{K_0^\flat\times {K}_0}, 0) \in C_c^\infty(G_{W_0}) \times C_c^\infty(G_{W_1})$. 
\item
\textup{(Inhomogeneous version)} The characteristic function $\mathbf{1}_{S(O_{F_0})} \in C_c^\infty(S)$ transfers to the pair of functions $(\mathbf{1}_{K_0},0) \in C_c^\infty(G_0)\times C_c^\infty(G_1)$.
\item\label{FLconj lie}
\textup{(Lie algebra version)} The characteristic function $\mathbf{1}_{\fks(O_{F_0})} \in  C_c^\infty(\fks)$ transfers to the pair of functions $(\mathbf{1}_{\fkk_0},0) \in C_c^\infty(\fkg_0)\times C_c^\infty(\fkg_1)$.
\end{altenumerate}
\end{conjecture}

We note that the equal characteristic analog of the FL conjecture was proved by Z.~Yun for $p> n$; J.~Gordon deduced the $p$-adic case for $p$ large, but unspecified, cf.\ \cite{Go,Y}.

\section{Review of the AFL conjecture}\label{s:AFL}
We continue with the notation and normalizations introduced in the last section. In particular, $F/F_0$ is unramified, and the special vectors $u_i$ have norm a unit in $O_{F_0}$.  Note that the FL conjecture predicts that the orbital integrals $\Orb(\gamma, \mathbf{1}_{G'(O_{F_0})})$, $\Orb(\gamma, \mathbf{1}_{S(O_{F_0})})$, and $\Orb(y, \mathbf{1}_{\fks(O_{F_0})})$ vanish whenever $\gamma$, resp.~$\gamma$, resp.~$y$ matches with elements in $G_{W_1}(F_0)_\rs$, resp.~$G_1(F_0)_\rs$, resp.\ $\fkg_1(F_0)_\rs$.  The AFL conjecture then proposes an identity for the derivatives of these orbital integrals at such elements, in terms of geometry.

\begin{conjecture}[Arithmetic fundamental lemma]\label{AFLconj}\hfill
\begin{altenumerate}
\renewcommand{\theenumi}{\alph{enumi}}
\item
\textup{(Homogeneous version)}\label{AFLconj homog} 
Suppose that $\gamma\in G'({F_0})_\rs$ matches an element $g\in G_{W_1}({F_0})_\rs$. Then 
\[
   \omega_{G'}(\gamma)\del\bigl(\gamma, \mathbf{1}_{G'(O_{F_0})}\bigr) 
	   = - 2\Int(g)\cdot\log q. 
\]
\item
\textup{(Inhomogeneous version)}\label{AFLconj inhomog} 
Suppose that $\gamma\in S({F_0})_\rs$ matches an element $g\in G_1({F_0})_\rs$. Then 
\[
   \omega_S(\gamma)\del\bigl(\gamma, \mathbf{1}_{S(O_{F_0})}\bigr) 
	   = -\Int(g)\cdot\log q. 
\]
\item
\textup{(Lie algebra version)}\label{AFLconj lie}
Suppose that $y\in \fks({F_0})_\rs$ matches an element $x\in \fkg_1({F_0})_\rs$, and assume that the intersection $\Delta \cap \Delta_x$ is an artinian scheme. Then 
\[
   \omega_\fks(y)\del\bigl(y, \mathbf{1}_{\fks(O_{F_0})}\bigr) 
	   = -\lInt(x)\cdot\log q. 
\]
\end{altenumerate}
\end{conjecture}

Let us explain the right-hand side in each part (as well as the expression $\Delta \cap \Delta_x$ in \eqref{AFLconj lie}).  For any $n \geq 1$, let $\CN_n = \CN_{n, F/F_0}$ denote the formal moduli scheme over $\Spf O_{\breve {F}}$ of \cite{KR-U1, VW}.\footnote{Recall that, since $F/F_0$ is unramified, we have $\breve F=\breve F_0$ as in the Introduction.}   In other words, we consider triples $(X, \iota, \lambda)$ over $\Spf O_{\breve F}$-schemes $S$, where $X$ is a $p$-divisible group of absolute height $2nd$ and dimension $n$ over $S$, equipped with an action $\iota$ of $O_{F}$ such that the induced action of $O_{F_0}$ on $\Lie X$ is via the structure morphism $O_{F_0}\to \CO_S$, and with a principal ($O_{F_0}$-relative) polarization $\lambda$. Here $d:=[{F_0}: \BQ_p]$.  Hence $(X, \iota|_{O_{F_0}})$ is a formal $O_{F_0}$-module of relative height $2n$ and dimension $n$. We require that the Rosati involution $\Ros_\lambda$ induces the non-trivial Galois automorphism in $\Gal({F}/{F_0})$ on $O_{F}$, and that the \emph{Kottwitz condition} of signature $(1, n-1)$ is satisfied, i.e.
\begin{equation}\label{kottwitzcond}
   \charac \bigl(\iota(a)\mid \Lie X\bigr)=(T-a)(T-\ov a)^{n-1} \in \CO_S[T]
	\quad\text{for all}\quad
	a\in O_{F} . 
\end{equation} 
An isomorphism $(X, \iota, \lambda) \isoarrow (X', \iota', \lambda')$ between two such triples is an $O_{F}$-linear isomorphism $\varphi\colon X\isoarrow X'$ such that $\varphi^*(\lambda')=\lambda$.

It is not hard to see that over the residue field $\ov k$ of $O_{\breve {F}}$ there is a unique such triple $(\BX_n, \iota_{\BX_n}, \lambda_{\BX_n})$ such that $\BX_n$ is  supersingular, up to $O_F$-linear quasi-isogeny compatible with the polarization, cf.~\cite[Prop.~1.15]{V2}. Then $\CN_n$ represents the functor over $\Spf O_{\breve F}$ that associates to each $S$ the set of isomorphism classes of quadruples $(X, \iota, \lambda, \rho)$ over $S$, where the final entry is an $O_F$-linear quasi-isogeny of height zero defined over the special fiber,
\[
   \rho \colon X\times_S\ov S \to \BX_n \times_{\Spec \ov k} \ov S,
\]
such that $\rho^*((\lambda_{\BX_n})_{\ov S}) = \lambda_{\ov S}$ (a \emph{framing}).

\begin{remark}
The definition of $\CN_n$ given here differs slightly from that in \cite{KR-U1, VW}.  First of all, $\CN_n$ is in fact an open and closed formal subscheme of the space defined in these papers (it is the locus, denoted by $\CN_0$ in \cite{VW}, where the framing $\rho$ has height zero).  Furthermore, within this locus, the moduli problem in these papers imposes the weaker condition that $\lambda_{\ov S}$ differs from $\rho^*((\lambda_{\BX_n})_{\ov S})$ locally on $\ov S$ by a unit in $O_{F_0}$, but it also weakens the notion of an isomorphism. This changes nothing in the outcome: the  formal scheme in loc.~cit.~is identical to the formal scheme defined here, cf.~also \cite[Rem.~3.6]{RSZ}. 
\end{remark}

The following theorem summarizes basic facts on the structure of the formal scheme $\CN_n$. 

\begin{theorem}[Vollaard--Wedhorn \cite{VW}]\hfill\label{structN}
\begin{altenumerate} 
\item\label{structN i} 
For any $n$, the formal scheme $\CN_n$ is formally locally of finite type, separated, essentially proper,\footnote{Recall that \emph{essentially proper} means that each irreducible component of the reduced underlying scheme is proper over $\Spec \ov k$.} and formally smooth of relative formal dimension $n-1$ over $\Spf O_{\breve {F}}$.  In particular, $\CN_n$ is regular of formal dimension $n$.
\item\label{structN ii}  
The underlying scheme $(\CN_n)_\red$ has a Bruhat--Tits stratification by Deligne--Lusztig varieties of dimensions $0, 1,\dotsc,\lfloor \frac{n-1}{2}\rfloor$ attached to unitary groups in an  odd number of variables and to Coxeter elements, with strata parametrized by the vertices of the Bruhat--Tits complex of the special unitary group for the non-split $n$-dimensional $F/F_0$-hermitian space.\qed 
\end{altenumerate}
\end{theorem}

\begin{corollary}  For $n \geq 2$, the product $\CN_{n-1, n}:=\CN_{n-1}\times_{\Spf O_{\breve {F}}}\CN_n $ is a regular formal scheme of formal dimension $2(n-1)$.\qed
\end{corollary}

Returning to our explanation of the statement of Conjecture \ref{AFLconj}, now recall from the Introduction the formal $O_F$-module $\BE$ over $\ov k$ and its canonical lift $\CE$ over $O_{\breve F}$, as well as the ``conjugate'' objects $\ov\BE$ and $\ov\CE$.  For $n \geq 2$, there is a natural closed embedding of formal schemes
\begin{equation*}
   \delta_\CN\colon 
	\xymatrix@R=0ex{
	   \CN_{n-1} \ar[r]  &  \CN_n\\
		(X, \iota, \lambda, \rho) \ar@{|->}[r]  &  \bigl(X\times\ov {\CE}, \iota\times \iota_{\ov\CE}, \lambda\times \lambda_{\ov\CE}, \rho\times\rho_{\ov\CE}\bigr),
		}
\end{equation*}
where we set $\BX_1 = \BE$ and inductively take
\begin{equation}\label{BX_n unram}
   \BX_n = \BX_{n-1} \times \ov\BE
\end{equation}
as the framing object for $\CN_n$.  Let 
\begin{equation*}
   \Delta_\CN \colon \CN_{n-1} \xra{(\id_{\CN_{n-1}},\delta_\CN)} \CN_{n-1}\times_{\Spf O_{\breve F}}\CN_n = \CN_{n-1,n}
\end{equation*}
be the graph morphism of $\delta_\CN$.  Then
\[
   \Delta := \Delta_\CN(\CN_{n-1})
\]
is a closed formal subscheme of half the formal dimension of $\CN_{n-1, n}$. Note that
\begin{equation}\label{Aut cong U}
   \Aut^\circ (\BX_n,\iota_{\BX_n},\lambda_{\BX_n}) \cong \U\bigl(\BV(\BX_n)\bigr)(F_0), 
\end{equation}
where the left-hand side is the group of self-framings of $\BX_n$, and where $\BV(\BX_n)$ is the hermitian space \eqref{defV} attached to $\BX_n$.  The left-hand side acts naturally on $\CN_n$ by acting on the framing: $g\cdot (X,\iota,\lambda,\rho) = (X,\iota,\lambda, g \circ \rho)$.  Furthermore $\BV(\BX_n)$ contains a natural special vector $u$ given by the inclusion of $\ov\BE$ in $\BX_n = \BX_{n-1} \times \ov\BE$ via the second factor. The norm of $u$ is $1$. It is easy to compute directly that $\BV(\BX_1) = \BV(\BE)$ is non-split, and then by induction and \eqref{chi decomp formula}, $\BV(\BX_n)$ is a non-split hermitian space of dimension $n$ for any $n$. Applying this to $n$ and $n-1$,  we can choose identifications  $W_1=\BV(\BX_n)$ and $W_1^\flat=\BV(\BX_{n-1})$ compatible with the natural inclusions on both sides. Hence we obtain an action of $H_1(F_0)$ on $\CN_{n-1}$, of $G_1(F_0)$ on $\CN_n$, and of $G_{W_1}(F_0)$ on $\CN_{n-1, n}$; and furthermore the maps $\delta_\CN$ and $\Delta_\CN$ are equivariant with respect to the respective embeddings $H_1(F_0) \inj G_1(F_0)$ and $H_1({F_0})\hookrightarrow G_{W_1}({F_0})$ defined in \S\ref{setup homog}. 

Now we are ready to define the right-hand side of the identities appearing in Conjecture \ref{AFLconj}. For $g \in G_{W_1}(F_0)$, we denote by $\Int(g)$ the intersection product on $\CN_{n-1, n}$ of $\Delta$ with its translate $g\Delta$, defined through the derived tensor product of the structure sheaves,
\begin{equation}\label{defintprod}
   \Int(g) := \la \Delta, g\Delta\ra_{\CN_{n-1, n}} := \chi({\CN_{n-1, n}},  \CO_\Delta\otimes^\BL\CO_{g\Delta}) . 
\end{equation}
We similarly define $\Int(g)$ for $g \in G_1(F_0)$,
\[
   \Int(g) := \bigl\la \Delta, (1 \times g)\Delta \bigr\ra_{\CN_{n-1,n}}.
\]
In both cases, when $g$ is regular semi-simple, the right-hand side of this definition is finite, at least when $F_0=\BQ_p$, cf.\ \cite[Lem.~2.8]{Z12}. The proof in loc.~cit.~uses global methods.

\begin{remark}\label{reltospecialunr}
In \cite{KR-U1}, there is associated to $u\in \BV(\BX_n)$ a special cycle $\CZ(u)$ in $\CN_n$, namely, the locus where the quasi-homomorphism $u\colon \ov\BE\to \BX_n$ lifts to a homomorphism from $\ov\CE$ to the universal object over $\CN_n$. By \cite[Prop.\ 3.5]{KR-U1}, $\CZ(u)$ is a relative divisor. Then $\delta_\CN$ induces  an obvious closed embedding $\CN_{n-1}\to \CZ(u)$. By \cite[Lem.\ 5.2]{KR-U1}, this is an isomorphism. Similarly, for $g\in \U(\BV(\BX_n))$, there is an identification $g \delta_\CN(\CN_{n-1})=\CZ(gu)$.
\end{remark}

We make an analogous definition in the Lie algebra case, as in \cite[\S4.4]{PRS}.  Note that, analogously to \eqref{Aut cong U}, we can identify the $F_0$-points of $\fkg_1 = \Lie\U(\BV(\BX_n))$ with a subgroup of the $O_F$-linear quasi-endomorphisms of $\BX_n$ (those $x$ for which $x + \Ros_{\lambda_{\BX_n}}(x) = 0$).  Now, for any quasi-endomorphism $x$ of $\BX_n$, define $\Delta_x \subset \CN_{n-1,n}$ to be the closed formal subscheme (abusing notation in the obvious way)
\begin{equation}\label{Delta_x}
   \Delta_x := \bigl\{\, (Y,X) \in \CN_{n-1,n} \bigm| x\colon \BX_n \rightarrow \BX_n \text{ lifts to a homomorphism } Y \times \ov\CE \rightarrow X \,\bigr\}.
\end{equation}
For $g \in G_1(F_0)$, $\Delta_g$ is simply the translate $(1 \times g) \Delta$.  For any $x$, in the case that the intersection $\Delta \cap \Delta_x$ is an artinian scheme, we define
\begin{equation}\label{lInt}
   \lInt(x) := \length(\Delta\cap\Delta_x).
\end{equation}
This completes the explanation of the statement of the AFL in all cases.  
\begin{remark}
As already pointed out in \cite[\S4.4]{RSZ}, we do not have a Lie algebra version of the AFL outside the \emph{non-degenerate case}, i.e.\ when $\Delta\cap\Delta_x$ is not an artinian scheme. 
\end{remark}

We note that the left-hand sides of the identities in Conjecture \ref{AFLconj} are constant on the orbit of $\gamma$ under $H_{1,2}'(F_0)$, resp.~$\gamma$ under $H'(F_0)$, resp.~$y$ under $H'(F_0)$.  Similarly, the right-hand sides of these identities are constant on the orbit of $g$ under $H_{W_1}(F_0)$, resp.~$g$ under $H_1(F_0)$, resp.~$x$ under $H_1(F_0)$, cf.~\cite[Rems.~4.3, 4.4, 4.6]{PRS}. Hence Conjecture \ref{AFLconj} ``makes sense'' in the sense that it does not depend on the choice of matching elements.
 
Our aim in Parts \ref{spaces} and \ref{conjectures} of the paper is to formulate variants of Conjecture \ref{AFLconj} which apply to cases where the various unramifiedness hypotheses in Conjecture \ref{AFLconj} are dropped. As explained in the Introduction, we still want to have identities like in the AFL, but the function on $G'(F_0)$ transferring to $(\mathbf{1}_{K_0^\flat\times {K}_0}, 0)$ on $G_{W_0}(F_0)\times G_{W_1}(F_0)$ (and analogously in the inhomogeneous and Lie algebra settings) will not be explicit anymore. Rather, the statement of the ATC will be that some choice of function may be found which transfers to $(\mathbf{1}_{K_0^\flat\times {K}_0}, 0)$ on $G_{W_0}(F_0)\times G_{W_1}(F_0)$, and which also satisfies a suitable analog of the AFL identity (and again analogously in the inhomogeneous and Lie algebra settings).   
\begin{remark}\label{pi-modular unram}
A first natural variant of the moduli problem defining $\CN_n$ (still with $n \geq 1$ and $F/F_0$ unramified) is to consider quadruples $(X,\iota,\lambda,\rho)$ exactly as above, except where the polarization $\lambda$ satisfies $\ker \lambda = X[\iota(\varpi)]$.  The resulting moduli space yields nothing new, however: there is an isomorphism from $\CN_n$ to it defined by $(X,\iota,\lambda,\rho) \mapsto (X,\iota,\varpi\lambda, \rho)$.
\end{remark}

\part{Some regular formal moduli spaces}\label{spaces}

In this part of the paper we define variants of the formal moduli space $\CN_n$ that appeared in the AFL conjecture.  In the various AT conjectures which we formulate in Part \ref{conjectures}, the right-hand side will involve an intersection number of cycles in a product of such moduli spaces.  Therefore we are looking for analogs of $\Delta_\CN$ where the target is a \emph{regular} formal scheme.  We do not know of a systematic method of finding such variants.  In the following sections we present the examples we have found.

\section{Unramified almost self-dual type}\label{unram non-max}

In this section we continue with the notation of \S\ref{s:AFL}. In particular, $F/F_0$ is an unramified extension. For any $n \geq 2$, we now define a variant $\wt\CN_n$ of the formal scheme $\CN_n$ over $\Spf O_{\breve F} = \Spf O_{\breve F_0}$.  This variant parametrizes isomorphism classes of quadruples $(X, \iota, \lambda, \rho)$ as in the case of $\CN_n$, except that instead of requiring the polarization $\lambda$ to be principal, we impose that
\begin{equation}\label{polarization cond almost varpi-modular}
   \ker \lambda \subset X[\iota(\varpi)]\ \text{\emph{is of rank $q^2$.}}
\end{equation}
Mimicking \cite[\S1]{V2}, one sees that as in the case of $\CN_n$, up to quasi-isogeny there is a unique supersingular triple $(\smash{\wt\BX_n}, \iota_{\wt\BX_n}, \lambda_{\wt\BX_n})$ for this moduli problem over $\ov k$.
To fix a particular choice, first let
\begin{equation}\label{ovBE'}
   \ov\BE' := \ov\BE,
	\quad
	\iota_{\ov\BE'} := \iota_{\ov\BE},
	\quad\text{and}\quad
	\lambda_{\ov\BE'} := \varpi\lambda_{\ov\BE}.
\end{equation}
Then we set
\begin{equation}\label{wtBX_n unram}
   \wt\BX_n := \BX_{n-1} \times \ov\BE',
	\quad
	\iota_{\wt\BX_n} := \iota_{\BX_{n-1}} \times \iota_{\ov\BE'},
	\quad\text{and}\quad
	\lambda_{\wt\BX_n} := \lambda_{\BX_{n-1}} \times \lambda_{\ov\BE'},
\end{equation}
where $\BX_{n-1}$ is the framing object for $\CN_{n-1}$ in \eqref{BX_n unram}.  We then take the target of the framing $\rho$ to be the constant object over the special fiber obtained from $(\smash{\wt\BX_n}, \iota_{\wt\BX_n}, \lambda_{\wt\BX_n})$.

To describe the basic structure of $\wt\CN_n$, we will give an analog of Theorem \ref{structN}\eqref{structN i}.  For both the statement and proof we will need the \emph{local model} for $\wt\CN_n$, cf.\ \cite{PRS,RZ}, whose definition we first recall; this is similar to (an instance of) Definition \ref{LM def} below. Endow $F^n$ with the (non-split) $F/F_0$-hermitian form $h$ given by the matrix $\diag(\varpi,1,\dotsc,1)$.  Fix an element $\delta \in O_F^\times$ satisfying $\ov \delta = -\delta$, and define the alternating $F_0$-bilinear form on $F^n$,
\[
   \la x,y \ra = \frac 1 2 \tr_{F/F_0}\bigl(\delta h(x,y)\bigr), \quad x,y \in F^n.
\]
Define the $O_F$-lattices in $F^n$,
\[
   \Lambda_0 := O_F^n
	\quad\text{and}\quad
	\Lambda_1 := \varpi\i O_F \oplus O_F^{n-1}.
\]
Then $\Lambda_1$ is the dual lattice of $\Lambda_0$ with respect to both $h$ and \aform.  The \emph{local model} $\wt N_n$ is the scheme over $\Spec O_F$ representing the functor that associates to each $O_F$-scheme $S$ the set of all pairs $(\CF_0,\CF_1)$ such that
\begin{altitemize}
\item for each $i = 0,1$, $\CF_i$ is an $O_F \otimes_{O_{F_0}} \CO_S$-subsheaf of $\Lambda_i \otimes_{O_{F_0}} \CO_S$ which Zariski-locally on $S$ is an $\CO_S$-direct summand of rank $n$; 
\item the natural maps $\Lambda_0 \otimes_{O_{F_0}} \CO_S \to \Lambda_1 \otimes_{O_{F_0}} \CO_S$ and $\smash[t]{\Lambda_1 \otimes_{O_{F_0}} \CO_S \xra{\varpi \otimes \id} \Lambda_0 \otimes_{O_{F_0}} \CO_S}$ carry $\CF_0$ into $\CF_1$ and $\CF_1$ into $\CF_0$, respectively; 
\item $\CF_0^\perp = \CF_1$ with respect to the natural perfect pairing $(\Lambda_0 \otimes_{O_{F_0}} \CO_S) \times (\Lambda_1 \otimes_{O_{F_0}} \CO_S) \to \CO_S$ induced by \aform; and 
\item the Kottwitz condition of signature $(n-1,1)$
\[
   \charac(a \otimes 1 \mid \CF_i) = (T - a)^{n-1} (T - \ov a) \in \CO_S[T]
	\quad\text{for all}\quad
	a \in O_F,\quad i = 0,1
\]
is satisfied.
\end{altitemize}
Note that for any point $(X,\iota,\lambda,\rho)$ on $\wt\CN_n$, there is a unique isogeny $\lambda'$ such that the composite $X \xra\lambda X^\vee \xra{\lambda'} X$ is $\iota(\varpi)$. It follows easily from this that the base change $(\wt N_n)_{\Spf O_{\breve F}}$ identifies with the local model for $\wt\CN_n$ as in \cite[Ch.~3]{RZ}. The analog of Theorem \ref{structN}\eqref{structN i} for $\wt\CN_n$ is now as follows.

\begin{theorem}\label{wtCN_n semistable}
The formal scheme $\wt{\CN}_n$ is formally locally of finite type, separated, essentially proper, and of semi-stable reduction over $\Spf O_{\breve {F}}$, of relative formal dimension $n-1$. More precisely, the local model $\wt N_n$ has semi-stable reduction such that the special fiber is the union of two smooth schemes of dimension $n-1$ intersecting along a smooth scheme of dimension $n-2$. In particular, $\wt\CN_n$ is regular of formal dimension $n$, and the completed local ring at any closed point of $\wt\CN_n$ is isomorphic to either $O_{\breve F}[[X_1,\ldots, X_{n-1}]]$ or to $O_{\breve F}[[X_0, X_1,\ldots, X_{n-1}]]/(X_0 X_1-\varpi)$. 
\end{theorem}

\begin{proof}
The first three properties are general properties of RZ spaces (see \cite[Lem.\ 2.3.23]{F} for separatedness).  For the rest of the theorem, by the general formalism of RZ spaces (see \cite[Prop.~3.33]{RZ}), it suffices to prove the claim for the local model.  Since $F/F_0$ is unramified and $S$ is an $O_F$-scheme, we have the standard isomorphism
\begin{equation*}
	\xymatrix@R=0ex{
      O_F \otimes_{O_{F_0}} \CO_S \ar[r]^-\sim  &  \CO_S \times \CO_S\\
		a \otimes b \ar@{|->}[r]  &  (ab, \ov a b).
	}
\end{equation*}
Correspondingly, for any $S$-point $(\CF_0,\CF_1)$ on $\wt N_n$, there are decompositions
\[
   \CF_i = \CF_i' \oplus \CF_i'' \subset \Lambda_i \otimes_{O_{F_0}} \CO_S = (\Lambda_i \otimes_{O_{F_0}} \CO_S)' \oplus (\Lambda_i \otimes_{O_{F_0}} \CO_S)'',
	\quad
	i = 0,1,
\]
where $\CO_S \times \CO_S$ acts via its first factor on the primed sheaves, and via its second factor on the double-primed sheaves.
By the Kottwitz condition, $\CF_i' \subset (\Lambda_i \otimes_{O_{F_0}} \CO_S)'$ is an $\CO_S$-locally direct summand of rank $n-1$. By the perpendicularity condition, $\CF_0'$ and $\CF_1'$ determine $\CF_1''$ and $\CF_0''$, respectively.  It follows that the map $(\CF_i)_{i=0,1} \mapsto (\CF_i' \subset (\Lambda_i \otimes_{O_{F_0}} \CO_S)')_{i=0,1}$ is an isomorphism from $\wt N_n$ to the standard local model over $\Spec O_F$ in \cite{Goertz} for the group $\GL_n$, the cocharacter $\mu = (1^{(n-1)},0)$, and the periodic lattice chain determined by the (adjacent) lattices $(\Lambda_0 \otimes_{O_{F_0}} O_F)' \subset (\Lambda_1 \otimes_{O_{F_0}} O_F)'$.  By \S4.4.5 in loc.~cit.~(in the case $\kappa = 1$ and $r = n-1$) this latter scheme has semi-stable reduction of the asserted form.
\end{proof}

Now define 
\begin{equation*}
   \wt{\CN}_{n-1, n} := \CN_{n-1}\times_{\Spf O_{\breve {F}}}\wt{\CN}_n.
\end{equation*}
Since $\wt{\CN}_n$ only occurs here as the ``bigger'' formal scheme, we won't need to know an analog of Theorem \ref{structN}\eqref{structN ii} for the structure of $(\wt{\CN}_n)_\red$. 

\begin{corollary}
The formal scheme $\wt{\CN}_{n-1, n}$ is regular of formal dimension $2(n-1)$.  \qed
\end{corollary}

\begin{remark}[$n=2$]
For $n=2$, the formal scheme $\wt{\CN}_n$ is isomorphic to 
\begin{equation}\label{excDrin}
   \wt\CN_2 \simeq \wh\Omega^2_{F_0}\times_{\Spf O_{F_0}}\Spf O_{\breve {F}} ,
\end{equation} cf.~\cite{KR-alt}.  Here $\wh\Omega^2_{F_0}$ denotes the formal scheme version of the Drinfeld halfspace corresponding to the local field $F_0$.
\end{remark}

\begin{remark}[$n=1$]
For $n = 1$, the definition of $\wt\CN_1$ still makes sense, but the result is trivial: this space is just $\Spf O_{\breve F}$ itself, with universal object $(\CE,\iota_\CE,\varpi \lambda_\CE,\rho_\CE)$.
\end{remark}

\begin{remark}\label{almost pi-modular unram}
Let us consider the variant of the moduli problem for $\wt\CN_n$ where the condition on the polarization $\lambda$ is replaced by
\[
   \ker\lambda \subset X[\iota(\varpi)] \text{ \emph{is of rank} } q^{2(n-1)}.
\]
Then, in analogy with Remark \ref{pi-modular unram}, the resulting moduli space is isomorphic to $\wt\CN_n$.  Indeed, given a point $(X,\iota,\lambda,\rho)$ on $\CN_n$, let $O_F$ act on $X^\vee$ via the rule $\ov\iota^\vee\colon a \mapsto \iota(\ov a)^\vee$.  Then $\lambda$ is $O_F$-linear.  Since $\ker\lambda \subset X[\iota(\varpi)]$, there exists a unique (automatically $O_F$-linear) isogeny $\lambda'\colon X^\vee \to X$ such that $\lambda'\circ\lambda = \iota(\varpi)$.  Since $F/F_0$ is unramified and $\lambda$ is a polarization, so is $\lambda'$.  Furthermore, since $\ker\lambda \subset X[\iota(\varpi)]$ of rank $q^2$, we have $\ker\lambda' \subset X^\vee[\ov\iota^\vee(\varpi)]$ of rank $q^{2(n-1)}$.  In this way we obtain a morphism $(X,\iota,\lambda,\rho) \mapsto (X^\vee,\ov\iota^\vee,\lambda',(\rho^\vee)\i)$ from $\wt\CN_n$ to the variant moduli space we have just defined, and it is easy to see that this is an isomorphism.
\end{remark}

\section{Ramified, even, $\pi$-modular type}\label{even ram max}

In this and the next three sections we define analogs of $\CN_n$ (and of $\wt\CN_n$) when the quadratic extension $F/F_0$ is ramified, beginning in this section with the case that $n \geq 2$ is even.  Recall that $F/F_0$ is implicit in the definition of $\CN_n$ in \S\ref{s:AFL}. Therefore we may recycle the notation $\CN_n=\CN_{n, {F}/{F_0}}$ in the present situation. We define $\CN_n$ to be the formal scheme over $\Spf O_{\breve {F}}$ that parametrizes isomorphism classes of quadruples $(X, \iota, \lambda, \rho)$ as in \S\ref{s:AFL}, except that instead of requiring the polarization $\lambda$ to be principal, we impose that it is \emph{$\pi$-modular}, i.e.
\begin{equation*}
   \ker \lambda =X[\iota(\pi)] . 
\end{equation*}
We furthermore require that $\iota$ satisfies the \emph{wedge condition}
\begin{equation}\label{wedge condition}
	\sidewedge{^2} \bigl(\iota(\pi)+\pi \mid \Lie X\bigr) = 0
\end{equation}
and the \emph{spin condition}
\begin{equation}\label{spin condition}
   \text{\emph{the endomorphism $\iota(\pi)\mid \Lie X$ is nowhere zero,}}
\end{equation}
i.e.~$\iota(\pi)$ is a nonzero operator on $\Lie X\otimes\kappa(s)$ for all points $s$ of the base scheme $S$.
Since $S$ is a $\Spf O_{\breve {F}}$-scheme, $\pi\cdot \kappa(s) = 0$ for all $s$. Therefore, in the presence of the wedge condition, the spin condition says that  $\iota(\pi) \mid \Lie X\otimes\kappa(s)$ has rank $1$ for all $s$. Note that this is a purely pointwise condition on $S$, i.e.~it holds if and only if it holds after base change to $S_\red$.  As before, up to quasi-isogeny there is a unique supersingular framing object $(\BX_n, \iota_{\BX_n}, \lambda_{\BX_n})$ over $\ov k$, which follows in the case of this moduli problem from Prop.~3.1 and its proof in \cite{RSZ}.   The last entry $\rho$ in the quadruple above is a framing to the constant object over $\ov S$ defined by $(\BX_n, \iota_{\BX_n}, \lambda_{\BX_n})$, as before.

\begin{remark}
It is not necessary to explicitly make the Kottwitz condition \eqref{kottwitzcond} part of the definition of $\CN_n$.  Indeed, let $\varphi$ denote the operator $\iota(\pi) + \pi$ acting on $\Lie X$.  Then conditions \eqref{wedge condition} and \eqref{spin condition} imply that $\im \varphi$ is a locally direct summand of $\Lie X$ of $\CO_S$-rank $1$.  Hence $\ker \varphi$ is a locally direct summand of rank $n-1$.  Since $\iota(\pi)$ visibly acts as multiplication by $\pi$ on $\im \varphi$ and by $-\pi$ on $\ker\varphi$, it follows that $X$ satisfies the Kottwitz condition.
\end{remark}

For use later in the paper we fix essentially the same explicit choice of framing object as in \cite[\S3.3]{RSZ}.  When $n = 2$, we take
\[
   \BX_2 := \BE \times \BE
\]
as a formal $O_{F_0}$-module, and we define $\iota_{\BX_2}$ by
\[
   \iota_{\BX_2}(a + b\pi) :=
	\begin{bmatrix}
		a  &  b\varpi\\
		b  &  a
	\end{bmatrix},
	\quad
	a,b \in O_{F_0}.
\]
(This identifies $\BX_2$ with the Serre tensor construction $O_F \otimes_{O_{F_0}} \BE$.)  For the polarization, for technical convenience later in the paper we take a rescaled version of the one in loc.\ cit.,
\begin{equation}\label{lambda_BX_2}
   \lambda_{\BX_2} :=
	\begin{bmatrix}
		-2\lambda_\BE\\
		  &  2\varpi \lambda_\BE
	\end{bmatrix}.
\end{equation}
We then take
\begin{align*}
   \BX_n &:= \BX_2 \times \ov\BE^{n-2},\\
   \iota_{\BX_n} &:= \iota_{\BX_2} \times \iota_{\ov\BE}^{n-2},\\
   \lambda_{\BX_n} &:= \lambda_{\BX_2} \times 
      \diag \biggl(
	     \underbrace{\begin{bmatrix} 
	           0  &  \lambda_{\ov\BE}\, \iota_{\ov\BE}(\pi)\\
	           -\lambda_{\ov\BE}\, \iota_{\ov\BE}(\pi)  &  0
			\end{bmatrix}
			,\dotsc,
			\begin{bmatrix} 
	           0  &  \lambda_{\ov\BE}\, \iota_{\ov\BE}(\pi)\\
	           -\lambda_{\ov\BE}\, \iota_{\ov\BE}(\pi)  &  0
			\end{bmatrix}}_{(n-2)/2 \text{ times}}\biggr).
\end{align*}

\begin{remark}
By \cite[Lem.~3.5]{RSZ}, the associated $F/F_0$-hermitian space $\BV(\BX_n)$ defined in \eqref{defV} is the non-split space of dimension $n$.  (Note that, since the polarization \eqref{lambda_BX_2} is a scalar multiple of the one considered in loc.\ cit.\ when $n = 2$, the splitness of the hermitian spaces is the same.)
\end{remark}

Unfortunately, we know less about the structure of $\CN_n$ than in the unramified case, but at least the analog of Theorem \ref{structN}\eqref{structN i} is known.

\begin{theorem}[Exotic smoothness {\cite[Prop.~3.8]{RSZ}}]\label{exotsmeven}
Recall that $n$ is even. The formal scheme $\CN_n$ is formally locally of finite type, separated, essentially proper, and formally smooth of relative formal dimension $n-1$ over $\Spf O_{\breve {F}}$. In particular, $\CN_n$ is regular of formal dimension $n$.  \qed
\end{theorem}

The analog of Theorem \ref{structN}\eqref{structN ii} is only known for low values of $n$.  To give the precise formulation, we first need to explain a natural decomposition of $\CN_n$ that occurs in the even ramified setting.  Let
\[
   \BM
	\quad\text{and}\quad
	\BN := \BM \otimes_{O_{\breve F_0}} \breve F_0
\]
denote the covariant relative Dieudonn\'e module and rational Dieudonn\'e module, respectively, of $\BX_n$.  The action $\iota_{\BX_n}$ makes $\BM$ into an $O_F \otimes_{O_{F_0}} O_{\breve F_0} = O_{\breve F}$-module. The polarization $\lambda_{\BX_n}$ induces a nondegenerate alternating $\breve F_0$-bilinear form \aform on $\BN$ satisfying
\[
   \la ax,y \ra = \la x, \ov a y \ra
	\quad\text{for all}\quad
	x,y \in \BN,\ a \in \breve F.
\]
The form
\begin{equation*}
   h(x,y) := \la \pi x, y \ra + \la x,y \ra\pi,
	\quad
	x,y \in \BN,
\end{equation*}
then makes $\BN$ into an $\breve F/\breve F_0$-hermitian space of dimension $n$.  By Dieudonn\'e theory, for a perfect field extension $K$ of $\ov k$, the set of $K$-points on $\CN_n$ identifies with a certain subset $\CS$ of $O_{\breve F} \otimes_{O_{\breve F_0}} W(K)$-lattices $M \subset \BN \otimes_{O_{\breve F_0}} W(K)$, all of which are $\pi$-modular, i.e.\ $M^\vee = \pi^{-1}M$, where $M^\vee$ denotes the common dual lattice with respect to $h$ and \aform.
For a lattice $M \in \CS$, let us say that the corresponding $K$-point on $\CN_n$ lies in $\CN_n^+$ or $\CN_n^-$ according as the $O_{\breve F} \otimes_{O_{\breve F_0}} W(K)$-length of the module
\begin{equation}\label{parity test lattice}
   \bigl(M + \BM \otimes_{O_{\breve F_0}} W(K)\bigr) \big/ \BM \otimes_{O_{\breve F_0}} W(K)
\end{equation}
is even or odd.  The parity of this length may also be described as follows.  Since the $\pi$-modular lattices in a hermitian space are all conjugate under the unitary group, and since $\BM \otimes_{O_{\breve F_0}} W(K)$ is itself $\pi$-modular in $\BN \otimes_{O_{\breve F_0}} W(K)$, there exists an element $g \in \U(\BN)(\breve F_0 \otimes_{O_{\breve F_0}} W(K))$ such that $g\cdot (\BM \otimes_{O_{\breve F_0}} W(K)) = M$.  The determinant $\det g$ is a norm one element in $\breve F \otimes_{O_{\breve F_0}} W(K)$, and hence lies in $O_{\breve F} \otimes_{O_{\breve F_0}} W(K)$ with reduction mod $\pi$ equal to $\pm 1 \in K$.  Then the length of \eqref{parity test lattice} is even or odd according as this reduction is $1$ or $-1$ (independent of the choice of $g$) by \cite[Lem.~3.2]{RSZ}.

\begin{proposition}\label{even n decomp lem}
$\CN_n^+$ and $\CN_n^-$ define a decomposition
\[
   \CN_n = \CN_n^+ \amalg \CN_n^-
\]
into open and closed formal subschemes.
\end{proposition}

\begin{proof}
We must show that the parity of the length of \eqref{parity test lattice} is a locally constant function on (perfect points of) the scheme $(\CN_n)_\red$.  Let $x \in (\CN_n)_\red$, and let
\[
   R := \varinjlim \bigl[\CO_{(\CN_n)_\red,x} \xra{\Frob} \CO_{(\CN_n)_\red,x} \xra{\Frob} \dotsb\bigr]
\]
be the direct perfection of the local ring $\CO_{(\CN_n)_\red,x}$ at $x$.  Since $(\CN_n)_\red$ is locally of finite type over $\Spec \ov k$, it suffices to show that the parity in question is constant on points of $\Spec R$ (which is homeomorphic to $\Spec \CO_{(\CN_n)_\red,x}$).  For this we apply Gabber's result, reproved by Lau in \cite[Th.~6.4]{Lau}, that the category of $p$-divisible groups over a perfect ring is equivalent to the category of Dieudonn\'e modules over the corresponding Witt vector ring.  Since $R$ is a perfect local ring, the rings
\[
   W := W(R) \quad\text{and}\quad W' := O_{\breve F} \otimes_{\breve F_0} W(R)
\]
are local and $\varpi$-torsion-free.  Thus to the tautological $R$-point $\Spec R \to \CN_n$ corresponds a $W'$-submodule
\[
   M \subset \BN \otimes_{O_{\breve F_0}} W
\]
which is free of rank $n$ and $\pi$-modular in the obvious sense, i.e.~the dual module
\[
   M^\vee := \bigl\{\, x \in \BN \otimes_{O_{\breve F_0}} W \bigm| h(x,M) \subset W' \,\bigr\}
\]
is equal to $\pi^{-1}M$ inside $\BN \otimes_{O_{\breve F_0}} W$.  

To proceed, we claim that the unitary group acts transitively on the $\pi$-modular submodules in $\BN \otimes_{O_{\breve F_0}} W$.  Indeed, this follows by the usual argument for hermitian spaces over complete, discretely valued fields; let us give a sketch.  It suffices to show that any $\pi$-modular $L$ admits a $W'$-basis with respect to which the hermitian form is $\bigl[\begin{smallmatrix} & -\pi\\ \pi \end{smallmatrix}\bigr] \oplus \dotsb \oplus \bigl[\begin{smallmatrix} & -\pi\\ \pi \end{smallmatrix}\bigr]$.  Let $e \in L$ be an element in a basis for $L$.  Since $L^\vee = \pi\i L$, there exists $f \in L$ such that $h(e,f) = \pi$; and furthermore $h(e,e), h(f,f) \in \pi W' \cap W = \varpi W$.  Say $h(e,e) = \varpi a$ for $a \in W$.  Then
\[
   h\biggl(e + \frac{\pi a}2 f, e + \frac{\pi a}2 f\biggr) = - \frac{\varpi a^2}4 h(f,f) \subset \varpi^2 W.
\]
Continuing by successive approximation, we reduce to the case that $e$ is isotropic. After adding an appropriate multiple of $\pi e$ to $f$, we may then assume that $f$ is isotropic too. Again since $L$ is $\pi$-modular, $h(e,L)$ and $h(f,L)$ are contained in $\pi W'$, and it follows that we can split $W'e + W'f$ off from $L$ as an orthogonal direct summand.  The claim now follows by induction.

By the claim, there exists $g \in \U(\BN)(W[\frac 1 \varpi])$ such that $g\cdot (\BM \otimes_{O_{\breve F_0}} W) = M$.  Since $\det g$ is a norm one element in $W'[\frac 1 \pi]$, we have $\det g \in W'$, and $a := \det g \bmod \pi$ is a square one element in $R$.  Since $R$ is a local ring of residue characteristic not $2$, it follows that $a = \pm 1$ in $R$. Since the parity of the length in question at each residue field of $R$ is determined by the image of $a$ in the residue field, the lemma follows.
\end{proof}

\begin{example}[$n = 2$]\label{CN_2}
Proposition \ref{even n decomp lem} generalizes \cite[Lem.~6.1]{RSZ}, which shows that $(\CN_2)_\red$ is a disjoint union of two points.  In fact, when $n = 2$, the space $\CN_2$ is already defined over $\Spf O_{\breve F_0}$ (the Kottwitz condition $\charac( \iota(\pi) \mid \Lie X ) = T^2- \varpi$ makes sense over $O_{\breve F_0}$, and after base change to $\Spf O_{\breve F}$ this implies the wedge condition \eqref{wedge condition}).  Let $(\CN_2)_{\Spf O_{\breve F_0}}$ denote this descended formal scheme. Then the decomposition in the lemma is induced by a decomposition
\[
   (\CN_2)_{\Spf O_{\breve F_0}} = (\CN_2)_{\Spf O_{\breve F_0}}^+ \amalg (\CN_2)_{\Spf O_{\breve F_0}}^-,
\]
and the summands on the right-hand side are both isomorphic to the universal deformation space $\CM$ of ``the'' formal $O_{F_0}$-module of relative height $2$ and dimension $1$ over $\ov k$; see \cite[Prop.~6.3]{RSZ}  (which remains valid over $\Spf O_{\breve F_0}$; here, in keeping with our definition of $\lambda_{\BX_2}$ as $-2$ times the polarization of the framing object defined in \cite{RSZ}, we also multiply the polarization in the definition of the isomorphism $\CM \isoarrow (\CN_2)_{\Spf O_{\breve F_0}}^+$ by $-2$). 
Note that when $F_0=\BQ_p$, the formal scheme $\CM$ can be identified with the completion of the moduli space of elliptic curves at a supersingular point.
\end{example}

Now let $\CN_n^\pm$ denote either of the formal subschemes $\CN_n^+$ and $\CN_n^-$.  The following analog of Theorem \ref{structN}\eqref{structN ii} has recently been proved by H. Wu.

\begin{theorem}[Wu \cite{Wu}]\label{stratification conj even n ram}
The reduced underlying scheme $(\CN_n^\pm)_\red$ has a Bruhat--Tits stratification by Deligne--Lusztig varieties of dimensions $0, 1,\dotsc, \frac{n}{2} -1$ attached to \emph{non-split} orthogonal groups in an  even number of variables and to Coxeter elements, with strata parametrized by the vertices of the Bruhat--Tits complex of $\SU(\BV(\BX_n))$.\qed
\end{theorem}

This result was previously shown to hold in the equal characteristic analog studied by G\"ortz--He \cite{Goertz-He}.

\begin{example}[$n = 4$]\label{CN_4 eg}
In the case of $\CN_4^\pm$, the description in \cite{Goertz-He} is that every irreducible component in $(\CN_4^\pm)_\red$ is a projective line containing $p^2 + 1$ special points, and every special point lies in the intersection of $p + 1$  projective lines.  The following figure illustrates the case $p = 3$.

\medskip
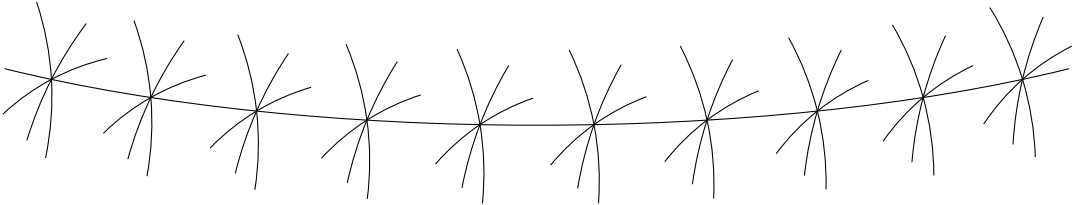
\begin{center}
\begin{tikzpicture}
   \draw (-7,0) .. controls (-3,-1) and (3,-1) .. (7,0)
	foreach \t in {0, ..., 9} {
	   pic [pos=0.05+0.1*\t, rotate=1.5*(-4.5+\t)] {code={
		   \draw (0,0) arc [radius = 4, start angle=10, end angle=25];
	      \draw (0,0) arc [radius = 4, start angle=10, end angle=-5];
      	\draw (0,0) arc [radius = 3, start angle=125, end angle=110];
	      \draw (0,0) arc [radius = 3, start angle=125, end angle=140];
			\draw (0,0) arc [radius = 5, start angle=160, end angle=150];
			\draw (0,0) arc [radius = 5, start angle=160, end angle=170];
		}}
	};
\end{tikzpicture}
\end{center}
\end{example}

\section{Ramified, odd, almost $\pi$-modular type}\label{odd ram max}

In this section we define the formal scheme $\CN_n$ over $\Spf O_{\breve F}$ when $F/F_0$ is ramified and $n \geq 1$ is odd.  We define $\CN_n$ to be the moduli space for quadruples $(X, \iota, \lambda, \rho)$ as before, with $\rho$ a framing into a fixed framing object, except this time we impose that the polarization $\lambda$ is \emph{almost $\pi$-modular}, i.e.
\begin{equation*}
   \ker \lambda \subset X[\iota(\pi)]\  \text{\emph{is of rank $q^{n-1}$}.} 
\end{equation*} 
We furthermore require that the triple $(X,\iota,\lambda)$ satisfies condition \eqref{odd spin cond} below; this condition is a little complicated to formulate and will require some preparation. (For the relationship of \eqref{odd spin cond} to the Kottwitz, wedge, and spin conditions introduced previously, see Remark \ref{odd N_n crit}.)  Since we will also need an analog of this condition in \S\ref{aux spaces} for even $n$, for the moment let $n$ be any positive integer. Let
\[
   m := \lfloor n/2 \rfloor.
\]
Let $e_1,\dotsc,e_n$ denote the standard basis in $F^n$, and let $h$ be the standard split $F/F_0$-hermitian form on $F^n$ with respect to this basis,
\begin{equation}\label{herm form}
   h(ae_i,be_j) := a \ov b \delta_{i,n+1-j} \quad\text{(Kronecker delta)}.
\end{equation}
Let \aform and \sform be the respective alternating and symmetric $O_{F_0}$-bilinear forms $F^n \times F^n \to F_0$ defined by
\begin{equation}\label{aform and sform}
   \langle x,y \rangle := \frac 1 2 \tr_{F/F_0} \bigl(\pi^{-1}h(x,y)\bigr)
   \quad\text{and}\quad
   (x,y) := \frac 1 2 \tr_{F/F_0}h(x,y).
\end{equation}
For $i = bn + c$ with $0 \leq c < n$, define the $O_F$-lattice
\[
   \Lambda_i := \sum_{j = 1}^c \pi^{-b-1}O_F e_j + \sum_{j = c+1}^n \pi^{-b} O_F e_j \subset F^n.
\]
For each $i$, the form \aform induces a perfect pairing
\[
   \Lambda_i \times \Lambda_{-i} \to O_{F_0}.
\]
In this way, for fixed nonempty $I \subset \{0,\dotsc,m\}$, the set
\[
   \Lambda_I := \{\, \Lambda_i \mid i \in \pm I + n\BZ \, \}
\]
forms a polarized chain of $O_F$-lattices over $O_{F_0}$ in the sense of \cite[Def.\  3.14]{RZ}.

Now define the $2n$-dimensional $F$-vector space
\begin{equation*}
   V := F^n \otimes_{F_0} F,
\end{equation*}
where $F$ acts on the right tensor factor.
The $n$th wedge power $\tensor*[^n] V {} := \bigwedge_F^n V$ admits a canonical decomposition
\begin{equation}\label{nV decomp}
   \tensor*[^n] V {} = 
      \bigoplus_{\substack{r+s = n\\\epsilon \in \{\pm1\}}} \tensor*[^n]{V}{_\epsilon^{r,s}}
\end{equation}
which is described in \cite[\S\S2.3, 2.5]{S}.%
\footnote{Here and below we replace the symbol $W$ used in loc.\ cit.\ with $\tensor*[^n] V{}$.}
Let us briefly review it.  The operator $\pi \otimes 1$ acts $F$-linearly on $V$ with eigenvalues $\pm \pi$; let
\[
   V = V_\pi \oplus V_{-\pi}.
\]
denote the corresponding eigenspace decomposition. For a partition $r + s = n$, define%
\footnote{Here and below we interchange $r$ and $s$ in the notation relative to \cite{S}.}
\[
   \tensor*[^n]V{^{r,s}} := \sideset{}{_F^r}\bigwedge V_\pi \otimes_F \sideset{}{_F^s}\bigwedge V_{-\pi},
\]
which is naturally a subspace of $\tensor*[^n] V{}$.  Furthermore, the symmetric form \sform splits after base change to $V$, and therefore there is a decomposition
\begin{equation}\label{^nV decomp}
   \tensor*[^n] V{} = \tensor*[^n] V{_1} \oplus \tensor*[^n] V{_{-1}}
\end{equation}
as an $\SO(\sform)(F)$-representation.  The subspaces $\tensor*[^n]V{_{\pm 1}}$ have the property that for any Lagrangian (i.e.~totally isotropic $n$-dimensional) subspace $\CF \subset V$, the line $\bigwedge_F^n \CF \subset \tensor*[^n] V{}$ is contained in one of them, and in this way they distinguish the two connected components of the orthogonal Grassmannian $\OGr(n,V)$ over $\Spec F$.  The subspaces $\tensor*[^n]V{_{\pm 1}}$ are canonical up to labeling, and we will follow the labeling conventions in loc.\ cit., to which we refer the reader for details.  The summands in the decomposition \eqref{nV decomp} are then given by
\[
   \tensor*[^n] V{_\epsilon^{r,s}} := \tensor*[^n] V{^{r,s}} \cap \tensor*[^n] V {_\epsilon}
\]
(intersection in $\tensor*[^n] V{}$) for $\epsilon \in \{\pm 1\}$.

Given an $O_F$-lattice $\Lambda \subset F^n$, now define
\[
   \tensor*[^n]\Lambda{} := \sideset{}{_{O_F}^n} \bigwedge (\Lambda \otimes_{O_{F_0}} O_F),
\]
which is naturally a lattice in $\tensor*[^n]V{}$.  For fixed $r$, $s$, and $\epsilon$, define
\begin{equation}\label{^nLambda_epsilon^r,s}
   \tensor*[^n]\Lambda{_{\epsilon}^{r,s}} := \tensor*[^n]\Lambda{}\cap \tensor*[^n] V{_\epsilon^{r,s}}
\end{equation}
(intersection in $\tensor*[^n] V{}$).  Then $\tensor*[^n]\Lambda{_{\epsilon}^{r,s}}$ is a direct summand of $\tensor*[^n]\Lambda{}$, since the quotient $\tensor*[^n]\Lambda{} / \tensor*[^n]\Lambda{_{\epsilon}^{r,s}}$ is torsion-free.  For an $O_F$-scheme $S$, define
\begin{equation}\label{L def}
   L_{i,\epsilon}^{r,s}(S) := 
	   \im \bigl[ \tensor*[^n]{(\Lambda_i)}{_{\epsilon}^{r,s}} \otimes_{O_F} \CO_S \rightarrow \tensor*[^n]\Lambda{_i} \otimes_{O_F} \CO_S \bigr].
\end{equation}
For nonempty $I \subset \{0,\dotsc,m\}$, let $\uAut(\Lambda_I)$ denote the scheme of automorphisms of the polarized $O_F$-lattice chain $\Lambda_I$ over $\Spec O_{F_0}$, in the sense of \cite[Th.~3.16]{RZ} or \cite[p.~581]{P} (this is denoted by $\CP$ in \cite{P}).

\begin{lemma}\label{L stability}
For any $O_F$-scheme $S$ and $\Lambda_i \in \Lambda_I$, the submodule $L_{i,\epsilon}^{r,s}(S) \subset \tensor*[^n]\Lambda{_i} \otimes_{O_F} \CO_S$ is stable under the natural action of $\uAut(\Lambda_I)(S)$ on $\tensor*[^n]\Lambda{_i} \otimes_{O_F} \CO_S$.
\end{lemma}

\begin{proof}
Let $C \subset \uAut(\Lambda_I)_{O_F}$ be the stabilizer of $L_{i,\epsilon}^{r,s}$,
\[
   C(S) := \bigl\{\, g \in \uAut(\Lambda_I)(S) \bigm| g \cdot L_{i,\epsilon}^{r,s}(S) = L_{i,\epsilon}^{r,s}(S) \,\bigr\}.
\]
Then $C$ is a closed subscheme of $\uAut(\Lambda_I)_{O_F}$.  Since $\uAut(\Lambda_I)_{F_0} \cong \U(h)$, it is obvious that $C$ contains the $F$-generic fiber $\uAut(\Lambda_I)_F$.  Therefore $C = \uAut(\Lambda_I)_{O_F}$, since $\uAut(\Lambda_I)_{O_F}$ is smooth, and hence flat, over $\Spec O_F$ by \cite[Th.\ 3.16]{RZ} and \cite[Th.\ 2.2(a)]{P}.
\end{proof}

This concludes our discussion for general $n$.  We now formulate our condition on the triple $(X,\iota,\lambda)$ over a $\Spf O_{\breve F}$-scheme $S$ in the case of odd $n$, which will make use of the above discussion in the case $I = \{m\}$. Let $M(X)$ and $M(X^\vee)$ denote the respective Lie algebras of the universal vector extensions of $X$ and $X^\vee$.  Since $\ker\lambda \subset X[\iota(\pi)]$, there is a unique (necessarily $O_F$-linear) isogeny $\lambda'$ such that the composite
\[
   X \xra{\lambda} X^\vee \xra{\lambda'} X
\]
is $\iota(\pi)$.  Since $\ker\lambda$ furthermore has rank $q^{n-1}$, the induced diagram
\[
   M(X) \xra{\lambda_*} M(X^\vee) \xra{\lambda'_*} M(X)
\]
extends periodically to a polarized chain of $O_F \otimes_{O_{F_0}} \CO_S$-modules of type $\Lambda_{\{m\}}$, in the terminology of \cite{RZ}.  By Theorem 3.16 in loc.\ cit., \'etale-locally on $S$ there exists an isomorphism of polarized chains
\begin{equation}\label{chain triv}
   [{}\dotsb \xra{\lambda'_*} M(X) \xra{\lambda_*} M(X^\vee) \xra{\lambda'_*} \dotsb{}] \isoarrow \Lambda_{\{m\}} \otimes_{O_{F_0}} \CO_S,
\end{equation}
which in particular gives an isomorphism of $O_F \otimes_{O_{F_0}} \CO_S$-modules
\begin{equation}\label{M(X) triv}
   M(X) \isoarrow \Lambda_{-m} \otimes_{O_{F_0}} \CO_S.
\end{equation}
The module $M(X)$ fits into the covariant Hodge filtration
\[
   0 \to \Fil^1 \to M(X) \to \Lie X \to 0
\]
for $X$, and the condition we finally impose is that
\begin{align*}
	&\text{\em upon identifying $\Fil^1$ with a submodule of $\Lambda_{-m} \otimes_{O_{F_0}} \CO_S$ via \eqref{M(X) triv}, the line bundle}\\
	\shortintertext{
	\begin{equation}\label{odd spin cond}
	   \sideset{}{_{\CO_S}^n}\bigwedge \Fil^1 \subset  \tensor[^n]\Lambda{_{-m}} \otimes_{O_F} \CO_S
	\end{equation}
	}
	&\text{\em is contained in $L_{-m,-1}^{n-1,1}(S)$.}
\end{align*}
Note that Lemma \ref{L stability} gives exactly what is needed to conclude that condition \eqref{odd spin cond} is independent of the choice of chain isomorphism in \eqref{chain triv}.

To complete the definition of $\CN_n$, it remains to specify a framing object $(\BX_n, \iota_{\BX_n}, \lambda_{\BX_n})$ for this moduli problem over $\ov k$. In contrast to the cases we have encountered previously, such a triple is \emph{not unique up to quasi-isogeny}. In fact, there are two isogeny classes (as always, up to $O_F$-linear quasi-isogeny compatible with the polarizations) such that $\BX_n$ is supersingular, corresponding to the two possible isometry classes of the hermitian space $C$ in the proof of \cite[Prop.~3.1]{RSZ}.\footnote{Note that the statement in loc.\ cit.\ considers quasi-isogenies which are compatible with the polarizations only up to scalar, whereas here we are requiring compatibility on the nose.}  As an explicit representative for which $\BV$ is non-split, we take the same framing object as in \cite{RSZ}.  Thus when $n = 1$ we define
\begin{equation}\label{BX_1}
   \BX_1^{(1)} := \BE,
	\quad
	\iota_{\BX_1^{(1)}} := \iota_\BE,
	\quad\text{and}\quad
	\lambda_{\BX_1^{(1)}} := -\lambda_\BE.
\end{equation}
Then $\BV(\BX_1^{(1)}) = \Hom_{O_F}^\circ(\ov\BE,\BE)$ is the $-1$-eigenspace $D^-$ for the conjugation action by $\iota_\BE(\pi)$ in $D$, endowed with the hermitian norm $(x,x) = -\RN x$, which is indeed non-split.  When $n \geq 3$, we take 
\begin{equation}\label{BX_n odd}
   \BX_n^{(1)} := \BX_{n-1} \times \ov \BE,
	\quad
	\iota_{\BX_n^{(1)}} := \iota_{\BX_{n-1}} \times \iota_{\ov\BE},
	\quad\text{and}\quad
	\lambda_{\BX_n^{(1)}} := \lambda_{\BX_{n-1}} \times \lambda_{\ov\BE};
\end{equation}
here $n-1$ is even and $\BX_{n-1}$ is as defined in \S\ref{even ram max}.  Then by \eqref{chi decomp formula} and the fact that $\BV(\BX_{n-1})$ is the non-split hermitian space of dimension $n-1$, $\BV(\BX_n^{(1)})$ is indeed the non-split space of dimension $n$.  To fix a framing object in the other isogeny class, for all $n$ we simply multiply the polarization of $\BX_n^{(1)}$ by a non-norm unit, i.e.~we fix $\ep\in O_{F_0}^\times\smallsetminus \RN F^\times$ and define
\begin{equation}\label{BX_n^{(0)}}
   \BX_n^{(0)} := \BX_n^{(1)},
	\quad
	\iota_{\BX_n^{(0)}} := \iota_{\BX_n^{(1)}},
	\quad\text{and}\quad
	\lambda_{\BX_n^{(0)}} := \ep\lambda_{\BX_n^{(1)}}.
\end{equation}
Then, since $n$ is odd and $\ep \notin \RN F^\times$, $\BV(\BX_n^{(0)})$ is the split hermitian space of dimension $n$.  Note that such an $\ep$ exists since $F/F_0$ is ramified.

Taking $\BX_n^{(0)}$ and $\BX_n^{(1)}$ as the framing objects, we obtain respective moduli spaces  $\CN_n^{(0)}$ and $\CN_n^{(1)}$. However, these spaces are isomorphic via the map
\[
   \xymatrix@R=0ex{
	   \CN_n^{(1)} \ar[r]^-\sim  &  \CN_n^{(0)}\\
		(X, \iota, \lambda, \rho) \ar@{|->}[r]  &  \bigl(X, \iota, \lambda\circ\iota(\ep), \rho\bigr).
	}
\]
To simplify notation in the rest of the paper, from now on we set
\[
   \CN_n := \CN_n^{(1)}
	\quad\text{and}\quad
	\BX_n := \BX_n^{(1)}.
\]

\begin{example}[$n=1$]\label{CN_1}
When $n = 1$, the moduli problem for $\CN_1$ is just the moduli problem of lifting $\BE$ as a formal $O_F$-module.  Thus the solution is $\CN_1 = \Spf O_{\breve F}$, with universal object the canonical lift $(\CE, \iota_\CE, -\lambda_\CE, \rho_\CE)$.  (In this case condition \eqref{odd spin cond} is redundant in the moduli problem.)
\end{example}

\begin{remark}\label{odd N_n crit}
Note that the Kottwitz condition \eqref{kottwitzcond}, the wedge condition \eqref{wedge condition}, and the spin condition \eqref{spin condition} all continue to make sense as written in the odd ramified setting.  The first two of these conditions are implied by condition \eqref{odd spin cond}, cf.~\cite[Lem.~5.1.2, Rem.~5.2.2]{S} (which shows that these implications hold on the local model\footnote{Strictly speaking, loc.~cit.~shows these these implications hold in the case of signature opposite to ours, but up to isomorphism the local model is the same, cf.~\cite[p.~19\ fn.~5]{RSZ}. The same remark applies to essentially all of the subsequent references we make to the local models in \cite{A}, \cite{PR}, \cite[\S2.6]{PRS}, and \cite{S}.}).  On the other hand, let $\CN_n^\circ$ be the moduli space of quadruples $(X,\iota,\lambda,\rho)$ as in the definition of $\CN_n$, except that instead of imposing condition \eqref{odd spin cond}, we impose conditions \eqref{wedge condition} and \eqref{spin condition}.  Then \eqref{wedge condition} and \eqref{spin condition} imply condition \eqref{odd spin cond}, and in this way $\CN_n^\circ$ is an open formal subscheme of $\CN_n$.  Indeed, this statement follows from the analogous statement for the corresponding local models, which is explained in \cite[\S3.3]{S}.   (More precisely, loc.~cit.~shows that the local model for $\CN_n^\circ$ is the complement of the ``worst point'' in the local model for $\CN_n$.)  We note that when $n \geq 3$, the framing objects $\BX_n^{(1)}$ and $\BX_n^{(0)}$ obviously satisfy \eqref{wedge condition} and \eqref{spin condition} (because $\BX_{n-1}$ does), and therefore they indeed satisfy \eqref{odd spin cond}; and it is trivial to check directly that $\BX_1^{(1)}$ and $\BX_1^{(0)}$ satisfy \eqref{odd spin cond} when $n=1$. We further note that \cite{RSZ} uses the notation $\CN_n$ for our $\CN_n^\circ$, and our $\CN_n$ is the ``better'' formal scheme alluded to in Rem.~3.13 of loc.~cit.
\end{remark}

\begin{remark}
There is a natural analog of condition \eqref{odd spin cond} on $\CN_n$ when $n$ is even (still with $F/F_0$ ramified).  However this analog is automatically satisfied on the whole space, which follows from the fact that this condition is automatically satisfied in the generic fiber of the local model for $\CN_n$, and this local model is already flat (in fact smooth).
\end{remark}

As when $n$ is even, we know less about the structure of $\CN_n$ than in the unramified case, but at least we have the analog of Theorem \ref{structN}\eqref{structN i}. 

\begin{theorem}[Exotic smoothness]\label{exotsmodd} Recall that $n$ is odd. 
The formal scheme $\CN_n$ is formally locally of finite type, separated, essentially proper, and formally smooth of relative formal dimension $n-1$ over $\Spf O_{\breve {F}}$. In particular, $\CN_n$ is regular of dimension $n$.  
\end{theorem}

\begin{proof}
This is proved in the same way as \cite[Prop.~3.8]{RSZ}, using that the local model for $\CN_n$ is a closed subscheme of the ``naive'' local model, and is furthermore smooth by Richarz's result \cite[Prop.~4.6]{A} and \cite[Th.\ 1.4]{S}.  (Note that \cite[Prop.~3.8]{RSZ} proves that $\CN_n^\circ$ satisfies all the conclusions of this theorem except essential properness.)
\end{proof}

The analog of Theorem \ref{structN}\eqref{structN ii} has again been recently proved by Wu, and was previously known to hold in the equal characteristic analog \cite{Goertz-He}.

\begin{theorem}[Wu \cite{Wu}]\label{BT conj odd n ram}
The reduced underlying scheme $(\CN_n)_\red$ has a Bruhat--Tits stratification by Deligne--Lusztig varieties of dimensions $0, 1,\ldots, \frac{n-1}{2}$ attached to orthogonal groups in an  odd number of variables and to Coxeter elements, with strata  parametrized by the vertices of the Bruhat--Tits complex of $\SU(\BV(\BX_n^{(1)}))$.\qed
\end{theorem}

\begin{example}[$n=3$]
In the case $n=3$, as in Example \ref{CN_4 eg}, there are strata of dimensions $0$ and $1$ in $(\CN_3)_\red$. Every irreducible component is a projective line containing $p + 1$ special points, and every special point lies in the intersection of $p + 1$ projective lines.  The following figure illustrates the case $p = 3$.

\medskip
\begin{center}
\begin{tikzpicture}
   \draw (-4,0) .. controls (-2,-0.75) and (2,-0.75) .. (4,0)
	foreach \t in {0, ..., 3} {
	   pic [pos=0.125+0.25*\t, rotate=3*(-1.5+\t)] {code={ 
		   \draw (0,0) arc [radius = 4, start angle=10, end angle=25];
	      \draw (0,0) arc [radius = 4, start angle=10, end angle=-5];
      	\draw (0,0) arc [radius = 3, start angle=125, end angle=110];
	      \draw (0,0) arc [radius = 3, start angle=125, end angle=140];
			\draw (0,0) arc [radius = 5, start angle=160, end angle=150];
			\draw (0,0) arc [radius = 5, start angle=160, end angle=170];
		}}
	};
\end{tikzpicture}
\end{center}
\end{example}
 
Taking Theorems \ref{exotsmeven} and \ref{exotsmodd} together, we obtain the following corollary. For arbitrary $n \geq 2$, define
\begin{equation*}
   \CN_{n-1, n} := \CN_{n-1}\times_{\Spf O_{\breve {F}}}\CN_n .
\end{equation*}

\begin{corollary} For arbitrary $n$, the formal scheme $\CN_{n-1, n}$ is regular of formal dimension $2(n-1)$.  \qed
\end{corollary}

\begin{remark}\label{self-dual special parahoric}
For $n$ even or odd, the space $\CN_n$ is an (open and closed formal subscheme of an) RZ space with ``special maximal parahoric level structure'' in the sense that, in the notation of \cite[\S1.38, Def.~3.18]{RZ}, it arises from a PEL datum with $B = F$, $V = F^n$, $\sform = \aform$, $b^* = \ov b$, $G = \GU(h)$, and $\CL = \Lambda_{\{m\}}$; and the stabilizer of the lattice chain $\Lambda_{\{m\}}$ in $\GU(h)(F_0)$ is a special maximal parahoric subgroup.  For $F/F_0$ ramified, up to conjugacy there is one further case of a special maximal parahoric subgroup in $\GU(h)(F_0)$, namely the stabilizer of a self-dual lattice in $F^n$ when $n$ is odd \cite[\S1.2.3]{PR}.  The corresponding RZ space, in which the polarization $\lambda$ in the moduli problem is \emph{principal} (still with Kottwitz condition \eqref{kottwitzcond} of signature $(1,n-1)$), is studied for $n$ both even and odd in \cite{RTW}.  This space is not regular for $n \geq 3$, and therefore it does not appear to fit into the framework of this paper.  The case $n = 2$ is exceptional,\footnote{Exceptional in three ways: in this case the space is defined over $\Spf O_{\breve F_0}$ instead of $\Spf O_{\breve F}$; it is regular; and the corresponding parahoric subgroup of $\GU(h)(F_0)$ is not a maximal parahoric subgroup, but an Iwahori subgroup, cf.\ \cite[Rem.~2.35]{PRS}.} however, and we now turn to it
in the next section.
\end{remark}

\section{Ramified self-dual type, $n = 2$}\label{ram self-dual n=2}

In this section, continuing from Remark \ref{self-dual special parahoric}, we consider the formal moduli space of \cite{RTW} in the special case $n = 2$, still with $F/F_0$ ramified; see also \cite{KR-alt}.  This is a moduli space for quadruples $(X, \iota, \lambda, \rho)$ as before, with $(X,\iota)$ satisfying the Kottwitz condition
\begin{equation}\label{Kottwitz cond n=2}
   \charac\bigl( \iota(\pi) \mid \Lie X\bigr) = T^2- \varpi,
\end{equation}
and where this time the polarization $\lambda$ is principal.  As in the previous section in the case of odd $n$ (and via the same argument), there are \emph{two} isogeny classes of supersingular framing objects over $\ov k$.  
Accordingly we obtain two moduli spaces, both defined over $\Spf O_{\breve F_0}$, which we denote by $\wt\CN_2^{(0)}$ and $\wt\CN_2^{(1)}$ according as the hermitian space \eqref{defV} of the framing object is split or non-split. These two spaces really are different, in contrast to the case of $\CN_n^{(0)}$ and $\CN_n^{(1)}$ in the previous section. In fact, we can describe both of these spaces explicitly.

\begin{altenumerate}
\item
For the formal scheme $\wt\CN_2^{(0)}$, there is a natural isomorphism
\begin{equation}\label{ramified Drinfeld isom}
   \wt\CN_2^{(0)}\cong \wh\Omega^2_{F_0}\times_{\Spf O_{F_0}}\Spf O_{\breve F_0}
\end{equation} 
which is equivariant for the respective actions of $\SU(\wt\BX_2^{(0)}) \cong \SL_2(F_0)$ on both sides.
This is the alternative description of the Drinfeld half-plane in  \cite{KR-alt}. 
\item\label{Gamma_0(varpi) isom}
For the formal scheme $\wt\CN_2^{(1)}$, there is an isomorphism
\begin{equation*}
   \wt\CN_2^{(1)}\simeq \Spf O_{\breve F_0}[[x, y]]/(xy-\varpi) .
\end{equation*}
\end{altenumerate}

More precisely, we are now going to show in \eqref{Gamma_0(varpi) isom} that $\wt\CN_2^{(1)}$ is isomorphic to the deformation space with Iwahori level structure $\CM_{\Gamma_0(\varpi)}$. Let us recall the definition of this latter space.
 
\begin{definition}
The formal scheme $\CM_{\Gamma_0(\varpi)}$ represents the functor over $\Spf O_{\breve F_0}$ that associates to each scheme $S$ the set of isomorphism classes of quadruples 
\[
   (Y,Y',\phi\colon Y\rightarrow Y', \rho_Y),
\]
where $Y$ and $Y'$ are formal $O_{F_0}$-modules of relative height $2$ and dimension $1$ over $S$, $\phi$ is an isogeny of degree $q$, and $\rho_Y \colon Y\times_S\ov S\to \BE\times_{\Spec \ov k}\ov S$ is an $O_{F_0}$-linear quasi-isogeny of height $0$.
\end{definition}

When $F_0=\BQ_p$, it follows that $\CM_{\Gamma_0(p)}$ can be identified with the formal completion at a supersingular point of the moduli space of elliptic curves with $\Gamma_0(p)$-structure. 

To define the desired isomorphism $\CM_{\Gamma_0(\varpi)} \to \wt\CN_2^{(1)}$, let $(Y,Y',\phi,\rho_Y)$ be an $S$-point on $\CM_{\Gamma_0(\varpi)}$.  Set
\[
   X := Y \times Y'.
\]
Since $\ker \phi$ is $\varpi$-power torsion and of rank $q$, i.e.~of relative $O_{F_0}$-height $1$, it is killed by $\varpi$.  Hence there exists a unique (necessarily $O_{F_0}$-linear) isogeny $\phi'$ such that the composite
\begin{equation}\label{phi'}
   Y \xra\phi Y' \xra{\phi'} Y
\end{equation}
is multiplication by $\varpi$.  Let
\[
   \iota(\pi) :=
	\begin{bmatrix}
		  &  \phi'\\
	   \phi
	\end{bmatrix} \in \End_{O_{F_0}}(X).
\]
Then $\iota(\pi)^2$ is multiplication by $\varpi$.  Hence $\iota(\pi)$ defines an $O_F$-action $\iota$ on $X$ extending the $O_{F_0}$-action.  It is easy to verify that $\iota$ satisfies the Kottwitz condition \eqref{Kottwitz cond n=2}.  Now define the composite quasi-isogeny over the special fiber $\ov S$,
\begin{equation}\label{rho_Y'}
   \rho_{Y'}\colon Y_{\ov S}' \xra{\phi^{-1}} Y_{\ov S} \xra{\rho_Y} \BE_{\ov S} \xra{\iota_\BE(\pi)} \BE_{\ov S}.
\end{equation}
Then the pullbacks $\rho_Y^*(\lambda_\BE)$ and $\rho_{Y'}^*(\lambda_\BE)$ lift to principal polarizations $\lambda_Y$ of $Y$ and $\lambda_{Y'}$ of $Y'$, respectively (since the same holds for the universal object over the Lubin--Tate space $\CM$ of formal $O_{F_0}$-modules of dimension $1$ and relative height $2$ equipped with a quasi-isogeny of height $0$ to $\BE$ in the special fiber).  Let
\[
   \lambda :=
	\begin{bmatrix}
		\lambda_Y\\
		  &  \lambda_{Y'}
	\end{bmatrix}
	\in \Hom_{O_{F_0}}(X,X^\vee).
\]

In the particular case that $S = \Spec \ov k$ and $(Y,Y',\phi,\rho_Y) = (\BE,\BE,\iota_\BE(\pi),\id_\BE)$, our construction produces the triple
\begin{equation}\label{BX def}
   \BX := \BE \times \BE,
	\quad
	\iota_\BX(\pi) :=
	\begin{bmatrix}
		  &  \iota_\BE(\pi)\\
		\iota_\BE(\pi)
	\end{bmatrix},
	\quad
	\lambda_\BX :=
	\begin{bmatrix}
		\lambda_\BE\\
		  &  \lambda_\BE
	\end{bmatrix}.
\end{equation}
It is straightforward to calculate that the space $\BV(\BX)$ defined in \eqref{defV} has no nonzero isotropic vectors, so that $\BV(\BX)$ is the non-split $F/F_0$-hermitian space of dimension $2$.  Hence we may and will take $(\BX,\iota_\BX,\lambda_\BX)$ as the framing object for $\wt\CN_2^{(1)}$.\footnote{Note that later on, for technical convenience, we will rescale the polarization $\lambda_\BX$, cf.\ Example \ref{CP_n diagram n=2}.}  For a general $S$ and quadruple $(Y,Y',\phi,\rho_Y)$, we define the framing map
\[
   \rho := \rho_Y \times \rho_{Y'}\colon X_{\ov S} \to \BX_{\ov S}.
\]
Then tautologically $\rho^*(\lambda_\BX) = \lambda_{\ov S}$, and $\rho$ is readily seen to be $O_F$-linear.  Since it is obvious that $\Ros_{\lambda_\BX}(\iota_\BX(\pi)) = - \iota_\BX(\pi)$, it follows that $\Ros_\lambda(\iota(\pi)) = - \iota(\pi)$.  In this way we have defined a morphism
\begin{equation}\label{M_Gamma_0 map}
\begin{gathered}
   \xymatrix@R=0ex{
	   \CM_{\Gamma_0(\varpi)} \ar[r]  &  \wt\CN_2^{(1)}\\
		(Y,Y',\phi,\rho_Y) \ar@{|->}[r]  &  (X,\iota,\lambda,\rho).
	}
\end{gathered}
\end{equation}

\begin{proposition}\label{M_Gamma_0 isom}
The morphism $\CM_{\Gamma_0(\varpi)} \to \wt\CN_2^{(1)}$ is an isomorphism.
\end{proposition} 

\begin{proof}
We first show that the source and target each have only one $\ov k$-valued point. For $\CM_{\Gamma_0(\varpi)}$, this is well-known and easily checked. For $\wt\CN_2^{(1)}$, consider the covariant isocrystal $\BN$ of the framing object $\BX$, which has the structure of an $\breve F/\breve F_0$-hermitian space of dimension $2$ as in \S\ref{even ram max}. By Dieudonn\'e theory, $\wt\CN_2^{(1)}(\ov k)$ identifies with the set of $O_{\breve F}$-lattices $M$ in \BN which are self-dual and satisfy $\varpi M\subset^2 VM\subset^2 M$, where $V$ denotes the Verschiebung on $\BN$. Let $\sigma$ denote the Frobenius operator on $\breve F_0$, let $\zeta \in O_{\breve F_0}^\times$ be a square root of $-1$, and let $\tau := \zeta \pi V^{-1}$. Then $\tau$ is a $\sigma$-linear operator on $\BN$ with all slopes equal to zero. Let $C := \BN^\tau$.  Then the restriction of the hermitian form to $C$ makes $C$ into a $2$-dimensional $F/F_0$-hermitian space, cf.~\cite[pp.~1170--1]{RTW}.\footnote{Note that the quantity $\eta$ in loc.~cit.~should be a square root of $-\epsilon^{-1}$, rather than  a square root of $\epsilon^{-1}$.}  We claim that $C$ is non-split.  Indeed we can see this in at least two ways.  First consider the framing object $\BX_2$ of $\CN_2$ defined in \S\ref{even ram max}.  The $O_{F_0}$-linear isogeny
\[
   \varsigma_0 \colon \BX_2 = \BE \times \BE \xra{\id \times \iota_\BE(\pi)} \BX = \BE \times \BE
\]
is in fact $O_F$-linear (in terms of the Serre construction $\BX_2 = O_F \otimes_{O_{F_0}} \BE$, $\varsigma_0$ is induced by the inclusion of $\BE$ into the first factor in the target) and satisfies $\varsigma_0^*(\lambda_\BX) = -\frac 1 2\lambda_{\BX_2}$.  Hence $\varsigma_0$ gives an identification of $\breve F/\breve F_0$-hermitian isocrystals between the (rescaled) isocrystal for $\BX_2$ and $\BN$.  The ``$\tau$-fixed'' space in the former isocrystal is non-split by \cite[Lem.~3.3]{RSZ} (as is used in the proofs of Prop.~3.1 and Lem.~6.1 in loc.~cit.), which proves the claim.  A second way to see that $C$ is non-split is by explicit computation: up to isomorphism, the isocrystal with $F$-action for $\BE$ is given by $\breve F_0^2$ with Verschiebung
\[
   \begin{bmatrix} & \varpi \\ 1 \end{bmatrix} \sigma\i
\]
and with $\pi$ acting by
\[
   \begin{bmatrix} & \varpi \\ 1 \end{bmatrix}.
\]
(Of course the alternating form on this $2$-dimensional space is uniquely determined up to scalar.) From this and \eqref{BX def} one obtains an explicit form for $\BN$, and it is easy to then calculate that $C$ has no nonzero isotropic vectors, which characterizes it as the non-split space of dimension $2$.

Now let $M \subset \BN$ correspond to a point in $\wt\CN_2^{(1)}(\ov k)$. By an obvious variant of \cite[Prop.~2.17]{RZ} (with $\breve F$ in place of $W_\BQ$, $\pi$ in place of $p$, etc.), since $\BN$ is $2$-dimensional over $\breve F$, either $M=\tau M$, or $M\neq \tau M$ and $M+\tau M=\tau(M+\tau M)$. In the first case $M^\tau$ is a self-dual $O_F$-lattice in $C$, and hence is the unique such lattice since $C$ is non-split of dimension $2$. The other case does not occur, since otherwise $\pi(M+\tau M)^\tau$ would be a $\pi$-modular lattice in $C$ (via the same proof as for \cite[Lem.~3.2(iii)]{KR-alt}), contrary to the fact that $C$ is non-split.  This completes the proof that $\wt\CN_2^{(1)}(\ov k)$ consists of a single point.

To complete the proof of the proposition, it now suffices to show that the morphism \eqref{M_Gamma_0 map} is formally \'etale. We will do this via the local models for the source and target spaces, which we now pause to discuss.
\let\qed\relax
\end{proof}

The local model for $\CM_{\Gamma_0(\varpi)}$ is the standard Iwahori local model for $\GL_2$ and the cocharacter $\mu = (1,0)$ over $\Spec O_{F_0}$, cf.~\cite{Goertz} or \cite[Eg.~2.4]{PRS}.  Let us review the definition.  Let $f_1,f_2$ denote the standard basis for $F_0^2$, and define the $O_{F_0}$-lattices
\[
   \lambda_0 := O_{F_0}f_1 + O_{F_0}f_2,
	\quad\text{and}\quad
	\lambda_1 := O_{F_0}\varpi\i f_1 + O_{F_0} f_2.
\] 
The local model $M_{\Gamma_0(\varpi)}$ associated to these data is the scheme representing the functor that associates to each $O_{F_0}$-scheme $S$ the set of pairs $(\CF_0,\CF_1)$, where $\CF_i \subset \lambda_i \otimes_{O_{F_0}} \CO_S$ is an $\CO_S$-locally direct summand of rank $1$ for each $i$, and where the natural maps
\begin{equation}\label{Phi Phi'}
   \Phi\colon \lambda_0 \otimes_{O_{F_0}} \CO_S \to \lambda_1 \otimes_{O_{F_0}} \CO_S
	\quad\text{and}\quad
	\Phi'\colon \lambda_1 \otimes_{O_{F_0}} \CO_S \xra{\varpi \otimes \id} \lambda_0 \otimes_{O_{F_0}} \CO_S
\end{equation}
carry $\CF_0$ into $\CF_1$ and $\CF_1$ into $\CF_0$, respectively.

To define the local model for $\wt\CN_2^{(1)}$, we use the notation of \S\ref{odd ram max} in the case $n = 2$.  In particular, recall the $O_F$-lattice $\Lambda_0 = O_F^2 \subset F^2$, which is self-dual with respect to the alternating form \aform defined in \eqref{aform and sform}.  The local model $\wt N_2^{(1)}$ is the scheme representing the functor that associates to each $O_{F_0}$-scheme $S$ the set of all $O_F \otimes_{O_{F_0}} \CO_S$-submodules $\CF \subset \Lambda_0 \otimes_{O_{F_0}} \CO_S$ which are $\CO_S$-locally direct summands of rank $2$, which are La\-grangian for the form $\aform \otimes_{O_{F_0}} \CO_S$, and which satisfy the Kottwitz condition $\charac(\pi \otimes 1 \mid \CF) = T^2 - \varpi$.  (This is Definition \ref{LM def} below in the special case $(r,s) = (1,1)$ and $I = \{0\}$, cf.~also Remark \ref{forgetful LM isom n=2}.)

We now define an analog of the map \eqref{M_Gamma_0 map} for $M_{\Gamma_0(\varpi)}$ and $\wt N_2^{(1)}$. First make the direct sum $\lambda_0 \oplus \lambda_1$ into an $O_F$-module by taking $S = \Spec O_{F_0}$ in \eqref{Phi Phi'} and letting $\pi$ act as
\[
   \begin{bmatrix}
		  &  \Phi'\\
		\Phi
	\end{bmatrix}.
\]
Of course this is well-defined since the square of this matrix is multiplication by $\varpi$. We then identify
\begin{equation}\label{lambda_0 oplus lambda_1}
   \lambda_0 \oplus \lambda_1 \cong \Lambda_0
\end{equation}
as $O_F$-modules by identifying the elements $f_2 \in \lambda_0$ and $\varpi^{-1}f_1 \in \lambda_1$ with the standard basis elements $e_1,e_2 \in \Lambda_0$, respectively.  Relative to these identifications, we define
\begin{equation}\label{LM isom n=2}
\begin{gathered}
   \xymatrix@R=0ex{
	   M_{\Gamma_0(\varpi)} \ar[r]  &  \wt N_2^{(1)}\\
		(\CF_0,\CF_1) \ar@{|->}[r]  &  \CF_0 \oplus \CF_1.
	}
\end{gathered}
\end{equation}
It is easy to see that $\CF_0 \oplus \CF_1$ satisfies all the conditions in the definition of $\wt N_2^{(1)}$, so that this morphism is well-defined.

\begin{lemma}\label{M_Gamma_0 LM isom}
The morphism $M_{\Gamma_0(\varpi)} \to \wt N_2^{(1)}$ is an isomorphism.
\end{lemma}

\begin{proof}
This is easily checked locally on standard affine charts. For example, the most interesting chart on the source consists of all points of the form
\[
   \CF_0 = \spann\{ f_1 + x f_2\},\quad \CF_1 = \spann\{y \varpi\i f_1 + f_2\}
\]
such that $xy = \varpi$, cf.~\cite[Eg.~2.4]{PRS}.  The most interesting chart on the target consists of all points expressible as the column span of
\[
   \begin{bmatrix}
		x_{11}  &  x_{12}\\
		x_{21}  &  x_{22}\\
		1\\
		  &  1
	\end{bmatrix}
\]
relative to the $O_{F_0}$-basis $e_1,e_2,\pi e_1, \pi e_2$ for $\Lambda_0$; this is easily seen to be the scheme of such matrix entries such that $x_{11} = x_{22} = 0$ and $x_{12}x_{21} = \varpi$, cf.~\cite[Rem.~2.35]{PRS}.  Thus the map \eqref{LM isom n=2} visibly identifies these two charts.  One sees similarly (and in fact more easily) that \eqref{LM isom n=2} is an isomorphism on all other standard charts, as desired.
\end{proof}

\begin{remark}
The map \eqref{LM isom n=2} makes explicit the isomorphism between these local models alluded to at the end of \cite[Rem.~2.35]{PRS}.
\end{remark}

\begin{proof}[Completion of the proof of Proposition \ref{M_Gamma_0 isom}]
It remains to show that the map \eqref{M_Gamma_0 map} is formally \'etale, which follows by combining Lemma \ref{M_Gamma_0 LM isom} with the \emph{local model diagrams}
\[
   \begin{gathered}
   \xy
	   (0,17)*+{\CM_{\Gamma_0(\varpi)}^\triv}="triv";
		(-16,0)*+{\CM_{\Gamma_0(\varpi)}}="RZ";
		(16,0)*+{(M_{\Gamma_0(\varpi)})_{\Spf O_{\breve F_0}}}="loc";
		{\ar "triv";"RZ" _-\mu};
		{\ar "triv";"loc" ^-{\mu'}}; 
	\endxy
	\end{gathered}
	\quad\text{and}\quad
	\begin{gathered}
	\xy
	   (0,17)*+{\bigl(\wt\CN_2^{(1)}\bigr)^\triv}="triv";
		(-16,0)*+{\wt\CN_2^{(1)}}="RZ";
		(16,0)*+{\bigl(\wt N_2^{(1)}\bigr)_{\Spf O_{\breve F_0}}}="loc";
		{\ar "triv";"RZ" _-\nu};
		{\ar "triv";"loc" ^-{\nu'}}; 
	\endxy
	\end{gathered}.
\]
Here $\CM_{\Gamma_0(\varpi)}^\triv$ is a moduli space for tuples $(Y,Y',\phi,\rho_Y,\gamma_0,\gamma_1)$, where the first four entries form a point on $\CM_{\Gamma_0(\varpi)}$, and where, letting $S$ denote the base scheme and $M$ denote the functor that assigns to a $p$-divisible group the Lie algebra of its universal vector extension,
\[
   \gamma_0 \colon M(Y) \isoarrow \lambda_0 \otimes_{O_{F_0}} \CO_S
	\quad\text{and}\quad
	\gamma_1 \colon M(Y') \isoarrow \lambda_1 \otimes_{O_{F_0}} \CO_S
\]
are isomorphisms of $\CO_S$-modules.  We require that $\gamma_0$ and $\gamma_1$ respect the alternating forms on both sides, where the forms on the targets are defined by restricting \aform via \eqref{lambda_0 oplus lambda_1}; and that the diagram
\[
   \begin{gathered}
   \xymatrix{
	   M(Y) \ar[r]^-{\phi_*} \ar[d]_-{\gamma_0}^-\sim
		   &  M(Y') \ar[r]^-{\phi'_*} \ar[d]_-{\gamma_1}^-\sim
			&  M(Y) \ar[d]_-{\gamma_0}^-\sim\\
	   \lambda_0 \otimes_{O_{F_0}} \CO_S \ar[r]^-\Phi  
		   &  \lambda_1 \otimes_{O_{F_0}} \CO_S \ar[r]^-{\Phi'} 
			&  \lambda_0 \otimes_{O_{F_0}} \CO_S
	}
	\end{gathered}
\]
commutes, where $\phi'$ is as in \eqref{phi'}.  Similarly, $(\wt\CN_2^{(1)})^\triv$ is a moduli space for points $(X,\iota,\lambda,\rho)$ on $\wt\CN_2^{(1)}$ endowed with an $O_F \otimes_{O_{F_0}} \CO_S$-linear trivialization $\gamma\colon M(X) \isoarrow \Lambda_0 \otimes_{O_{F_0}} \CO_S$ that respects the alternating forms on both sides.  Furthermore $\mu$ and $\nu$ are the natural forgetful maps, and
\begin{align*}
   \mu'(Y,Y',\phi,\rho_Y,\gamma_0,\gamma_1) &:= \bigl(\gamma_0(\ker[M(Y) \rightarrow \Lie Y]), \gamma_1(\ker[M(Y') \rightarrow \Lie Y'])\bigr),\\
	\nu'(X,\iota,\lambda,\rho) &:= \gamma(\ker[M(X) \rightarrow \Lie X]).
\end{align*}
It is obvious that the maps \eqref{M_Gamma_0 map} and \eqref{LM isom n=2}, together with the map
\[
   \xymatrix@R=0ex{
	   \CM_{\Gamma_0(\varpi)}^\triv \ar[r]  &  \bigl(\wt\CN_2^{(1)}\bigr)^\triv \\
		(Y,Y',\phi,\rho_Y,\gamma_0,\gamma_1) \ar@{|->}[r]  &  (X,\iota,\lambda,\rho, \gamma_0 \oplus \gamma_1)
	}
\]
defined in terms of \eqref{M_Gamma_0 map} and the identification \eqref{lambda_0 oplus lambda_1}, give a morphism of local model diagrams.  Since the map \eqref{LM isom n=2} on local models is an isomorphism, it follows from \cite[Prop.~3.33]{RZ} that \eqref{M_Gamma_0 map} is formally \'etale.
\end{proof}

\begin{remark}[Equivariance] The group $O_D^\times = \Aut_{O_{F_0}}(\BE)$ acts naturally on $\CM_{\Gamma_0(\varpi)}$ by composing with the framing $\rho_Y$; hence so does its subgroup $D^1$ of norm one elements. Similarly, $\U(\BV(\wt\BX_2^{(1)}))$ acts on $\wt\CN_2^{(1)}$; hence so does its subgroup $\SU(\BV(\wt\BX_2^{(1)}))$. The isomorphism \eqref{M_Gamma_0 map} is then equivariant for these actions under the natural identification $D^1 \cong \SU(\wt\BX_2^{(1)})$.  Note that this equivariance does not express the actions of the larger groups $O_D^\times$ and $\U(\wt\BX_2^{(1)})$ in terms of the other side of the isomorphism \eqref{M_Gamma_0 map}. 
\end{remark}

\section{Some auxiliary spaces in the ramified case}\label{aux spaces}

In this section we give two closely related generalizations (different than in \cite{RTW}) of the space $\wt\CN_2^{(1)}$ of the previous section to any even $n$.  However the spaces we define are not regular, and they will play only an auxiliary role in the ATC we formulate in \S\ref{s:ATCeven}.  Let $n \geq 2$ be even, still with $F/F_0$ ramified.

The first space we define, which we denote by $\CP_n$, is a moduli space over $\Spf O_{\breve F}$ for quadruples $(X,\iota,\lambda,\rho)$ as before, with the framing object to be specified below, except in this case we impose on the polarization $\lambda$ that
\begin{equation*}
   \ker \lambda \subset X[\iota(\pi)]\  \text{\emph{is of rank $q^{n-2}$}}. 
\end{equation*}
We further require that the triple $(X,\iota,\lambda)$ satisfies the wedge condition \eqref{wedge condition} and condition \eqref{even wtCN spin cond} below, which is the natural analog of condition \eqref{odd spin cond}.  As in \S\ref{odd ram max}, let $M(X)$ and $M(X^\vee)$ denote the respective Lie algebras of the universal vector extensions of $X$ and $X^\vee$.  Since $\ker \lambda$ is contained in $X[\iota(\pi)]$ and of rank $q^{n-2}$, there is a unique (necessarily $O_F$-linear) isogeny $\lambda'$ such that the composite
\[
   X \xra{\lambda} X^\vee \xra{\lambda'} X
\]
is $\iota(\pi)$, and the induced diagram
\[
   M(X) \xra{\lambda_*} M(X^\vee) \xra{\lambda'_*} M(X)
\]
then extends periodically to a polarized chain of $O_F \otimes_{O_{F_0}} \CO_S$-modules of type $\Lambda_{\{m-1\}}$.  By \cite[Th.~3.16]{RZ}, \'etale-locally on $S$ there exists an isomorphism of polarized chains
\begin{equation*}
   [{}\dotsb \xra{\lambda'_*} M(X) \xra{\lambda_*} M(X^\vee) \xra{\lambda'_*} \dotsb{}] \isoarrow \Lambda_{\{m-1\}} \otimes_{O_{F_0}} \CO_S,
\end{equation*}
which in particular gives an isomorphism of $O_F \otimes_{O_{F_0}} \CO_S$-modules
\begin{equation}\label{M(X) triv 2}
   M(X) \isoarrow \Lambda_{-(m-1)} \otimes_{O_{F_0}} \CO_S.
\end{equation}
Denoting by $\Fil^1 \subset M(X)$ the covariant Hodge filtration for $X$, the analog of \eqref{odd spin cond} we impose is that
\begin{align*}
	&\text{\em upon identifying $\Fil^1$ with a submodule of $\Lambda_{-(m-1)} \otimes_{O_{F_0}} \CO_S$ via \eqref{M(X) triv 2}, the line bundle}\\
	\intertext{
	\begin{equation}\label{even wtCN spin cond}
      \sideset{}{_{\CO_S}^n}\bigwedge \Fil^1 \subset  \tensor[^n]\Lambda{_{-(m-1)}} \otimes_{O_F} \CO_S
	\end{equation}
	}
	&\text{\em is contained in $L_{-(m-1),-1}^{n-1,1}(S)$, cf.~(\ref{L def}).}
\end{align*}
Just as before, condition \eqref{even wtCN spin cond} is independent of the above choice of chain isomorphism by Lemma \ref{L stability}.

As in the previous two sections, over $\ov k$ there are \emph{two} supersingular isogeny classes of framing objects for this moduli problem, distinguished by the splitness of the hermitian space $\BV$ in \eqref{defV}.  To complete the definition of $\CP_n$, we will choose a framing object $\wt\BX_n$ for which $\BV(\wt\BX_n)$ is non-split.  First let $\ov\BX_1$ be the framing object $\BX_1 = \BX_1^{(1)}$ (cf.\ \eqref{BX_1}) endowed with the conjugate $O_F$-action,
\[
   \bigl(\ov\BX_1, \iota_{\ov\BX_1}, \lambda_{\ov\BX_1}\bigr) :=
	   \bigl(\ov\BE,\iota_{\ov\BE},-\lambda_{\ov\BE}\bigr).
\]
Then we take $\wt\BX_n$ to be the product of the framing object $\BX_{n-1} = \BX_{n-1}^{(1)}$ (cf.\ \eqref{BX_n odd}) and $\ov\BX_1$,
\begin{equation}\label{wtBX_n^(1) def}
   \bigl(\wt\BX_n, \iota_{\wt\BX_n}, \lambda_{\wt\BX_n} \bigr) := 
	   \bigl(\BX_{n-1} \times \ov\BX_1, \iota_{\BX_{n-1}} \times \iota_{\ov\BX_1}, \lambda_{\BX_{n-1}} \times \lambda_{\ov\BX_1} \bigr).
\end{equation}
Since $\BV(\BX_{n-1})$ is non-split and $n$ is even, it follows from \eqref{chi decomp formula} that $\BV(\wt\BX_n)$ is indeed the non-split space of dimension $n$.  Furthermore, $\wt\BX_n$ obviously satisfies \eqref{wedge condition} because $\BX_{n-1}$ does, and it is easy to see that $\wt\BX_n$ satisfies \eqref{even wtCN spin cond} because $\BX_{n-1}$ satisfies \eqref{odd spin cond}.

\begin{remark}\label{kottwitzonP2}
As in Remark \ref{odd N_n crit}, the Kottwitz condition \eqref{kottwitzcond} is implied on $\CP_n$ by condition \eqref{even wtCN spin cond}, cf.~\cite[Lem.~5.1.2]{S}.  Furthermore, it will be shown in forthcoming work of Yu \cite{Yu} that \eqref{even wtCN spin cond} also implies the wedge condition \eqref{wedge condition} (and hence \eqref{wedge condition} is redundant in the definition of $\CP_n$).  On the other hand, when $n = 2$ the Kottwitz condition implies \eqref{wedge condition} and \eqref{even wtCN spin cond}, since this is obviously true in the generic fiber of the local model, and the local model defined by the Kottwitz condition alone is already flat. Hence $\CP_2 \simeq (\wt\CN_2^{(1)})_{\Spf O_{\breve F}}$.  (Strictly speaking, to make such an identification precise we must specify a quasi-isogeny between the framing objects $\wt\BX_2$ and $\BX$ on the two sides; we will do so in Example \ref{CP_n diagram n=2} below.)  Note that although $\wt\CN_2^{(1)}$ is regular, the space $\CP_2$ is not, since the ramified base change $O_{\breve F_0} \to O_{\breve F}$ does not preserve semi-stable reduction.
\end{remark}

\begin{remark}
Of course one obtains another formal moduli space by taking a framing object in the supersingular isogeny class for which $\BV$ is split; this recovers the space $(\wt\CN_2^{(0)})_{\Spf O_{\breve F}}$ when $n = 2$. As the case $n = 2$ shows, and in contrast to the situation in \S\ref{odd ram max}, the resulting space really is different from $\CP_n$.  However, when $n \geq 4$ we will have no occasion to consider this space further.
\end{remark}

It follows from the general theory of RZ spaces \cite{RZ} that $\CP_n$ is formally locally of finite type and essentially proper over $\Spf O_{\breve F}$.  However, as we have already noted, it is not regular.  Having fixed the polarization type in the moduli problem for $\CP_n$, the remaining conditions in its definition are the ``right'' ones in the sense that the corresponding local model is flat, which will be shown by Yu in \cite{Yu}, cf.~Remark \ref{Yu flatness} below.

The second moduli space we define, which we denote by $\CP_n'$, will turn out to be isomorphic to two copies of $\CP_n$, and in this sense is another generalization of $\wt\CN_2^{(1)}$. The advantage of $\CP_n'$ is that, by definition, it will also admit a natural map to $\CN_n$.  To state the definition, we first need to fix an $O_F$-linear isogeny of degree $q$,
\begin{equation}\label{phi_0}
   \phi_0 \colon \BX_n \to \wt\BX_n,
\end{equation}
such that $\phi_0^*(\lambda_{\wt\BX_n}) = \lambda_{\BX_n}$; here we recall the framing object $\BX_n$ from \S\ref{even ram max}. When $n = 2$, we have
\[
   \BX_2 = O_F \otimes_{O_{F_0}} \BE
	\quad\text{and}\quad
	\wt\BX_2 = \BE \times \ov\BE
\]
as $O_F$-modules, and we define $\phi_0$ to be the $O_F$-linear isogeny induced by adjunction from the diagonal map $\BE \xra{(\id,\id)} \BE \times \BE$. It is easy to verify from the explicit expressions \eqref{lambda_BX_2} and \eqref{wtBX_n^(1) def} of the polarizations that $\phi_0^*(\lambda_{\wt\BX_2}) = \lambda_{\BX_2}$.  When $n \geq 4$, we have
\[
   \BX_n = \wt\BX_n = \BX_{n-2} \times \ov\BE \times \ov\BE
\]
as $O_F$-modules, and we define
\[
   \phi_0 := \id_{\BX_{n-2}} \times
	   \begin{bmatrix}
			1  &  \iota_{\ov\BE}(\pi)/2\\
			1  &  -\iota_{\ov\BE}(\pi)/2
		\end{bmatrix}.
\]
It is again easy to check that $\phi_0^*(\lambda_{\wt\BX_n}) = \lambda_{\BX_n}$.

We now define $\CP_n'$ to be the moduli space for tuples
\[
   (X, \iota, \lambda, \rho, X', \iota', \lambda', \phi\colon X \rightarrow X'),
\]
where $(X, \iota, \lambda, \rho)$ is a point on $\CN_n$, $(X',\iota',\lambda',\rho')$ is a point on $\CP_n$, and $\phi$ is an $O_F$-linear isogeny of degree $q$.  Here we define $\rho'$ to be the composite quasi-isogeny
\[
   \rho'\colon X'_{\ov S} \xra{\phi_{\ov S}\i} X_{\ov S} \xra\rho \BX_{n,\ov S} \xra{\phi_{0,\ov S}} \wt\BX_{n,\ov S},
\]
where $S$ denotes the base scheme. Note that since $\phi_0^*(\lambda_{\wt\BX_n}) = \lambda_{\BX_n}$, $\rho^*(\lambda_{\BX_n,\ov S}) = \lambda_{\ov S}$, and $\rho^{\prime *}(\lambda_{\wt\BX_n,\ov S}) = \lambda'_{\ov S}$, it follows from rigidity for quasi-isogenies that $\phi^*(\lambda') = \lambda$.  The notion of isomorphism between tuples as above is the obvious one.

By definition, there are tautological projection maps
\begin{equation}\label{CP_n diagram}
	\begin{gathered}
   \xymatrix@C-2ex{
      & \CP_n' \ar[dl] \ar[dr]^-{\varphi}\\
	  \CP_n & & \CN_n .
   }
	\end{gathered}
\end{equation}
By Proposition \ref{even n decomp lem}, $\CN_n$ decomposes into a disjoint union $\CN_n = \CN_n^+ \amalg \CN_n^-$.  Pulling back along $\varphi$, we obtain a decomposition
\begin{equation}\label{wtCN'pm}
   \CP_n' = (\CP_n')^+ \amalg (\CP_n')^-.
\end{equation}

\begin{theorem}\label{wtCN'pm isom}
Writing $(\CP_n')^\pm$ for either of the summands in \eqref{wtCN'pm}, the projection $\CP_n' \to \CP_n$ induces an isomorphism
\[
   (\CP_n')^\pm \isoarrow \CP_n.
\]
\end{theorem}

\begin{proof}
We follow the same strategy as in the proof of Proposition \ref{M_Gamma_0 isom}.  We first show that the map is a bijection on $\ov k$-points. Let $\BN$ denote the covariant isocrystal of the framing object $\BX_n$, made into an $\breve F/\breve F_0$-hermitian space as in \S\ref{even ram max}. Let $V$ denote the Verschiebung on $\BN$.  By Dieudonn\'e theory, $\CP_n'(\ov k)$ identifies with the set of pairs of $O_{\breve F}$-lattices $L \subset L'$ in $\BN$ such that $\varpi L \subset VL \subset L$ and $\varpi L' \subset VL' \subset L'$, such that\footnote{Here the condition $VL' \subset^{\leq 1} VL' + \pi L'$ is manifestly equivalent to the wedge condition in the moduli problem for $\CP_n$, via the canonical isomorphism $\Lie X' \cong L'/VL'$.  For the purposes of the proof, we only need to know that the moduli problem \emph{implies} this condition on $L'$; we leave it as an exercise to show that, for $\ov k$-points, condition \eqref{even wtCN spin cond} imposes nothing further.}
\begin{equation*}
   VL \subset^1 VL + \pi L
	\quad\text{and}\quad
	VL' \subset^{\leq 1} VL' + \pi L',
\end{equation*}
and such that
\begin{equation}\label{L,L'}
   L \subset^1 L' \subset^{n-2} (L')^\vee \subset^1 L^\vee = \pi\i L.
\end{equation}
Our problem is to show that, starting from $L'$, we can uniquely recover $L$. 
Note that a $\pi$-modular lattice $L$ fits into the diagram \eqref{L,L'} if and only if
\begin{equation}\label{L,L' diagram}
   \pi(L')^\vee \subset^1 L \subset^1 L'.
\end{equation}
The hermitian form on $\BN$ induces a nondegenerate \emph{symmetric} $\ov k$-bilinear form on the $2$-dimensional space $L'/\pi(L')^\vee$, and the $\pi$-modular lattices $L$ satisfying \eqref{L,L' diagram} then correspond to the isotropic lines in the $2$-dimensional space $L'/\pi(L')^\vee$.  This gives at most two possibilities for $L$; since these possibilities have odd colength in their sum $L'$, by definition, one of them will correspond to a point on $(\CP_n')^+$, and the other a point on $(\CP_n')^-$.  Thus to complete this part of the proof, we have to show that any $\pi$-modular $L$ satisfying \eqref{L,L' diagram} automatically satisfies $\varpi L \subset VL \subset L$ and $VL \subset^1 VL + \pi L$.  Since the Dieudonn\'e module $\BM$ of $\BX_n$ satisfies $V\BM \subset^1 V\BM + \pi\BM$, it follows from \cite[Lem.~3.3]{RSZ} that the colength of $VL$ in $VL + \pi L$ is odd. The diagram
\begin{equation*}
\begin{gathered}
   \xymatrix@R-2ex@C-2ex{
	   VL + \pi L \ar@{}[d]|-*{\cup} \ar@{}[r]|-*{\subset}  &  V L' + \pi L' \ar@{}[d]|-*{\tensor[^{\makebox[0ex][r]{$\scriptstyle \leq 1$}}]{\cup}{}}\\
		VL \ar@{}[r]|-*{\subset^{\smash{1}}}  &  VL'
	}
\end{gathered}
\end{equation*}
then shows that this colength is $\leq 2$. Hence the colength is $1$.  This implies, in turn, that $\pi$ kills the quotient $(VL + \pi L)/VL$, so that $\varpi L \subset VL$.  The containment $VL \subset L$ follows similarly from the relation $L \subset^1 L + \pi\i VL$, which in turn follows similarly from $L \subset^1 L' \subset^{\leq 1} L' + \pi\i VL'$, where in the second containment we use that, since the isocrystal $\BN$ is supersingular, the lattices $L'$ and $\pi\i V L'$ have the same colength in any lattice that contains them both.
	
To complete the proof of the proposition, since $(\CP_n')^\pm$ is formally locally of finite type over $\Spf O_{\breve F}$, it now suffices to show that $(\CP_n')^\pm \to \CP_n$ is formally \'etale.  This follows from Proposition \ref{lm isom}\eqref{lm isom ii} below in the case $I = \{m-1\}$, via an entirely similar argument involving the local model diagrams as in the proof of Proposition \ref{M_Gamma_0 isom}. The easy details are left to the reader.
\end{proof}

\begin{example}[$n=2$]\label{CP_n diagram n=2}
When $n = 2$, Proposition \ref{M_Gamma_0 isom} identifies $\wt\CN_2^{(1)} \cong \CM_{\Gamma_0(\varpi)}$. Thus by Remark \ref{kottwitzonP2}, $\CP_2$ is isomorphic to the base change $(\CM_{\Gamma_0(\varpi)})_{\Spf O_{\breve F}}$.  To specify a particular isomorphism, it is convenient to first replace the polarization $\lambda_\BX$ of the framing object \BX for $\wt\CN_2^{(1)}$ by $-2 \lambda_\BX$, and to then modify the isomorphism $\CM_{\Gamma_0(\varpi)} \isoarrow \wt\CN_2^{(1)}$ to send
\[
   (Y,Y',\phi,\rho_Y) \mapsto \bigl(Y \times Y', \iota, -2(\lambda_Y \times \lambda_{Y'}), \rho\bigr),
\]
where $\iota$ and $\rho$ are as defined in \S\ref{ram self-dual n=2}.  Up to canonical isomorphism, this leaves the moduli problem for $\wt\CN_2^{(1)}$ unchanged, and the modified morphism remains an isomorphism.  Furthermore, using this rescaled $\lambda_\BX$, and with respect to the $O_{F_0}$-linear decompositions $\wt\BX_2 = \BX = \BE^2$, the morphism
\begin{equation}\label{psi_0}
   \psi_0\colon \BX \xra[\undertilde]{\left[\begin{smallmatrix} 1 & 1\\ 1 & -1 \end{smallmatrix}\right]} \wt\BX_2
\end{equation}
is an $O_F$-linear isomorphism such that $\psi_0^*(\lambda_{\wt\BX_2}) = \lambda_\BX$.  Hence $\psi_0$ determines an isomorphism $(\wt\CN_2^{(1)})_{\Spf O_{\breve F}} \isoarrow \CP_2$ sending $(X,\iota,\lambda,\rho) \mapsto (X,\iota,\lambda,\psi_0 \circ \rho)$, and in this way we identify $\CP_2$ with $(\CM_{\Gamma_0(\varpi)})_{\Spf O_{\breve F}}$.

In fact, the entire diagram \eqref{CP_n diagram} comes by extension of scalars from the diagram over $\Spf O_{\breve F_0}$,
\[
   \xymatrix@C-10ex{
      & \CM_{\Gamma_0(\varpi)} \amalg \CM_{\Gamma_0(\varpi)} \ar[dl] \ar[dr]^-{\varphi}\\
	  \CM_{\Gamma_0(\varpi)} & & \CM \amalg \CM .
   }
\]
Here on the lower right we use the identification $(\CN_2)_{\Spf O_{\breve F_0}} \cong \CM \amalg \CM$, where $\CM$ denotes the Lubin--Tate moduli space of formal $O_{F_0}$-modules of dimension $1$ and relative height $2$ equipped with a quasi-isogeny of height $0$ to $\BE$ in the special fiber, cf.~Example \ref{CN_2}.  By definition the morphism $\varphi$ respects the disjoint union decompositions in its source and target.  Thus $\varphi$ is given by two maps $\CM_{\Gamma_0(\varpi)} \to \CM$, and one easily unwinds the definitions to find that these are the tautological projections $(Y,Y',\phi,\rho_Y) \mapsto (Y,\rho_Y)$ and $(Y,Y',\phi,\rho_Y) \mapsto (Y',\rho_{Y'})$, where $\rho_{Y'}$ is defined in \eqref{rho_Y'}.
\end{example}

\begin{remark}
As Example \ref{CP_n diagram n=2} shows, the morphism $\varphi$ is finite when $n = 2$.  When $n \geq 4$, this is no longer the case.  Indeed, the fiber over a $\ov k$-point $(X,\iota,\lambda,\rho)$ is given by the scheme of isotropic subgroup schemes of order $q$ in $X[\iota(\pi)]$.
\end{remark}

The remaining step in the proof of Theorem \ref{wtCN'pm isom} is to show that the corresponding map on local models is an isomorphism.  In fact we will formulate the result we need for any signature.
For the moment let $n$ be even or odd, fix a partition $r + s = n$, and set
\[
   E := F_0
	\quad\text{if}\quad
	r = s,
	\quad\text{and}\quad
	E = F
	\quad\text{if}\quad
	r \neq s.
\]
We set $m := \lfloor n/2 \rfloor$ and again take up the notation of \S\ref{odd ram max}.

\begin{definition}\label{LM def}
For nonempty $I \subset \{0,\dotsc,m\}$, the \emph{naive local model $M_I^\naive$} is the scheme over $\Spec O_E$ representing the functor that assigns to each $O_E$-scheme $S$ the set of all families
\[
   (\CF_i \subset \Lambda_i \otimes_{O_{F_0}} \CO_S)_{i \in \pm I + n\BZ}
\]
such that
\begin{altenumerate}
\renewcommand{\theenumi}{\arabic{enumi}}
\item\label{LM rank cond} for all $i$, $\CF_i$ is an $O_F \otimes_{O_{F_0}} \CO_S$-submodule of $\Lambda_i \otimes_{O_{F_0}} \CO_S$ which is Zariski-locally on $S$ an $\CO_S$-direct summand of rank $n$;
\item\label{LM functoriality cond} for all $i < j$, the natural arrow $\Lambda_i \otimes_{O_{F_0}} \CO_S \to \Lambda_j \otimes_{O_{F_0}} \CO_S$ carries $\CF_i$ into $\CF_j$;
\item\label{LM periodic cond} for all $i$, the isomorphism $\Lambda_i \otimes_{O_{F_0}} \CO_S \xra[\undertilde]{\pi \otimes 1} \Lambda_{i-n} \otimes_{O_{F_0}} \CO_S$ identifies $\CF_i \isoarrow \CF_{i-n}$;
\item\label{LM perp cond} for all $i$, the perfect $\CO_S$-bilinear pairing
\[
   \aform_i\colon (\Lambda_i \otimes_{O_{F_0}} \CO_S) \times (\Lambda_{-i} \otimes_{O_{F_0}} \CO_S)
   \xra{\aform \otimes \CO_S} \CO_S
\]
identifies $\CF_i^\perp$ with $\CF_{-i}$ inside $\Lambda_{-i} \otimes_{O_{F_0}} \CO_S$; and
\item (\emph{Kottwitz condition}) for all $i$, the section $\pi \otimes 1 \in O_F \otimes_{O_{F_0}} \CO_S$ acts on $\CF_i$ as an $\CO_S$-linear endomorphism with characteristic polynomial
\[
   \charac(\pi \otimes 1 \mid \CF_i) = (T-\pi)^r(T+\pi)^s \in \CO_S[T].
\]
\setcounter{filler}{\value{enumi}}
\end{altenumerate}

The \emph{local model $M_I^\loc$} is the scheme-theoretic closure of the generic fiber in $M_I^\naive$.  We define $M_I$ to be the closed subscheme of $M_I^\naive$ defined by the additional conditions\footnote{Here conditions \eqref{LM wedge cond} and \eqref{LM str spin cond} make sense as written whenever $S$ is an $O_F$-scheme, and when $r = s$ they descend to conditions on $M_I^\naive$ over $\Spec O_E = \Spec O_{F_0}$.}
\begin{altenumerate}
\renewcommand{\theenumi}{\arabic{enumi}}
\setcounter{enumi}{\value{filler}}
\item\label{LM wedge cond} (\emph{wedge condition}) for all $i$, the operators
\[
   \sidewedge{_{\CO_S}^{s+1}} (\pi \otimes 1 - 1\otimes \pi \mid \CF_i ) = 0
	\quad\text{and}\quad
	\sidewedge{_{\CO_S}^{r+1}} (\pi \otimes 1 + 1\otimes \pi \mid \CF_i ) = 0;
\]
and
\item\label{LM str spin cond} for all $i$, the line bundle $\bigwedge_{\CO_S}^n \CF_i \subset  \tensor[^n]\Lambda{_i} \otimes_{O_F} \CO_S$ is contained in $L_{i,(-1)^r}^{r,s}(S)$.
\end{altenumerate}
\end{definition}

We will typically abbreviate the notation for points on $M_I^\naive$ to $(\CF_i)_{i \in I}$, since by conditions \eqref{LM periodic cond} and \eqref{LM perp cond} such a tuple determines the full tuple $(\CF_i)_{i \in \pm I + n\BZ}$.

In the generic fiber, the Kottwitz condition implies conditions \eqref{LM wedge cond} and \eqref{LM str spin cond}. Therefore there are inclusions of closed subschemes
\[
   M_I^\loc \subset M_I \subset M_I^\naive.
\]
On the other hand, \eqref{LM str spin cond} implies the Kottwitz condition in general by \cite[Lem.~5.1.2]{S}.  In the particular situation of Remark \ref{odd N_n crit} (which is for $n$ odd, $(r,s) = (n-1,1)$, and $I = \{m\}$), condition \eqref{LM str spin cond} also implies the wedge condition, but we do not know if this implication holds in general.  In all cases, the following is conjectured in \cite{S}.\footnote{Strictly speaking loc.~cit.~requires $I$ to have the property that if $n$ is even and $m - 1 \in I$, then $m \in I$, but here we drop this requirement. Note that on the other hand, Proposition \ref{lm isom} below allows one to view sets $I$ not satisfying this property as being redundant in some sense.}

\begin{conjecture}\label{str spin conj}
For any signature $(r,s)$ and nonempty $I \subset \{0,\dotsc,m\}$, the scheme $M_I$ is flat over $\Spec O_E$, or in other words $M_I^\loc = M_I$.
\end{conjecture}

\begin{remark}
Condition \eqref{LM str spin cond} can be regarded as a kind of common refinement of the Kottwitz condition and the \emph{spin condition} introduced in \cite{PR}.  Let us recall the latter.  Let $S$ be an $O_F$-scheme, and recall the decomposition $\tensor*[^n] V{} = \tensor*[^n] V{_1} \oplus \tensor*[^n] V{_{-1}}$ in \eqref{^nV decomp}.  For each $i \in \pm I + n\BZ$, in analogy with the definitions of $\tensor*[^n]\Lambda{_{\epsilon}^{r,s}}$ in \eqref{^nLambda_epsilon^r,s} and $L_{i,\epsilon}^{r,s}(S)$ in \eqref{L def}, define
\begin{equation}\label{^nLambda_epsilon}
   \tensor[^n]{(\Lambda_i)}{_\epsilon} := \tensor[^n]{(\Lambda_i)}{}\cap \tensor*[^n] V{_\epsilon}
	\quad\text{and}\quad
	L_{i,\epsilon}(S) := \im \bigl[ \tensor[^n]{(\Lambda_i)}{_\epsilon} \otimes_{O_F} \CO_S \to \tensor*[^n]\Lambda{_i} \otimes_{O_F} \CO_S \bigr].
\end{equation}
Then the \emph{spin condition} on an $S$-point $(\CF_i)_i$ of $M_I^\naive$ is that
\begin{equation}\label{LM spin cond}
	\text{\em for all $i$, the line bundle $\textstyle{\bigwedge_{\CO_S}^n} \CF_i \subset  \tensor[^n]\Lambda{_i} \otimes_{O_F} \CO_S$ is contained in $L_{i,(-1)^r}(S)$.}
\end{equation}
As in the formulation of conditions \eqref{LM wedge cond} and \eqref{LM str spin cond}, when $r = s$ the spin condition descends from $(M_I^\naive)_{O_F}$ to $M_I^\naive$.  Of course, the spin condition is trivially implied by condition \eqref{LM str spin cond}.
\end{remark}

\begin{remark}\label{I=m}
In the special case that $n$ is even and $I = \{m\}$, the above discussion can be simplified.  For $S$ an $O_E$-scheme, consider the $\CO_S$-bilinear form
\begin{equation*}
   \sform_m\colon (\Lambda_m \otimes_{O_{F_0}} \CO_S) \times (\Lambda_m \otimes_{O_{F_0}} \CO_S) \xra[\undertilde]{\id \times (\pi \otimes 1)} (\Lambda_m \otimes_{O_{F_0}} \CO_S) \times (\Lambda_{-m} \otimes_{O_{F_0}} \CO_S) \xra{\aform_m} \CO_S,
\end{equation*}
which is \emph{split symmetric}.  Then, using condition \eqref{LM periodic cond}, condition \eqref{LM perp cond} is equivalent to requiring that $\CF_m$ is Lagrangian for $\sform_m$.  Thus $M_{\{m\}}^\naive$ naturally embeds into the orthogonal Lagrangian Grassmannian $\OGr(n,\Lambda_m \otimes_{O_{F_0}} O_F)$.  Now, the generic fiber $(M_{\{m\}}^\naive)_E$ is connected (cf.~\cite[\S1.5.3]{PR}), and the spin condition is then that $\CF_m$ lies on the one of the two connected components of $\OGr(n,\Lambda_m \otimes_{O_{F_0}} O_F)$ marked by $(M_{\{m\}}^\naive)_E$; cf.~Rem.~2.32 and the paragraph following it in \cite{PRS}.  Thus the spin condition is a purely pointwise condition on $M_{\{m\}}^\naive$; after base change to $\Spec O_F$, it may also be characterized as the condition that the rank of $\pi \otimes 1 - 1\otimes \pi$ on $\CF_m$ at each point has the same parity as the common parity of $r$ and $s$.

The relevant signature in the context of Theorem \ref{wtCN'pm isom} is $(r,s) = (n-1,1)$.  In this case, in the presence of the wedge condition, the spin condition is simply that $\pi \otimes 1 - 1 \otimes \pi$ acts on $\CF_m$ with rank $1$ at each point.  Furthermore, the calculations in \cite[\S5.3]{PR} show that in this case, Conjecture \ref{str spin conj} holds true, and that in fact $M_{\{m\}}^\loc$ is smooth and is characterized just by conditions \eqref{LM rank cond}--\eqref{LM perp cond}, the condition $\bigwedge_{\CO_S}^2 (\pi \otimes 1 - 1\otimes \pi \mid \CF_m ) = 0$, and the spin condition.
\end{remark}

\begin{remark}\label{Yu flatness}
In the special case that $n$ is even, $I = \{m-1\}$, and $(r,s) = (n-1,1)$, Conjecture \ref{str spin conj} will be proved in the forthcoming paper \cite{Yu} of Yu.  Furthermore, she will show that in this case $M_{\{m-1\}}^\loc$ is characterized just by conditions \eqref{LM rank cond}--\eqref{LM perp cond} and \eqref{LM str spin cond}.
\end{remark}

We now return to our assumption that $n$ is even.
For $j \in \{m-1,m\}$, let $C_{\{j\}}$ be the closed subscheme of $\prod_{i \in \pm j + n\BZ} \Gr(n,\Lambda_i \otimes_{O_{F_0}} O_E)$ defined by conditions \eqref{LM rank cond}--\eqref{LM perp cond} in Definition \ref{LM def} and the spin condition \eqref{LM spin cond}.  Then projection onto the $j$th factor identifies $C_{\{j\}}$ with a closed subscheme in $\Gr(n,\Lambda_j\otimes_{O_{F_0}} O_E)$, and $M_{\{j\}}^\naive$ is naturally a closed subscheme of $C_{\{j\}}$.  Note that we do \emph{not} impose the Kottwitz condition on $C_{\{j\}}$, and therefore $C_{\{j\}}$ depends on the signature $(r,s)$ only through the common parity of $r$ and $s$ in the spin condition.

We are going to define a morphism $C_{\{m-1\}} \to C_{\{m\}}$ via the following lemma.  Quite generally, let $S$ be an $O_E$-scheme, and for any $i$, let
\begin{equation}\label{T_i}
   T_i\colon \Lambda_i \otimes_{O_{F_0}} \CO_S \to \Lambda_{i+1} \otimes_{O_{F_0}} \CO_S
\end{equation}
denote the natural map.

\begin{lemma}\label{LM mapping lem}
Let $\CF_{m-1} \subset \Lambda_{m-1} \otimes_{O_{F_0}} \CO_S$ be an $S$-point on $C_{\{m-1\}}$.  Then there exists a unique $\CO_S$-locally direct summand $\CF_m \subset \Lambda_m \otimes_{O_{F_0}} \CO_S$ which contains $T_{m-1}(\CF_{m-1})$, is Lagrangian for the symmetric form $\sform_m$ (cf.~Remark \ref{I=m}), and satisfies the spin condition. Furthermore, $\CF_m$ is an $S$-point on $C_{\{m\}}$, and $T_m(\CF_m) \subset \CF_{m+1}$.
\end{lemma}

\begin{proof}
By the uniqueness claim, it suffices to prove the lemma in the case that $S$ is the spectrum of a local ring $R$.  Let $\kappa$ denote the residue field of $R$, and let us systematically use a bar to denote base change from $R$ to $\kappa$.  Over $\kappa$, the map $\ov T_{m-1}$ has kernel of dimension $1$ or $0$ (according as $\varpi R$ is contained in the maximal ideal or not).  Either way, there exists an $(n-1)$-dimensional subspace in $\ov\CF_{m-1}$ mapped isomorphically by $\ov T_{m-1}$ to its image in $\Lambda_m \otimes_{O_{F_0}} \kappa$.  
Using Nakayama's lemma, it follows that we may choose a free basis $a_1,\dotsc,a_n$ of $\CF_{m-1}$ such that the images $b_1,\dotsc,b_{n-1}$ of the first $n-1$ basis elements under $T_{m-1}$ form a free basis for a summand $\CF$ of $\Lambda_m \otimes_{O_{F_0}} R$.

Now, for any $i$, it is obvious from the definition of the pairings in condition \eqref{LM perp cond} in Definition \ref{LM def} that there is an adjunction relation
\begin{equation}\label{adjunction}
   \bigl\la u, T_{-i-1}(v) \bigr\ra_i = \bigl\la T_i(u), v \bigr\ra_{i+1}
	\quad\text{for all}\quad
	u \in \Lambda_i \otimes_{O_{F_0}} R,\ v \in \Lambda_{-i-1} \otimes_{O_{F_0}} R.
\end{equation}
Applying this when $i = m-1$, and using condition \eqref{LM periodic cond} and the equality $\CF_{-(m-1)} = \CF_{m-1}^\perp$, it follows that $T_{m-1}(\CF_{m-1})$ is totally isotropic for $\sform_m$. Hence
\[
   \CF \subset T_{m-1}(\CF_{m-1}) \subset \CF^\perp.
\]
The quotient $\CF^\perp/\CF$ is then naturally made by $\sform_m$ into a free quadratic $R$-module of rank $2$.  Since $R$ is local, $2 \in R^\times$, $\CF$ is totally isotropic, and $\sform_m$ is split, it follows from Roy's cancellation theorem \cite[Th.~8.1]{Roy} that $\CF^\perp/\CF$ is split.  Hence there exist isotropic elements $c,d \in \CF^\perp$ such that $b_1,\dotsc,b_{n-1},c,d$ form a free basis of $\CF^\perp$, and
\begin{equation*}
   \spann_R\{b_1,\dotsc,b_{n-1},c\} \quad\text{and}\quad \spann_R\{b_1,\dotsc,b_{n-1},d\}
\end{equation*}
are the (only) two Lagrangian submodules of $\Lambda_m \otimes_{O_{F_0}} R$ containing $\CF$.  Since the intersection of these Lagrangians has odd corank in each of them, they are classified by points on different connected components of $\OGr(n,\Lambda_m \otimes_{O_{F_0}} O_F)$, cf.~\cite[\S7.1.4]{PR}.  Hence exactly one of them, say the first, satisfies the spin condition, and we define this to be $\CF_m$.  Note that this already proves the uniqueness claim for $\CF_m$ in the statement of the lemma.

To show that $\CF_m$ contains $T_{m-1}(\CF_{m-1})$, first note that since $\CF_{m-1}$ is spanned by $a_1,\dotsc,a_n$, $T_{m-1}(\CF_{m-1})$ is generated by
\[
	b_1,\dotsc,b_{n-1}, xc+yd
\]
for some $x,y \in R$. Let
\begin{align*}
   w &:= b_1 \wedge \dotsb \wedge b_{n-1} \wedge (xc +yd)\\
	 &\phantom{:}= x \cdot b_1 \wedge \dotsb \wedge b_{n-1} \wedge c + y \cdot b_1 \wedge \dotsb \wedge b_{n-1} \wedge d \in \tensor[^n]{\Lambda}{_m} \otimes_{O_F} R.
\end{align*}
Consider the commutative diagram
\[
   \xymatrix@R-3ex{
	   \tensor[^n]{(\Lambda_{m-1})}{_{(-1)^r}} \otimes_{O_F} R \ar[r] \ar[dd]_-\sim 
		   &  \tensor[^n]{(\Lambda_m)}{_{(-1)^r}} \otimes_{O_F} R \ar[dd]^-\sim \\
		\\
		L_{m-1,(-1)^r}(R) \ar@{}[d]|-*{\cap}  
		   &  L_{m,(-1)^r}(R) \ar@{}[d]|-*{\cap} \\
		\tensor[^n]{\Lambda}{_{m-1}} \otimes_{O_F} R \ar[r]^-{\bigwedge^n T_{m-1}} 
		   &  \tensor[^n]{\Lambda}{_m} \otimes_{O_F} R,
	}
\]
where we recall the modules $\tensor[^n]{({\Lambda_i})}{_{\pm 1}}$ and $L_{i,\pm 1}(R)$ from \eqref{^nLambda_epsilon}.  Then $w$ generates the image of $\bigwedge^n \CF_{m-1}$ under the bottom map.  Since $\bigwedge^n \CF_{m-1}$ is contained in $L_{m-1,(-1)^r}(R)$ by the spin condition, the diagram shows that $w \in L_{m,(-1)^r}(R)$.  Now, in the special case $i = m$, it is easy to verify from the definitions in \cite[\S7]{PR} or \cite[\S2.3]{S} that there is a direct sum decomposition
\[
   \tensor[^n]{{\Lambda_m}}{} = \tensor[^n]{({\Lambda_m})}{_1} \oplus \tensor[^n]{({\Lambda_m})}{_{-1}},
\]
and hence
\[
   \tensor[^n]{{\Lambda_m}}{} \otimes_{O_F} R = L_{m,1}(R) \oplus L_{m,-1}(R).
\]
Since $w$ lies in $L_{m,(-1)^r}(R)$, its component in the summand $L_{m,(-1)^{r+1}}(R)$ is zero.  Hence $y = 0$ and $T_{m-1}(\CF_{m-1}) \subset \CF_m$.

To complete the proof, it remains to show that $T_m(\CF_m) \subset \CF_{m+1}$ and that $\CF_m$ is $\pi \otimes 1$-stable.  The first of these is an easy consequence of the adjunction relations \eqref{adjunction} and conditions \eqref{LM periodic cond} and \eqref{LM perp cond} for $\CF_{m-1}$ and $\CF_{m+1}$.  For the second, note that the action of $\pi \otimes 1$ on $\Lambda_m \otimes_{O_{F_0}} R$ identifies with the composite
\[
   \xymatrix{
      \Lambda_m \otimes_{O_{F_0}} R \ar[r]^-{T_m}  
		   &  \Lambda_{m+1} \otimes_{O_{F_0}} R \ar[d]_-{\sim}^-{\pi \otimes 1} \\
	   &  \Lambda_{-m+1} \otimes_{O_{F_0}} R \ar[rrr]^-{T_{m-2}\circ \dotsb \circ T_{-m+1}}
		   & & &   \Lambda_{m-1} \otimes_{O_{F_0}} R \ar[r]^-{T_{m-1}} 
			&  \Lambda_m \otimes_{O_{F_0}} R.
	}
\]
Thus the containment $(\pi \otimes 1)(\CF_m) \subset \CF_m$ follows from the containment $T_m(\CF_m) \subset \CF_{m+1}$, from conditions \eqref{LM functoriality cond} and \eqref{LM periodic cond} on $C_{\{m-1\}}$, and from the containment $T_{m-1}(\CF_{m-1}) \subset \CF_m$.
\end{proof}

In the notation of the lemma, we now define
\[
   \nu\colon
   \xymatrix@R=0ex{
	   C_{\{m-1\}} \ar[r]  &  C_{\{m\}}\\
		\CF_{m-1} \ar@{|->}[r]  &  \CF_m.
	}
\]
Then $\nu$ is plainly an isomorphism between generic fibers, and it restricts to an isomorphism between the generic fibers of the closed subschemes $M_{\{m-1\}}^\naive$ and $M_{\{m\}}^\naive$ in the source and target, respectively.  Hence, passing to flat closures, $\nu$ restricts to a map
\begin{equation}\label{nu}
   \nu\colon M_{\{m-1\}}^\loc \to M_{\{m\}}^\loc.
\end{equation}
We now have the following.

\begin{proposition}\label{lm isom}
Suppose that $m - 1 \in I$ and $m \notin I$, and fix the signature $(r,s)$.
\begin{altenumerate}
\item
The natural forgetful morphism
\[
   \xymatrix@R=0ex{
      M_{I \cup \{m\}}^\loc \ar[r]  &  M_I^\loc\\
	   \bigl((\CF_i)_{i\in I},\CF_m\bigr) \ar@{|->}[r]  &  (\CF_i)_{i\in I}
	}
\] 
is an isomorphism.
\item\label{lm isom ii}
The analogous forgetful morphism $M_{I\cup \{m\}} \to M_I$ is an isomorphism when $(r,s) = (n-1,1)$.
\item\label{lm isom iii}
If Conjecture \ref{str spin conj} holds true for signature $(r,s)$ and the set $I = \{m-1\}$, then the analogous forgetful morphism $M_{I\cup \{m\}} \to M_I$ is an isomorphism for any $I$ as above.
\end{altenumerate}
\end{proposition}

\begin{proof}
By Lemma \ref{LM mapping lem}, the inverse in all cases is given by $(\CF_i)_{i\in I} \mapsto ((\CF_i)_{i\in I},\nu(\CF_{m-1}))$.  Note that in \eqref{lm isom iii}, flatness of $M_{\{m-1\}}$ ensures that $\nu$ carries $M_{\{m-1\}}$ into $M_{\{m\}}$, and hence that this inverse carries $M_I$ into $M_{I\cup \{m\}}$.  Part (\ref{lm isom ii}) then follows from (\ref{lm isom iii}) in light of Yu's forthcoming result (cf.\ Remark \ref{Yu flatness}), but let us also sketch a direct proof of (\ref{lm isom ii}).  It similarly suffices to show that $\nu$ carries $M_{\{m-1\}}$ into $M_{\{m\}}$.

As in Lemma \ref{L stability}, consider the automorphism scheme $\uAut(\Lambda_{\{m-1,m\}})$ of the polarized $O_F$-lattice chain $\Lambda_{\{m-1,m\}}$ over $\Spec O_{F_0}$.  Then $\uAut(\Lambda_{\{m-1,m\}})_{O_F}$ acts on $M_{\{m-1\}}$ and $M_{\{m\}}$, and $\nu$ is equivariant for these actions.  The $\ov k$-point $\CF := (\pi \otimes 1)(\Lambda_{m-1} \otimes_{O_{F_0}}\ov k) \subset \Lambda_{m-1} \otimes_{O_{F_0}} \ov k$ on $M_{\{m-1\}}$ (which is the unique $\ov k$-point on which $\pi \otimes 1$ acts as zero) is in the closure of every orbit for $\uAut(\Lambda_{\{m-1,m\}})_{\ov k}$ on the geometric special fiber, and therefore we reduce to considering points $\CF_{m-1}$ on $M_{\{m-1\}}$ in an affine chart around $\CF$.  For $j \in \{m-1,m\}$, we take the ordered $O_{F_0}$-basis for $\Lambda_{j}$,
\begin{equation}\label{basis}
   \pi\i e_1,\dotsc, \pi\i e_j, e_{j+1},\dotsc,e_n, e_1,\dotsc,e_j, \pi e_{j+1},\dotsc, \pi e_n.
\end{equation}
With respect to this basis in the case $j = m-1$, we look at the affine chart on $M_{\{m-1\}}$ of points $\CF_{m-1}$ of the form
\begin{equation}\label{CF_m-1}
   \CF_{m-1} = \colspan
   \begin{bmatrix}
		X\\
		1_n
	\end{bmatrix},
	\quad
	X = (x_{ij}),
\end{equation}
where the blocks are of size $n \times n$.  (Thus $\CF$ is the $\ov k$-point $X = 0$.)  Write $X$ in the block form
\[
   X =
	\left[
	\begin{array}{ccc}
		A  &  *  &  B\\
		R_1  &  x_{mm}  &  R_2\\
		C  &  *  &  D
	\end{array}
	\right],
\]
where $A$ has size $(m-1) \times (m-1)$, $B$ has size $(m-1) \times m$, $C$ has size $m \times (m-1)$, $D$ has size $m \times m$, $R_1$ has size $1 \times (m-1)$, $R_2$ has size $1 \times m$, and the blocks marked $*$ are columns of the appropriate sizes.  We claim that, with respect to the basis \eqref{basis} for $\Lambda_m$, we have
\begin{equation}\label{nu(CF_m-1)}
   \renewcommand{\arraystretch}{1.1}
   \nu(\CF_{m-1}) = \colspan
	\left[
	\begin{array}{c|c|c}
		A  &  &  B\\
		0  & \smash{\raisebox{.5\normalbaselineskip}{$R_2^\ad$}}  &  0\\
		\hline
		  &  1  &\\
		\smash{\raisebox{.5\normalbaselineskip}{$C$}}  & - R_1^\ad  &  \smash{\raisebox{.5\normalbaselineskip}{$D$}}\\
		\hline
		1_{m-1} &  &\\
		R_1  &  &  R_2\\
		  &  &  1_m
	\end{array}
	\right],
\end{equation}
where, defining
\[
   \setlength{\arraycolsep}{2.5pt}
   H_i :=
	\begin{bmatrix}
		  &  &  1\\
		  &  \smash{\iddots}\\
		1
	\end{bmatrix}
\]
to be the unit antidiagonal matrix of size $i \times i$, we set
\[
   R_1^\ad := H_{m-1}\tensor*[^t]{R}{_1}
	\quad\text{and}\quad
	R_2^\ad := H_m\tensor*[^t]{R}{_2}.
\] 
Indeed, in terms of our bases, the map $\Lambda_{m-1} \to \Lambda_m$ is given by the $2n \times 2n$ matrix
\[
   U :=
   \begin{bmatrix}
		1_{m-1}\\
		  &  &  &  \varpi\\
		  &  &  1_{n-1}\\
		  &  1\\
		  &  &  &  &  1_m
	\end{bmatrix},
\]
where the $\varpi$ is the $(m,n+m)$-entry, and the off-diagonal $1$ is the $(n+m,m)$-entry.  The first $m-1$ and last $m$ columns in \eqref{nu(CF_m-1)} are the image under multiplication by $U$ of the corresponding columns in \eqref{CF_m-1}; since these columns evidently span a direct summand of rank $n-1$, the proof of Lemma \ref{LM mapping lem} shows that $\nu(\CF_{m-1})$ is the unique Lagrangian subspace containing these columns and satisfying the spin condition.  To show that the right-hand side of \eqref{nu(CF_m-1)} satisfies these properties, first recall the symmetric form $\sform_m$ from Remark \ref{I=m}.  With respect to the basis \eqref{basis} for $\Lambda_m$, the matrix of $\sform_m$ is
\begin{equation}\label{pairing matrix}
   \begin{bmatrix}
		  &  &  H_m\\
		  &  -H_n\\
		H_m
	\end{bmatrix}.
\end{equation}
It follows that the $m$th column in the matrix in \eqref{nu(CF_m-1)} pairs to zero under $\sform_m$ with every column in this matrix.  Hence the right-hand side in \eqref{nu(CF_m-1)} is Lagrangian (recall from the proof of Lemma \ref{LM mapping lem} that the other $n-1$ columns in \eqref{nu(CF_m-1)}, being in the image of the point $\CF_{m-1}$ on $M_{\{m-1\}}$, automatically span a totally isotropic subspace).  Thus to prove the equality asserted in \eqref{nu(CF_m-1)}, it remains to show that the right-hand side satisfies the spin condition.  By Remark \ref{I=m}, both this and the fact that the right-hand side is a point on $M_{\{m\}}$ follow from the single claim that every column of the matrix in \eqref{nu(CF_m-1)} becomes a multiple of the $m$th one after applying the operator $\pi \otimes 1 - 1 \otimes \pi$ (note that the $m$th column obviously continues to span a rank one summand after applying this operator).

With respect to our basis for $\Lambda_m$, the operator $\pi \otimes 1 - 1 \otimes \pi$ is given by the matrix
\[
   \begin{bmatrix}
		-\pi 1_n  &  \varpi 1_n\\
		1_n  &  -\pi 1_n
	\end{bmatrix}.
\]
Writing the product of this with the matrix in \eqref{nu(CF_m-1)} in block form
$
   \begin{bmatrix}
		Y \\
		Z
	\end{bmatrix}
$
($n \times n$ blocks), we have
\[
   \renewcommand{\arraystretch}{1.1}
	Z =
	\left[
	\begin{array}{c|c|c}
		A - \pi 1_{m-1}  &  &  B \\
		-\pi R_1  & \smash{\raisebox{.5\normalbaselineskip}{$R_2^\ad$}}  &  -\pi R_2\\
		\hline
		  &  1  &\\
		\smash{\raisebox{.5\normalbaselineskip}{$C$}}  & - R_1^\ad  &  \smash{\raisebox{.5\normalbaselineskip}{$D- \pi 1_m$}}
	\end{array}
	\right]
	\quad\text{and}\quad
	Y = -\pi Z.
\]
Our problem is thus to show that every column in $Z$ is a multiple of its $m$th one.  In the case of the $j$th column for $j \neq m,m+1$, this means that we have to show that
\begin{equation}\label{want}
\begin{aligned}
	x_{ij} &= x_{m+1,j}x_{m,i^\vee},  &  &\quad i < m,\ i \neq j,\\
	x_{jj} - \pi &= x_{m+1,j} x_{m,j^\vee}  &  &\quad \text{if } j < m,\\
	-\pi x_{mj} &= x_{m+1,j} x_{m,m+1},\\
	x_{jj} - \pi &= -x_{m+1,j} x_{m,j^\vee},  &  &\quad \text{if } j > m+1,\\
	x_{ij} &= -x_{m+1,j}x_{m,i^\vee},  &  &\quad i > m+1,\ i \neq j,
\end{aligned}
\end{equation}
where for any $i$ we set
\[
   i^\vee := n+1-i.
\]
Similarly, in the case of the $(m+1)$th column in $Z$, we have to show that
\begin{equation}\label{m+1}
\begin{aligned}
   x_{i,m+1} &= (x_{m+1,m+1}-\pi)x_{m,i^\vee}  &  &\quad i < m,\\
	-\pi x_{m,m+1} &= (x_{m+1,m+1}-\pi)x_{m,m+1},\\
	x_{i,m+1} &= -(x_{m+1,m+1}-\pi)x_{m,i^\vee},  &  &\quad i > m + 1.
\end{aligned}
\end{equation}
We are going to deduce all of these relations from the fact that the columns in \eqref{nu(CF_m-1)} are pairwise orthogonal under $\sform_m$ and from condition (\ref{LM str spin cond}) in the definition of $M_{\{m-1\}}$.

Let us first make explicit the orthogonality relations we need.  For $i$ and $j$ distinct from $m$ and $m+1$, pairing the $i$th and $j$th columns in (\ref{nu(CF_m-1)}) with respect to (\ref{pairing matrix}), we obtain
\begin{equation}\label{orthog rel}
\begin{aligned}
   0 &= x_{m+1,i}x_{mj} + x_{j^\vee i} + x_{i^\vee j} + x_{mi}x_{m+1,j},  &  &\quad i,j < m,\\
	0 &= x_{j^\vee i} - x_{m+1,i}x_{m,j} - x_{i^\vee,j} - x_{mi}x_{m+1,j}  ,  &  &\quad i < m,\ m+1 < j,\\
	0 &= x_{j^\vee i} - x_{m+1,i}x_{mj} - x_{mi}x_{m+1,j} + x_{i^\vee j},  &  &\quad m+1 < i,j.
\end{aligned}
\end{equation}
Similarly, for $j \neq m,m+1$, pairing the $j$th column with the $(m+1)$th column, we obtain
\begin{equation}\label{orthog 2}
\begin{aligned}
   0 &= x_{m+1,j}x_{m,m+1} + x_{j^\vee,m+1} + x_{mj}x_{m+1,m+1},  &  & \quad j<m,\\
	0 &= -x_{m+1,j}x_{m,m+1} - x_{mj}x_{m+1,m+1} + x_{j^\vee,m+1},  &  &\quad j > m+1;
\end{aligned}
\end{equation}
and pairing the $(m+1)$th column with itself (and using that $2$ is a unit), we obtain
\begin{equation}\label{orthog 3}
	0 = x_{m+1,m+1}x_{m,m+1}.
\end{equation}

Now we turn to condition \eqref{LM str spin cond} in the definition of $M_{\{m-1\}}$.  We continue to use the notation of \S\ref{odd ram max}.  We closely follow the analysis in \cite{S}.  Similarly to \S4 in loc.\ cit., we introduce the basis elements in $V$,
\[
   g_i := e_i \otimes 1 + \pi e_i \otimes \pi\i,
	\quad g_{n+i} := \frac{e_i \otimes 1 - \pi e_i \otimes \pi\i}2,
	\quad
	i = 1,\dotsc,n.
\]
Then $g_1,\dotsc,g_n$ is a basis for $V_\pi$, and $g_{n+1},\dotsc,g_{2n}$ is a basis for $V_{-\pi}$.  For $S = \{i_1 < \dots < i_n\} \subset \{1,\dotsc,2n\}$ of cardinality $n$, we define
\[
   g_S := g_{i_1} \wedge \dots \wedge g_{i_n} \in \tensor*[^n]{V}{}.
\]
Furthermore, we say that $S$ has \emph{type $(r,s)$} if $S \cap \{1,\dotsc,n\}$ has $r$ elements and $S \cap \{n+1,\dotsc,2n\}$ has $s$ elements.  As $S$ varies through the sets of type $(n-1,1)$, the elements $g_S$ form a basis of $\tensor*[^n]{V}{^{n-1,1}}$, and the elements $g_S - \sgn(\sigma_S)g_{S^\perp}$ span $\tensor*[^n]{V}{_{-1}^{n-1,1}}$, cf.\ \cite[Lem.\ 4.2]{S}.  Here we set
\[
   S^* := \{\,2n+1-i \mid i \in S\,\}
	\quad\text{and}\quad
	S^\perp := \{1,\dotsc,2n\} \smallsetminus S^*,
\]
and $\sigma_S$ is the permutation on $\{1,\dotsc,2n\}$ sending the elements $1,\dotsc,n$ onto $S$ in increasing order, and $n+1,\dotsc,2n$ onto the complement of $S$ in increasing order.  By \cite[Lem.\ 2.8]{S}, $\sgn(\sigma_S) = (-1)^{\Sigma S + n/2}$, where $\Sigma S$ denotes the sum of the elements in $S$.  If $S$ is of type $(n-1,1)$, then $S = \{1,\dotsc,\wh\jmath,\dotsc,n,n+i\}$ for unique integers $1 \leq i,j \leq n$, where the hat means that the entry is omitted; and, since $n$ is even,
\[
   \sgn(\sigma_S) = (-1)^{\frac{n(n+1)}2 -j + n+i + \frac n 2} = (-1)^{i+j}.
\]

Now let $T = \{i_1 < \dotsb < i_n \}\subset \{1,\dotsc,2n\}$ be of cardinality $n$. Regarding the (images of the) basis elements \eqref{basis} in the case $j = m-1$ as an ordered $O_F$-basis for $\Lambda_{m-1} \otimes_{O_{F_0}} O_F$, define 
\[
   e_T := (i_1\text{th basis element}) \wedge \dotsb \wedge (i_n\text{th basis element}) \in \tensor*[^n]{\Lambda}{_{m-1}} \subset \tensor*[^n]{V}{}. 
\]
For $1 \leq i<j \leq n$ with $i,j \notin \{m,m+1\}$, consider the set
\[
   T_0:= \bigl\{m,m+1,n+1,\dotsc,\wh{n+i},\dotsc,\wh{n+j},\dotsc,2n\bigr\}.
\]
In taking the wedge product of the columns of the matrix in \eqref{CF_m-1} and expressing this as a linear combination of $e_T$'s for varying $T$, the coefficient $a_{T_0}$ of $e_{T_0}$ is $\pm(x_{mi}x_{m+1,j}-x_{m+1,i}x_{mj})$.  On the other hand, for $S$ of type $(n-1,1)$, when $g_S$ is written as a linear combination of $e_T$'s, the only $e_T$'s that occur are for $T$ containing at most one pair of the form $k,n+k$.  Since $T_0$ contains two such pairs (for $k = m$ and $m+1$), condition \eqref{LM str spin cond} on $\CF_{m-1}$ implies that $a_{T_0}$ must be zero.  We conclude that
\begin{equation}\label{wedge 2}
   x_{mi}x_{m+1,j} = x_{m+1,i}x_{mj}
	\quad\text{for all}\quad
	i,j \notin\{m,m+1\}.
\end{equation}

Now let us now show that the first relation in \eqref{want} holds in the case that $j < m$.  Consider the set
\[
   S := \{1,\dotsc,\wh\jmath,\dotsc,n,n+i\}.
\]
Then
\[
   S^\perp = \bigl\{1,\dotsc, \wh{i^\vee},\dotsc,n,n+j^\vee\bigr\}.
\]
We recall the notion of ``worst term'' from \cite[Def.\ 4.8]{S}, which we take with respect to the $e_T$-basis for $\tensor[^n]{V}{}$. We denote the worst term of a vector $v$ by $\WT(v)$.  Using that
\begin{equation}\label{double wedge}
   g_i \wedge g_{n+i} = (\pi\i e_i \otimes \pi) \wedge (e_i \otimes 1)
\end{equation}
for any $i$, we have
\begin{multline*}
   \WT(g_S) = (e_1 \otimes 1) \wedge \dotsb \wedge \wh{(e_j \otimes 1)} \wedge \dotsb \wedge (e_{m-1} \otimes 1)\\ \wedge (\pi e_m \otimes \pi\i) \wedge \dotsb \wedge (\pi e_n \otimes \pi\i) \wedge (-\pi\i e_i \otimes \pi)
\end{multline*}
and
\begin{multline*}
   \WT(g_{S^\perp}) = (e_1 \otimes 1) \wedge \dotsb \wedge (e_{m-1} \otimes 1)\\ \wedge (\pi e_m \otimes \pi\i) \wedge \dotsb \wedge \wh{(\pi e_{i^\vee} \otimes \pi\i)} \wedge \dotsb \wedge (\pi e_n \otimes \pi\i) \wedge (e_{j^\vee} \otimes 1).
\end{multline*}
Let
\[
   T_0 := \bigl\{i, n+1,\dotsc, \wh{n+j},\dotsc,2n\bigr\}
	\quad\text{and}\quad
	T_0' := \bigl\{ j^\vee, n+1, \dotsc, \wh{n+i^\vee},\dotsc,2n\bigr\}.
\]
Then the $e_{T_0}$-term and the $e_{T_0'}$-term in the wedge of the columns of \eqref{CF_m-1} are, respectively,
\[
   (e_1 \otimes 1) \wedge \dotsb \wedge \underbrace{(\pi\i e_i \otimes x_{ij})}_{\makebox[0ex][c]{$\scriptstyle \text{in place of } e_j \otimes 1$}} \wedge \dotsb \wedge (e_{m-1} \otimes 1) \wedge (\pi e_m \otimes 1) \wedge \dotsb \wedge (\pi e_n \otimes 1)
\]
and
\[
   (e_1 \otimes 1) \wedge \dotsb \wedge (e_{m-1} \otimes 1) \wedge (\pi e_m \otimes 1) \wedge \dotsb \wedge \underbrace{(e_{j^\vee} \otimes x_{j^\vee i^\vee})}_{\makebox[0ex][c]{$\scriptstyle \text{in place of } \pi e_{i^\vee} \otimes 1$}} \wedge \dotsb \wedge (\pi e_n \otimes 1).
\]
The sum of these terms is
\begin{align*}
   &(-1)^{n-j+1} x_{ij}\pi^m \WT(g_S) + (-1)^{n-i^\vee} x_{j^\vee i^\vee} \pi^m \WT(g_{S^\perp})\\
	   &\qquad \qquad = (-1)^{j+1} \pi^m \bigl[x_{ij}\WT(g_S) - (-1)^{i+j+1} x_{j^\vee i^\vee} \WT(g_{S^\perp})\bigr] \\
		&\qquad \qquad \qquad \qquad = (-1)^{j+1} \pi^m \bigl[ x_{ij}\WT(g_S) - (-x_{j^\vee i^\vee}) \sgn(\sigma_S)\WT(g_{S^\perp}) \bigr].
\end{align*}
Since $S$ is the only set of type $(n-1,1)$ for which $g_S$, when expressed in terms of the $e_T$-basis, involves $e_{T_0}$, and ditto for $S^\perp$ and $e_{T_0'}$, we conclude from condition \eqref{LM str spin cond} on $\CF_{m-1}$ that
\[
   x_{ij} = -x_{j^\vee i^\vee}.
\]
Combining this with \eqref{wedge 2} and the second orthogonality relation in \eqref{orthog rel} (where $j$ is replaced by $j^\vee$), and using that $2$ is a unit, we obtain the first relation in \eqref{want}.  The case that $j > m$ is handled in a similar way, as is the last relation in \eqref{want}.

We next show that the third relation in \eqref{want} holds.  We first need to obtain an analog of \eqref{wedge 2} in the case that $i = m+1$.  Assume that $j < m$, and consider the sets
\begin{align*}
	T_0 &:= \bigl\{m,n+1,\dotsc,\dotsc,\wh{n+j},\dotsc,2n\bigr\},\\
   T_0' &:= \bigl\{m,m+1,n+1,\dotsc,\wh{n+j},\dotsc,\wh{n+m+1},\dotsc,2n\bigr\}.
\end{align*}
When we express the wedge product of the columns in \eqref{CF_m-1} as a linear combination of $e_T$'s, the $e_{T_0}$-term is
\begin{equation}\label{e_T_0 term}
   (-1)^{j-1}x_{mj}e_{T_0},
\end{equation}
and the $e_{T_0'}$-term is
\[
   (-1)^{m-j}\det
	\begin{bmatrix}
		x_{mj}  &  x_{m,m+1}\\ 
		x_{m+1,j}  &  x_{m+1,m+1}
	\end{bmatrix}
	e_{T_0'}.
\]
On the other hand, the only set $S$ of type $(n-1,1)$ for which $e_{T_0}$ and $e_{T_0'}$ occur when we express $g_S$ as a linear combinator of $e_T$'s is
\[
   S = \{1,\dotsc,\wh\jmath,\dotsc,n,n+m\}.
\]
Using \eqref{double wedge} in the case $i = m$, the worst term of $g_S$ is its $e_{T_0}$-term,
\begin{equation}\label{e_T_0 term 2}
\begin{aligned}
   \WT(g_S) &= (e_1 \otimes 1) \wedge \dotsb \wedge \wh{(e_j \otimes 1)} \wedge \dotsb \wedge (e_{m-1} \otimes 1)\\ 
	   &\qquad \qquad\wedge (\pi e_m \otimes \pi\i) \wedge \dotsb \wedge (\pi e_n \otimes \pi\i) \wedge (e_m \otimes 1)\\
		&= - \pi^{-(m+1)} e_{T_0}.
\end{aligned}
\end{equation}
The $e_{T_0'}$-term in $g_S$ is
\begin{align*}
   &(e_1 \otimes 1)\wedge \dotsb \wedge \wh{(e_j \otimes 1)} \wedge \dotsb \wedge (e_{m-1} \otimes 1) \wedge (\pi e_m \otimes \pi\i) \\
	&\qquad\qquad\wedge (e_{m+1} \otimes 1) \wedge (\pi e_{m+2} \otimes \pi\i) \wedge \dotsb \wedge (\pi e_n \otimes \pi\i) \wedge (e_m \otimes 1)\\
 &\qquad\qquad\qquad\qquad\qquad\qquad\qquad\qquad\qquad\qquad\qquad\qquad\qquad = (-1)^{m}\pi^{-m}e_{T_0'}.
\end{align*}
Condition \eqref{LM str spin cond} on $\CF_{m-1}$ then implies that
\[
   \det
	\begin{bmatrix}
		x_{mj}  &  x_{m,m+1}\\ 
		x_{m+1,j}  &  x_{m+1,m+1}
	\end{bmatrix}
	= \pi x_{mj}.
\]
If $j > m+1$, then one similarly finds that the same relation holds.  We conclude that
\begin{equation}\label{wedge 2'}
   x_{mj}x_{m+1,m+1} - x_{m+1,j}x_{m,m+1} = \pi x_{mj}, \quad j \neq m,m+1.
\end{equation}

To apply this to the third relation in \eqref{want}, suppose that $j < m$.  Combining \eqref{wedge 2'} and \eqref{orthog 2} (and again using that $2$ is a unit), it suffices to show that
\[
   \pi x_{mj} = x_{j^\vee,m+1}.
\]
We again consider the set $S = \{1,\dotsc,\wh\jmath,\dotsc,n,n+m\}$, and also its perp
\[
   S^\perp = \bigl\{ 1,\dotsc,\wh{m+1},\dotsc,n,n+j^\vee\bigr\}.
\]
Let
\[
   T_0'' := \bigl\{j^\vee,n+1,\dotsc, \wh{n+m+1},\dotsc,2n\bigr\}.
\]
The $e_{T_0''}$-term that occurs in the wedge product of the columns of \eqref{CF_m-1} is
\begin{equation}\label{e_T_0'' term}
   (-1)^m x_{j^\vee,m+1}e_{T_0''}.
\end{equation}
On the other hand, $S^\perp$ is the only set of type $(n-1,1)$ for which $g_{S^\perp}$ involves $e_{T_0''}$, and the $e_{T_0''}$-term in $g_S - \sgn(\sigma_S)g_{S^\perp} = g_S - (-1)^{j+m}g_{S^\perp}$ is then
\begin{align*}
   &(-1)^{j+m+1}(e_1 \otimes 1) \wedge \dotsb \wedge (e_{m-1} \otimes 1)\\ 
	   &\qquad \qquad \wedge (\pi e_m \otimes \pi\i) \wedge (\pi e_{m+2} \otimes \pi\i) \wedge \dotsb \wedge (\pi e_n \otimes \pi\i) \wedge (e_{j^\vee} \otimes 1)\\
		&\qquad\qquad\qquad\qquad\qquad\qquad\qquad\qquad\qquad\qquad\qquad\qquad\qquad = (-1)^{j+m} \pi^{-m} e_{T_0''}.
\end{align*}
Since $\CF_{m-1}$ satisfies condition \eqref{LM str spin cond}, it follows from this, \eqref{e_T_0 term}, \eqref{e_T_0 term 2}, and \eqref{e_T_0'' term} that $\pi x_{mj} = x_{j^\vee,m+1}$, as desired.
A similar argument shows that the third relation in \eqref{want} holds when $j > m+1$.

We next show that the second and fourth relations in \eqref{want} hold.  Assume $j < m$.  By the second orthogonality relation in \eqref{orthog rel} (with $j$ in place of $i$ and $j^\vee$ in place of $j$) and by \eqref{wedge 2} (with $j^\vee$ in place of $i$), both of the desired relations follow (again using that $2$ is a unit) from showing that
\[
   x_{jj} - \pi = \pi - x_{j^\vee j^\vee}.
\]
For any $1 \leq i \leq n$, define
\[
   S_i := \{1,\dotsc,\wh\imath,\dotsc,n, n+i\}.
\]
Then $S_i^\perp = S_{i^\vee}$ and $\sgn(\sigma_S) = 1$.
By a similar calculation to the one in \cite[Lem.\ 4.11]{S}, for $i \leq m$,
\begin{align*}
   g_{S_i} - \sgn(\sigma_S)g_{S_i^\perp} 
	   &= g_{S_i} - g_{S_{i^\vee}}\\
	   &= (-1)^i g_1 \wedge \dotsb \wedge \wh g_i \wedge \dotsb \wedge \wh g_{i^\vee} \wedge \dotsb \wedge g_n\\
		&\qquad\qquad \wedge [(e_i \otimes 1) \wedge (e_{i^\vee} \otimes 1) - (\pi\i e_i \otimes 1) \wedge (\pi e_{i^\vee} \otimes 1)].
\end{align*}
It follows that $S_m$ and $S_{m+1}$ are the only sets $S$ of type $(n-1,1)$ for which $g_S - \sgn(\sigma_S)$ involves $e_{\{n+1,\dotsc,2n\}}$; and that for
\[
   T_j := \bigl\{j,n+1,\dotsc,\wh{n+j},\dotsc,2n\bigr\}
	\quad\text{and}\quad
	T_{j^\vee} := \bigl\{j^\vee,n+1,\dotsc,\wh{n+j^\vee},\dotsc,2n\bigr\},
\]
the sets $S_j,S_m,S_{m+1},S_{j^\vee}$ are the only sets $S$ of type $(n-1,1)$ for which $g_S - \sgn(\sigma_S)g_{S^\perp}$ involves $e_{T_j}$ (and ditto for $e_{T_{j^\vee}}$).  Furthermore, essentially the same argument as in the proof of \cite[Prop.\ 4.12]{S} (especially the last two paragraphs) shows that for $a_i \in F$,
\[
   \sum_{i = 1}^m a_i(g_{S_i}- g_{S_{i^\vee}}) \in \tensor*[^n]{\Lambda}{_{m-1}}
	\iff
	\ord_\pi(a_i) \geq m \text{ for all $i$, and $\ord_\pi(a_m) \geq m+1$}.
\]
Now note that the worst term in $g_{S_m}-g_{S_{m+1}}$ is its $e_{\{n+1,\dotsc,2n\}}$-term,
\begin{equation}\label{1}
   \WT(g_{S_m}-g_{S_{m+1}}) = (-1)^{m+1}\pi^{-(m+1)}e_{\{n+1,\dotsc,2n\}},
\end{equation}
and that the $e_{T_j}$- and $e_{T_{j^\vee}}$-terms in $g_{S_m}-g_{S_{m+1}}$ are, respectively,
\begin{equation}\label{2}
   (-1)^{m+1}\pi^{-m} (e\otimes 1) \wedge \dotsb \wedge \underbrace{(\pi\i e_j \otimes 1)}_{\makebox[0ex][c]{$\scriptstyle \text{in place of } e_j \otimes 1$}} \wedge \dotsb \wedge (e_{m-1}\otimes 1) \wedge (\pi e_m \otimes 1) \wedge \dotsb \wedge (\pi e_n \otimes 1)
\end{equation}
and
\begin{equation}\label{3}
	(-1)^{m+1}\pi^{-m} (e\otimes 1) \wedge \dotsb \wedge (e_{m-1}\otimes 1) \wedge (\pi e_m \otimes 1) \wedge \dotsb \wedge \underbrace{(e_{j^\vee} \otimes 1)}_{\makebox[0ex][c]{$\scriptstyle \text{in place of } \pi e_{j^\vee} \otimes 1$}} \wedge \dotsb \wedge (\pi e_n \otimes 1).
\end{equation}
Similarly, the $e_{T_j}$- and $e_{T_{j^\vee}}$-terms in $g_{S_j}-g_{S_{j^\vee}}$ are, respectively,
\begin{equation}\label{4}
   (-1)^{i+1}\pi^{-m} (e\otimes 1) \wedge \dotsb \wedge \underbrace{(\pi\i e_j \otimes 1)}_{\makebox[0ex][c]{$\scriptstyle \text{in place of } e_j \otimes 1$}} \wedge \dotsb \wedge (e_{m-1}\otimes 1) \wedge (\pi e_m \otimes 1) \wedge \dotsb \wedge (\pi e_n \otimes 1)
\end{equation}
and
\begin{equation}\label{5}
	(-1)^i\pi^{-m} (e\otimes 1) \wedge \dotsb \wedge (e_{m-1}\otimes 1) \wedge (\pi e_m \otimes 1) \wedge \dotsb \wedge \underbrace{(e_{j^\vee} \otimes 1)}_{\makebox[0ex][c]{$\scriptstyle \text{in place of } \pi e_{j^\vee} \otimes 1$}} \wedge \dotsb \wedge (\pi e_n \otimes 1).
\end{equation}
Note that the coefficients in (\ref{2}) and (\ref{3}) are $\pi$ times the coefficient in (\ref{1}), and that there is a relative sign of $-1$ between the coefficients in (\ref{4}) and (\ref{5}).  When we express the wedge product of the columns of (\ref{CF_m-1}) as a linear combination of $e_T$'s, since the coefficient of $e_{\{n+1,\dotsc,2n\}}$ is evidently $1$, and since $\CF_{m-1}$ satisfies condition \eqref{LM str spin cond}, it follows that
\[
   x_{jj} = \pi + c
	\quad\text{and}\quad x_{j^\vee j^\vee} = \pi - c
\]
for some scalar $c$.  Hence $x_{jj} - \pi = \pi - x_{j^\vee j^\vee}$, as desired.

To complete the proof, it remains to show that the relations \eqref{m+1} hold.  This is now very easy: the first and third relations are obvious from the third relation in \eqref{want} and from \eqref{orthog 2} (both with $i^\vee$ in place of $j$), and the second relation is obvious from the orthogonality relation \eqref{orthog 3}.
\end{proof}

\begin{remark}[$n=2$]\label{forgetful LM isom n=2}
Proposition \ref{lm isom} is a generalization of (part of) \cite[Rem. 2.35]{PRS}, which contains the case $n = 2$.
\end{remark}

\begin{remark}\label{not flat}
In the case that $(r,s) = (n-1,1)$, we remark that if one weakens the moduli problem defining $M_{\{m-1\}}$ by substituting the spin condition \eqref{LM spin cond} for condition \eqref{LM str spin cond}, then the resulting scheme can fail to be carried by $\nu$ into $M_{\{m\}}$, and hence can fail to be flat.  To give a simple example, take $R := O_F/\varpi O_F$, and let $\CF_{m-1} := (\pi \otimes 1)(\Lambda_{m-1} \otimes_{O_{F_0}} R) \subset \Lambda_{m-1} \otimes_{O_{F_0}} R$.  Then $\pi \otimes 1$ acts as $0$ on $\CF_{m-1}$.  If the characteristic of $k$ divides $n - 2$, then $\CF_{m-1}$ satisfies the Kottwitz condition, and it is easy to see that $\CF_{m-1}$ furthermore gives a point in $M_{\{m-1\}}^\naive(R)$ satisfying the wedge and spin conditions.  However, the corresponding point $\nu(\CF_{m-1}) \in C_{\{m\}}(R)$ is 
given by
\[
   \nu(\CF_{m-1}) = \spann_R\{e_1\otimes 1,\dotsc,e_{m-1}\otimes 1,e_{m+1}\otimes 1, \pi e_{m+1}\otimes 1,\dotsc, \pi e_n\otimes 1\},
\]
which does not satisfy the wedge condition.  This shows that for even $n = 2m$ and signature $(r,s) = (n-1,1)$, the statement of \cite[Conj.~7.3]{PR} can fail for $I = \{m-1\}$.  Now, strictly speaking, such an $I$ is not allowed in this conjecture (cf.~\S\S1.4, 7.2.1 in loc.~cit.), so this does not give an honest counterexample.  It would be interesting to see if this example sheds any light on the validity of this conjecture for those $I$ which are allowed.  (It is known that this conjecture can fail for odd $n$ and signature $(n-1,1)$, cf.~\cite{S}.)
\end{remark}

\part{The conjectures}\label{conjectures}

In this part of the paper, we formulate various arithmetic transfer conjectures arising from morphisms between the spaces we have defined previously.  Except where noted to the contrary, we denote by $S = S_n$ the scheme over $\Spec F_0$ defined in \eqref{S_n}.

\section{Arithmetic transfer conjecture, $F/F_0$ unramified, almost self-dual type}\label{s:ATunr}
In this section $F/F_0$ is unramified and $n\geq 2$, as in \S\ref{unram non-max}.  We take the special vectors $u_i\in W_i$ (cf.~\S\ref{setup homog}) to satisfy $(u_i, u_i)=\varpi$.  Recall the $p$-divisible group $\ov\CE$ from the Introduction, with its $O_F$-action $\iota_{\ov\CE}$, principal polarization $\lambda_{\ov\CE}$, and framing isomorphism $\rho_{\ov\CE}$.  Let
\[
   \ov\CE' := \ov{{\CE}},
	\quad
	\iota_{\ov\CE'} := \iota_{\ov\CE},
	\quad
	\lambda_{\ov\CE'} := \varpi\lambda_{\ov\CE},
	\quad\text{and}\quad
	\rho_{\ov\CE'} := \rho_{\ov\CE}\colon \ov\CE'_{\ov k} \isoarrow \ov\BE',
\]
where $\ov\BE'$ is defined in \eqref{ovBE'}.  Then $\rho_{\ov\CE'}$ is $O_F$-linear, and $\rho_{\ov\CE'}^*(\lambda_{\ov\BE'}) = (\lambda_{\ov\CE'})_{\ov k}$.  We define a closed embedding of formal schemes
\begin{equation*}
   \wt\delta_\CN \colon
	\xymatrix@R=0ex{
	   \CN_{n-1} \ar[r]  &  \wt\CN_n\\
		(X, \iota, \lambda, \rho) \ar@{|->}[r]
		   &  \bigl(X \times \ov\CE', \iota \times \iota_{\ov\CE'},
                 \lambda \times \lambda_{\ov\CE'}, \rho \times \rho_{\ov\CE'} \bigr) .
   }
\end{equation*}
Here the last entry is a framing to the constant object in the special fiber defined by $\wt\BX_n = \BX_{n-1} \times \ov\BE'$, cf.~\eqref{wtBX_n unram}.   We therefore obtain as in the case of the AFL a closed embedding
\begin{equation*}
   \wt\Delta_\CN \colon \CN_{n-1} \xra{(\id_{\CN_{n-1}},\wt\delta_\CN)} \wt\CN_{n-1,n} = \CN_{n-1}\times_{\Spf O_{\breve F}}\wt\CN_n,
\end{equation*}
whose image we denote by
\[
   \wt\Delta := \wt\Delta_\CN (\CN_{n-1}).
\]
Note that the canonical vector
\[
   u := (0,\id_{\ov\BE}) \in \BV\bigl(\wt\BX_n\bigr) = \BV(\BX_{n-1}) \oplus \BV\bigl(\ov\BE'\bigr)
\]
has norm $\varpi$.  Since $\BV(\BX_{n-1})$ is non-split, it follows from \eqref{chi decomp formula} that $\BV(\wt\BX_n)$ is split.  Thus we may identify the triple $(\BV(\wt\BX_n), u, \BV(\BX_{n-1}))$ with $(W_0,u_0,W_0^\flat)$. In this way $H_0(F_0)$ acts on $\CN_{n-1}$, $G_{W_0}$ acts on $\wt\CN_n$, and $\wt\Delta_\CN$ is equivariant for the embedding $H_0(F_0) \hookrightarrow G_{W_0}(F_0)$.

\begin{remark}
We note that the image of $\wt\delta_\CN$ is contained in the smooth locus of $\widetilde{\CN}_n$. Indeed, via the usual local model diagram argument, it suffices to prove the analogous claim for the corresponding map on local models.  Identifying the local model for $\wt\CN_n$ with the standard local model for $\GL_n$ as in the proof of Theorem \ref{wtCN_n semistable}, points  $(\CF_i' \subset (\Lambda_i \otimes_{O_{F_0}} \CO_S)')_{i=0,1}$ which lie in the image of this map have the property that the induced map $(\Lambda_0 \otimes_{O_{F_0}} \CO_S)'/\CF_0'\to (\Lambda_1\otimes_{O_{F_0}} \CO_S)'/\CF_1'$ is an isomorphism.  Such points are contained in the smooth locus, which is easy to deduce from the explicit description in \cite[\S4.4.5]{Goertz} (in the case $\kappa = 1$ and $r = n-1$).
\end{remark}

\begin{remark}\label{diag cycle}
As in the case of the formal scheme $\CN_n$ in Remark \ref{reltospecialunr}, one can define a special cycle $\wt\CZ(u)$ in $\wt\CN_n$ as the locus where the quasi-homomorphism $u\colon \ov\BE\to \BX_n$ lifts to a homomorphism from $\ov\CE$ to the universal object over $\wt\CN_n$.  When $n=2$, it is known that $\wt\CZ(u)$ is a relative divisor and that $\wt\delta_\CN$ induces an isomorphism $\CN_1\simeq \wt\CZ(u)$, cf.\ \cite[Th.~3.14]{San}. In this case
\[
   (1\times g)\wt\Delta = \wt\CZ(gu) , \quad g\in \U\bigl(\BV(\BX_2)\bigr) .
\]
Hence for $n=2$,  the intersection number $\Int(g)$ appearing in Conjecture \ref{conjunramified} below is  a special case of the intersection number of two \emph{unitary special divisors}, in the terminology of Sankaran.
\end{remark}

Now fix a self-dual lattice
\[
   \Lambda_1^\flat \subset W_1^\flat,
\]
which exists (and is unique up to $H_1(F_0)$-conjugacy) since $W_1^\flat$ is split by \eqref{chi decomp formula}. Let
\[
   \Lambda_1 := \Lambda_1^\flat \oplus O_F u_1 \subset W_1.
\]
Then $\Lambda_1$ is an almost self-dual $O_F$-lattice, i.e.~$\Lambda_1\subset^1 \Lambda_1^\vee\subset \varpi^{-1}{\Lambda}_1$.  Let
\[
   K_1^\flat \subset H_1(F_0)
\]
denote the stabilizer of $\Lambda_1^\flat$, and let
\[
	K_1 \subset G_1(F_0)
	\quad\text{and}\quad
	\fkk_1 \subset \fkg_1(F_0)
\]
denote the respective stabilizers of $\Lambda_1$.  Then $K_1^\flat$ and $K_1$ are both maximal parahoric subgroups. We normalize the Haar measure on $H_1(F_0)$ (and hence the product measure on $H_{W_1}(F_0) = H_1(F_0) \times H_1(F_0)$) by assigning $K_1^\flat$ volume $1$. We normalize the Haar measures on $H'_1 ({F_0})$ and $H'_2 ({F_0})$ as  in \S\ref{s:FL}, and we take the transfer factors $\omega_{G'}$ on $G'(F_0)_\rs$, $\omega_S$ on $S(F_0)_\rs$, and $\omega_\fks$ on $\fkg(F_0)_\rs$ as in \eqref{sign}.

Before stating the ATC in the present situtation, we first formulate a fundamental lemma conjecture; this will, in turn, motivate an AFL conjecture as part of the ATC. On the general linear group side, let
\begin{align*}
   K_0(\varpi) := \biggl\{\,g\in \GL_n(O_F) \biggm| g\equiv 
	\begin{bmatrix}
		*  &  0  \\
      *  &  *
   \end{bmatrix}  \bmod \varpi \,\biggr\}
\end{align*}
(diagonal blocks of respective sizes $(n-1) \times (n-1)$ and $1 \times 1$), and define
\begin{align}\label{eqn K' S}
   K' := S(O_{F_0})\cap K_0(\varpi)\subset S(F_0).
\end{align}
We also make the Lie algebra versions of these definitions,
\begin{align}\label{eqn fkk0}
   \fkk_0(\varpi) :=
	   \biggl\{\, y\in \gl_n(O_F) \biggm| y\equiv 
		\begin{bmatrix}
		   *  &  0  \\
         *  &  *
      \end{bmatrix}  \bmod \varpi \,\biggr\}
\end{align}
(same block sizes as before) and
\begin{align}\label{eqn fkk'}
   \fkk' := \fks(O_{F_0})\cap \fkk_0(\varpi).
\end{align}

\begin{conjecture}[Fundamental lemma]\label{FL almost self-dual}\hfill
\begin{altenumerate}
\renewcommand{\theenumi}{\alph{enumi}}
\item\label{FL almost self-dual homog}
\textup{(Homogeneous version)} The function $(-1)^{n-1}\frac{q^n-1}{q-1}\mathbf{1}_{\GL_{n-1}(O_F)\times K_0(\varpi)} \in C_c^\infty(G')$ transfers to the pair of functions $(0,\mathbf{1}_{K_1^\flat\times K_1}) \in C_c^\infty(G_{W_0}) \times C_c^\infty(G_{W_1})$. 
\item\label{FL almost self-dual inhomog}
\textup{(Inhomogeneous version)} The function $(-1)^{n-1}\mathbf{1}_{K'} \in C_c^\infty(S)$ transfers to the pair of functions $(0,\mathbf{1}_{K_1}) \in C_c^\infty(G_0)\times C_c^\infty(G_1)$.
\item\label{FL almost self-dual lie}
\textup{(Lie algebra version)} The function $(-1)^{n-1}\mathbf{1}_{\fkk'} \in  C_c^\infty(\fks)$ transfers to the pair of functions $(0,\mathbf{1}_{\fkk_1}) \in C_c^\infty(\fkg_0)\times C_c^\infty(\fkg_1)$.
\end{altenumerate}
\end{conjecture}

We will show in Theorem \ref{unram variant FL} below that part (\ref{FL almost self-dual lie}) of this conjecture is equivalent to the Lie algebra FL Conjecture \ref{FLconj}(\ref{FLconj lie}), and that parts (\ref{FL almost self-dual homog}) and (\ref{FL almost self-dual inhomog}) follow from these when $q \geq n$.

We now come to the ATC, as well as an AFL conjecture for the functions appearing on the general linear group side of Conjecture \ref{FL almost self-dual}. In analogy with \eqref{defintprod}, for $g \in G_{W_0}(F_0)_\rs$ we set
\[
   \Int(g) := \bigl\la \wt\Delta, g \wt\Delta \bigr\ra _{\wt\CN_{n-1, n}},
\]
and for $g \in G_0(F_0)_\rs$ we set
\[
   \Int(g) := \bigl\la \wt\Delta, (1 \times g) \wt\Delta \bigr\ra _{\wt\CN_{n-1, n}}.
\]
In the Lie algebra setting, for any quasi-endomorphism $x$ of $\wt\BX_n$, and in particular for $x \in \fkg_0(F_0) = \Lie \RU(\BV(\wt\BX_n))(F_0) \cong \{ x \in \End_{O_F}^\circ(\wt\BX_n) \mid x + \Ros_{\lambda_{\wt\BX_n}}(x) = 0 \}$, we define (abusing notation in the obvious way, as in \eqref{Delta_x})
\[
   \wt\Delta_x := \bigl\{\, (Y,X) \in \wt\CN_{n-1,n} \bigm| x\colon \wt\BX_n \rightarrow \wt\BX_n \text{ lifts to a homomorphism } Y \times \ov\CE' \rightarrow X \,\bigr\}.
\]
As in the case of \eqref{Delta_x}, when $g \in G_0(F_0)$ we have $\wt\Delta_g = (1 \times g)\wt\Delta$.
When $\wt\Delta \cap \wt\Delta_x$ is an artinian scheme, we define
\[
   \lInt(x) := \length\bigl(\wt\Delta \cap \wt\Delta_x\bigr).
\]

\begin{conjecture}[Arithmetic transfer conjecture and arithmetic fundamental lemma, almost self-dual case]\label{conjunramified} \hfill
\begin{altenumerate}
\renewcommand{\theenumi}{\alph{enumi}}
\item\label{conjunramified homog}
\textup{(Homogeneous version)}
Let $f' \in C^{\infty}_c (G')$ be a function transferring to the pair $(0, \mathbf{1}_{K_1^\flat \times K_1}) \in C_c^\infty(G_{W_0}) \times C_c^\infty(G_{W_1})$.  Then there exists a function $f'_\corr \in C^{\infty}_c (G')$ such that, for any $\gamma \in G'(F_0)_\rs$ matching an element $g \in G_{W_0}(F_0)_\rs$,
\begin{equation*}
   \omega_{G'} (\gamma)\del(\gamma, f') = - 2\Int(g)\cdot \log q + \omega (\gamma)\Orb(\gamma, f'_\corr).
\end{equation*}
Furthermore, for all such matching $\gamma$ and $g$, there is an AFL identity
\begin{equation*}
   \omega_{G'} (\gamma)\del\biggl(\gamma, (-1)^{n-1}\frac{q^n-1}{q-1} \cdot \mathbf{1}_{\GL_{n-1}(O_F)\times K_0(\varpi)} \biggr) = - 2\Int(g)\cdot \log q.
\end{equation*}
\item\label{conjunramified inhomog}
\textup{(Inhomogeneous version)}
Let $f' \in C^{\infty}_c(S)$ be a function transferring to the pair $(0, \mathbf{1}_{K_1}) \in C_c^\infty(G_0) \times C_c^\infty(G_1)$. Then there exists a function $f'_\corr \in C^{\infty}_c (S)$ such that, for any $\gamma \in S(F_0)_\rs$ matching an element $g \in G_0(F_0)_\rs$,
\begin{equation*}
   \omega_S (\gamma)\del(\gamma, f') = - \Int(g)\cdot \log q + \omega (\gamma)\Orb(\gamma, f'_\corr).
\end{equation*}
Furthermore, for all such matching $\gamma$ and $g$, there is an AFL identity
\begin{equation*}
   \omega_S (\gamma)\del\bigl(\gamma, (-1)^{n-1}\mathbf{1}_{K'}\bigr) = - \Int(g)\cdot \log q .
\end{equation*}
\item\label{conjunramified lie}
\textup{(Lie algebra version)}
Let $\phi' \in C^{\infty}_c (\fks)$ be a function transferring to the pair $(0, \mathbf{1}_{\fkk_1}) \in C_c^\infty(\fkg_0) \times C_c^\infty(\fkg_1)$. Then there exists a function $\phi'_\corr \in C^{\infty}_c (S)$ such that, for any $y \in \fks(F_0)_\rs$ matching an element $x \in \fkg_0(F_0)_\rs$ for which the intersection $\wt\Delta \cap \wt\Delta_x$ is an artinian scheme,
\begin{equation*}
   \omega_\fks (y)\del(y, \phi') = - \lInt(x)\cdot \log q + \omega (y)\Orb(y, \phi'_\corr).
\end{equation*}
Furthermore, for all such matching $y$ and $x$, there is an AFL identity
\begin{equation*}
   \omega_\fks (y)\del\bigl(y, (-1)^{n-1}\mathbf{1}_{\fkk'}\bigr) = - \lInt(x)\cdot \log q .
\end{equation*}
\end{altenumerate}
\end{conjecture}

In \S\ref{results atc afl} we will show that Conjecture \ref{conjunramified} reduces to the AFL in the self-dual case (i.e.\ Conjecture \ref{AFLconj}) in the case of a nondegenerate intersection.  In particular, this will show that the conjecture holds when $n = 2$ or $n = 3$.

\section{Arithmetic transfer conjecture, $F/F_0$ ramified,  $n$ odd}\label{s:ATCodd}
In this section $F/F_0$ is ramified, $n \geq 3$ is odd, and we take the special vectors $u_i\in W_i$ to have norm $1$.  This case is the subject of \cite{RSZ}, to which we refer the reader for more details.  We continue to use the $p$-divisible group $\ov\CE$ and its attendant structure from the Introduction.  Recall from \S\ref{odd ram max} that we take $\CN_n=\CN^{(1)}_n$, i.e.~we require that the hermitian space $\BV(\BX_n)$ attached to the framing object $\BX_n = \BX_n^{(1)}$ is non-split. As in the case of the AFL, taking the product with $\ov\CE$ defines a closed embedding of formal schemes
\begin{equation*}
   \delta_\CN\colon
   \xymatrix@R=0ex{
	   \CN_{n-1} \ar[r]  &  {\CN}_n\\
		(X, \iota, \lambda, \rho) \ar@{|->}[r]  
		   &  \bigl(X \times \ov\CE, \iota \times \iota_{\ov\CE}, \lambda \times \lambda_{\ov\CE}, \rho \times \rho_{\ov\CE}\bigr).
	}
\end{equation*}
Here the last entry is a framing to the constant object over the special fiber defined by $\BX_n = \BX_{n-1} \times \ov{\BE}$.  We obtain as before a closed embedding
\begin{equation*}
   \Delta_\CN\colon \CN_{n-1} \xra{(\id_{\CN_{n-1}},\delta_\CN)} \CN_{n-1,n} = \CN_{n-1} \times_{\Spf O_{\breve F}} \CN_n,
\end{equation*}
whose image we again denote by
\[
   \Delta := \Delta_\CN(\CN_{n-1}).
\]
Let
\[
   u := (0,\id_{\ov\BE}) \in \BV(\BX_n) = \BV(\BX_{n-1}) \oplus \BV\bigr(\ov\BE\bigl),
\]
which is a vector of norm $1$.  Since $\BV(\BX_n)$ is non-split, we may therefore identify the triple $(\BV(\BX_n), u, \BV(\BX_{n-1}))$ with $(W_1,u_1,W_1^\flat)$.  In this way $\Delta_\CN$ is equivariant for the embedding $H_1(F_0) \hookrightarrow G_{W_1}(F_0)$ as before.

Now recall that in the ramified case, an $F/F_0$-hermitian space contains a $\pi$-modular lattice $\Lambda$ (i.e.~$\Lambda^\vee = \pi^{-1}\Lambda$) if and only if the space is split of even dimension (and all $\pi$-modular lattices are then conjugate under the unitary group). By \eqref{chi decomp formula} this is the case for $W_0^\flat$, and therefore we may fix a $\pi$-modular lattice
\[
   \Lambda_0^\flat \subset W_0^\flat.
\]
Let
\[
   \Lambda_0 := \Lambda_0^\flat \oplus O_Fu_0 \subset W_0.
\]
Then $\Lambda_0$ is almost $\pi$-modular, i.e.~$\Lambda_0 \subset \Lambda_0^\vee \subset^1 \pi^{-1}\Lambda_0$.  Let
\[
   K_0^\flat \subset H_0(F_0)
\]
denote the stabilizer of $\Lambda_0^\flat$, and let
\[
	K_0 \subset G_0(F_0)
	\quad\text{and}\quad
	\fkk_0 \subset \fkg_0(F_0)
\]
denote the respective stabilizers of $\Lambda_0$. Then $K_0^\flat $ is a special maximal parahoric subgroup, and $K_0$ contains a special maximal parahoric subgroup with index $2$, cf.~\cite[\S4.a]{PR-TLG}. We normalize the Haar measure on $H_0(F_0)$ (and hence the product measure on $H_{W_0}(F_0) = H_0(F_0) \times H_0(F_0)$) by assigning $K_0^\flat$ volume $1$. We normalize the Haar measures on $H'(F_0)$, $H'_1 ({F_0})$, and $H'_2 ({F_0})$ as in \S\ref{s:FL}.  We finally define the intersection number $\Int(g)$ for $g \in G_{W_1}(F_0)_\rs$ or $g \in G_1(F_0)_\rs$, and the formal locus $\Delta_x$ and the quantity $\lInt(x)$ for $x \in \fkg_1(F_0)_\rs$, as in \S\ref{s:AFL}.

\begin{conjecture}[Arithmetic transfer conjecture]\label{ATC ram odd}\hfill
\begin{altenumerate}
\renewcommand{\theenumi}{\alph{enumi}}
\item
\textup{(Homogeneous version)}
Let $f' \in C^{\infty}_c (G')$ be a function transferring to the pair $(\mathbf{1}_{K_0^\flat \times K_0}, 0) \in C_c^\infty(G_{W_0}) \times C_c^\infty(G_{W_1})$. Then there exists a function $f'_\corr \in C^{\infty}_c (G')$ such that, for any $\gamma \in G'(F_0)_\rs$ matching an element $g \in G_{W_1}(F_0)_\rs$,
\begin{equation*}
   \omega_{G'} (\gamma)\del(\gamma, f') = - \Int(g)\cdot \log q + \omega (\gamma)\Orb(\gamma, f'_\corr).
\end{equation*}
Furthermore, there exist such functions $f'$ for which $f'_\corr = 0$.
\item
\textup{(Inhomogeneous version)}
Let $f' \in C^{\infty}_c (S)$ be a function transferring to the pair $(\mathbf{1}_{K_0}, 0) \in C_c^\infty(G_0) \times C_c^\infty(G_1)$. Then there exists a function $f'_\corr \in C^\infty_c(S)$ such that, for any $\gamma \in S(F_0)_\rs$ matching an element $g \in G_1(F_0)_\rs$,
\begin{equation*}
   2 \omega_S (\gamma)\del(\gamma, f') = - \Int(g)\cdot \log q + \omega (\gamma)\Orb(\gamma, f'_\corr).
\end{equation*}
Furthermore, there exist such functions $f'$ for which $f'_\corr = 0$.
\item
\textup{(Lie algebra version)}
Let $\phi' \in C^{\infty}_c (\fks)$ be a function transferring to the pair $(\mathbf{1}_{\fkk_0},0) \in C_c^\infty(\fkg_0) \times C_c^\infty(\fkg_1)$. Then there exists a function $\phi'_\corr \in C^{\infty}_c (\fks)$ such that, for any $y \in \fks(F_0)_\rs$ matching an element $x \in \fkg_1(F_0)_\rs$ for which the intersection $\Delta \cap \Delta_x$ is an artinian scheme,
\begin{equation*}
   2\omega_\fks (y)\del(y, \phi') = - \lInt(x)\cdot \log q + \omega (y)\Orb(y, \phi'_\corr).
\end{equation*}
Furthermore, there exist such functions $\phi'$ for which $\phi'_\corr = 0$.
\end{altenumerate}
\end{conjecture}

A proof of this conjecture in the case $F_0 = \BQ_p$ and $n=3$ is contained in \cite{RSZ}.

\section{Arithmetic transfer conjecture, $F/F_0$ ramified, $n$ even}\label{s:ATCeven}
In this section $F/F_0$ is ramified, $n \geq 2$ is even, and we take the special vectors $u_i \in W_i$ to have norm $-1$.  We are going to formulate an AT conjecture analogous to the ones we have already encountered, but the setup this time is more complicated.  Indeed, in this case
``taking the product with $\ov\CE$'' defines a map on $\CN_{n-1}$ with values in the space $\CP_n$ defined in \S\ref{aux spaces}. Since $\CP_n$ is not regular, we cannot emulate the formulation of the previous AFL and AT conjectures in the most literal-minded fashion, with $\CP_n$ in the place of $\CN_n$.  Instead, we will use the results of \S\ref{aux spaces} to compose the morphism $\CN_{n-1} \to \CP_n$ with a morphism (in fact, two of them) from $\CP_n$ to $\CN_n$.

To make this precise, recall from Example \ref{CN_1} the universal object $(\CE,\iota_\CE,-\lambda_\CE,\rho_\CE)$ over $\CN_1$.
Taking the product with the conjugate of this object defines a closed embedding
\begin{equation}\label{wtdelta_CN}
	\begin{gathered}
   \wt\delta_\CN \colon
	\xymatrix@R=0ex{
	   \CN_{n-1} \ar[r]  &  \CP_n\\
		(X, \iota, \lambda, \rho) \ar@{|->}[r]  
		   &  \bigl(X \times \ov\CE, \iota \times \iota_{\ov\CE}, 
			       \lambda \times (-\lambda_{\ov\CE}), \rho \times \rho_{\ov\CE} \bigr),
	}
	\end{gathered}
\end{equation}
where the last entry is a framing to the constant object over the special fiber defined by $\wt\BX_n = \BX_{n-1} \times \ov\BX_1$, cf.~\eqref{wtBX_n^(1) def}.  Note that for $\wt\delta_\CN$ to be well-defined, it must produce quadruples satisfying the wedge condition and condition \eqref{even wtCN spin cond} on $\CP_n$.  It is straightforward, albeit tedious, to verify directly that this follows from the wedge condition and condition \eqref{odd spin cond} on $\CN_{n-1}$.  More conceptually, note that there is an obvious analog of the morphism $\wt\delta_\CN$ between the naive local models (cf.\ Definition \ref{LM def}) for $\CN_{n-1}$ and $\CP_n$; this is explicitly given by the formula for the first arrow in (\ref{LM delta_CN^pm}) below.  Since the (honest) local model for $\CN_{n-1}$ is flat, it is automatically carried into the (honest) local model for $\CP_n$, which proves that $\wt\delta_\CN$ is well-defined.

Now recall from \S\ref{aux spaces} the space $\CP_n'$, its tautological projections $\CP_n' \to \CP_n$ and $\CP_n' \xra\varphi \CN_n$, and its decomposition $\CP_n' = (\CP_n')^+ \amalg (\CP_n')^-$.  By Theorem \ref{wtCN'pm isom}, the projection to $\CP_n$ sends each of these summands isomorphically to $\CP_n$, and we denote by $\psi^\pm$ the inverse isomorphism $\CP_n \isoarrow (\CP_n')^\pm$.  We then define the composite morphism
\[
   \delta_\CN^\pm \colon 
	   \CN_{n-1} \xra{\wt\delta_\CN} 
		\CP_n \xra[\undertilde]{\psi^\pm}
		(\CP_n')^\pm \xra\varphi
		\CN_n.
\]

\begin{proposition}\label{delta^pm closed emb}
The morphism $\delta_\CN^\pm$ is a closed embedding.
\end{proposition}

\begin{proof}
Since $\CN_{n-1}$ is essentially proper over $\Spf O_{\breve F}$, it suffices to show that $\delta_\CN^\pm$ is universally injective and formally unramified.  For the first claim, for notational simplicity we just show that $\delta_\CN^\pm$ is an injection on $\ov k$-points; the argument for points in an arbitrary algebraically closed field is the same.  Let $\BN_n$, $\wt\BN_n$, $\BN_{n-1}$, and $\ov\BN_1$ denote the (covariant, as always) rational Dieudonn\'e modules of $\BX_n$, $\wt\BX_n$, $\BX_{n-1}$, and $\ov\BX_1$, respectively. We endow all of these with an $\breve F/\breve F_0$-hermitian structure as before.  In particular, $\ov\BN_1$ is a $1$-dimensional $\breve F$-vector space, and we choose a basis vector $e$ such that $O_{\breve F}e \subset \ov\BN_1$ is the (self-dual) Dieudonn\'e lattice of $\ov\BX_1$.  Identify $\BN_n$ with $\wt\BN_n = \BN_{n-1} \oplus \ov\BX_1$ via the isogeny $\phi_0$ in \eqref{phi_0}.  Let $(Y,\iota,\lambda,\rho) \in \CN_{n-1}(\ov k)$ with Dieudonn\'e module $L \subset \BN_{n-1}$. Our problem is to show that if $M$ is a $\pi$-modular Dieudonn\'e lattice in $\BN_n$ with $\pi L^\vee \oplus \pi O_{\breve F} e \subset^1 M \subset^1 L \oplus O_{\breve F} e$, then we can recover $L$ uniquely from $M$.  Since the image of $M$ in $(L \oplus O_{\breve F} e)/(\pi L^\vee \oplus \pi O_{\breve F} e)$ is an isotropic line and $e$ has non-zero norm modulo $\pi$, we have $e \notin M$.  Hence $M + O_{\breve F}e = L \oplus O_{\breve F} e$, from which we can indeed recover $L$. 

To show that $\delta_\CN^\pm$ is formally unramified, it suffices to show that the corresponding map on local models is a closed embedding; the conclusion then follows from the usual local model diagram argument.  We use (essentially) the notation of \S\ref{odd ram max} for lattices and related objects in $F^n$, and we use a $\flat$ to denote the analogous objects in $F^{n-1}$.  In particular, we denote by $e_1^\flat,\dotsc,e_{n-1}^\flat$ the standard basis in $F^{n-1}$, endowed with the standard split hermitian form \eqref{herm form}.  Consider a one-dimensional hermitian space $Fe$ with basis vector $e$ of norm $-1$, and let $\Lambda := O_F e \subset Fe$ denote the self-dual lattice.  We take the vectors
\[
   e_i := e_i^\flat, \quad 1 \leq i \leq m-1; \quad e_m := e_m^\flat - e; \quad e_{m+1} := \frac{e_m^\flat + e} 2;  \quad e_i := e_{i-1}^\flat, \quad m+2 \leq i \leq n
\]
as a basis for $F^{n-1} \oplus Fe$.  With respect to this basis, the orthogonal sum of the hermitian forms on $F^{n-1}$ and $Fe$ is the standard one \eqref{herm form}.  We define the lattices $\Lambda_i \subset F^{n-1} \oplus Fe$ with respect to $e_1,\dotsc,e_n$ as in \S\ref{odd ram max}.  In particular, we have $\Lambda_{m-1} = \Lambda_{m-1}^\flat \oplus \Lambda$.  Let $M_{\{m-1\}}$ and $M_{\{m\}}$ denote the schemes over $\Spec O_F$ defined in Definition \ref{LM def} in the case of signature $(r,s) = (n-1,1)$,\footnote{When $(r,s) = (1,1)$, these schemes are defined over $\Spec O_{F_0}$, and here we implicitly replace them with their base change to $\Spec O_F$.} and analogously let $M_{\{m-1\}}^\flat$ denote the scheme defined with respect to the basis $e_1^\flat,\dotsc,e_{n-1}^\flat$ in the case of signature $(r,s) = (n-2,1)$ and $I = \{\lfloor \frac{n-1}2 \rfloor\} = \{m-1\}$. Then $M_{\{m-1\}}^\flat$, $M_{\{m-1\}}$, and $M_{\{m\}}$ are the local models for $\CN_{n-1}$, $\CP_n$, and $\CN_n$, respectively.  The map between local models corresponding to $\delta_\CN^\pm$ is the composite
\begin{equation}\label{LM delta_CN^pm}
\begin{gathered}
	\xymatrix@R=0ex{
	   M_{\{m-1\}}^\flat \ar[r]  &  M_{\{m-1\}} \ar[r]^-\nu  &  M_{\{m\}}\\
		\CF_{m-1} \ar@{|->}[r]  &  \CF_{m-1} \oplus \CF
	}
\end{gathered},
\end{equation}
where, for $S$ a base scheme, $\CF$ is the rank one submodule $(e \otimes \pi + \pi e \otimes 1) \subset \Lambda \otimes_{O_{F_0}} \CO_S$, and $\nu$ is the morphism \eqref{nu} (here we use Proposition \ref{lm isom}(\ref{lm isom ii}) in the case $I = \{m-1\}$, which, as in the proof of this proposition, is equivalent to the fact that $\nu$ carries $M_{\{m-1\}}$ into $M_{\{m\}}$). Since the local models are proper, to show that the morphism \eqref{LM delta_CN^pm} is a closed embedding, we need only show that it is a monomorphism.  Let $\CF_{m-1}$ be an $S$-point on $M_{\{m-1\}}^\flat$. Note that the inclusions $\Lambda_{m-1}^\flat \subset \Lambda_{m-1} \subset \Lambda_m$ present $\Lambda_{m-1}^\flat$ as a direct summand of $\Lambda_m$.  Therefore, in the notation of \eqref{T_i}, $T_{m-1}(\Lambda_{m-1}^\flat \otimes_{O_{F_0}} \CO_S)$ is a direct summand of $\Lambda_m \otimes_{O_{F_0}} \CO_S$, and $T_{m-1}(\CF_{m-1})$ is a locally direct summand of rank $n-1$. The proof of Lemma \ref{LM mapping lem} then shows that $\nu(\CF_{m-1} \oplus \CF)$ is the unique Lagrangian submodule in $\Lambda_m \otimes_{O_{F_0}} \CO_S$ for the form $\sform_m$ (cf.~Remark \ref{I=m}) which contains $T_{m-1}(\CF_{m-1})$ and satisfies the spin condition.  Now, one sees readily that $\sform_m$ restricts to a nondegenerate split symmetric form on $T_{m-1}(\Lambda_{m-1}^\flat \otimes_{O_{F_0}} \CO_S)$, and hence $T_{m-1}(\CF_{m-1})$ is Lagrangian in $T_{m-1}(\Lambda_{m-1}^\flat \otimes_{O_{F_0}} \CO_S)$.  Hence
\[
   \nu(\CF_{m-1} \oplus \CF) \cap T_{m-1}(\Lambda_{m-1}^\flat \otimes_{O_{F_0}} \CO_S) = T_{m-1}(\CF_{m-1}),
\]
which shows that we can recover $\CF_{m-1}$ from its image in $M_{\{m\}}$.
\end{proof}

As usual, we take the graph morphism of $\delta_\CN^\pm$ to obtain a closed embedding
\begin{equation*}
   \Delta_\CN^\pm \colon \CN_{n-1} \xra{(\id_{\CN_{n-1}},\delta_\CN^\pm)} \CN_{n-1, n} 
	   = \CN_{n-1} \times_{\Spf O_{\breve F}} \CN_n.
\end{equation*}
Let $\Delta$ denote the (disjoint) union of the images of these embeddings,
\[
   \Delta := \Delta_\CN^+(\CN_{n-1}) \amalg \Delta_\CN^-(\CN_{n-1}).
\]
Of course, $\Delta$ can also be described as the image of the closed embedding
\[
   \Delta_\CN\colon \CN_{n-1} \times \{\pm 1\}
	   = \bigl(\CN_{n-1} \times \{+1\}\bigr) \amalg \bigl(\CN_{n-1} \times \{-1\}\bigr)
		\xra{\Delta_\CN^+ \amalg \Delta_\CN^-} \CN_{n-1,n}.
\]
Identifying $\BV(\BX_n)$ and $\BV(\wt\BX_n) = \BV(\BX_{n-1} \times \ov\BX_1)$ via the isogeny $\phi_0$ in \eqref{phi_0}, let
\begin{equation}\label{u}
   u := (0,\id_{\ov\BE}) \in \BV(\BX_n) \simeq \BV\bigl(\wt\BX_n\bigr) = \BV(\BX_{n-1}) \oplus \BV\bigl(\ov\BX_1\bigr).
\end{equation}
Then $u$ has norm $-1$, and therefore we may identify $(\BV(\BX_n),u,\BV(\BX_{n-1}))$ with $(W_1,u_1,W_1^\flat)$.  In this way $\Delta_\CN$ is equivariant for the embedding $H_1(F_0) \inj G_{W_1}(F_0)$, where $H_1(F_0)$ acts via the Kottwitz map on $\{\pm1\}$ and diagonally on the product $\CN_{n-1} \times \{\pm 1\}$.  (Note that the embeddings $\Delta_\CN^+$ and $\Delta_\CN^-$ are not separately $H_1(F_0)$-equivariant.)

\begin{example}[$n=2$]\label{CN_2 ATC eg}
Let us make the the above discussion ``concrete'' when $n = 2$.  Recall that $\CN_1 = \Spf O_{\breve F}$ is the universal deformation space for $\BE$ as a formal $O_F$-module, with universal object $(\CE,\iota_\CE,\rho_\CE)$, cf.~Example \ref{CN_1}; that $\CN_2$ identifies with two copies of the Lubin--Tate space $\CM_{\Spf O_{\breve F}}$, cf.~Example \ref{CN_2}; and that $\CP_2$ identifies with $(\CM_{\Gamma_0(\varpi)})_{\Spf O_{\breve F}}$, cf.~Remark \ref{CP_n diagram n=2}.  In terms of these identifications, we claim that $\wt\delta_\CN\colon \CN_1 \to \CP_2$ is the morphism
\[
   \xymatrix@R=0ex{
	   \Spf O_{\breve F} \ar[r]  &  (\CM_{\Gamma_0(\varpi)})_{\Spf O_{\breve F}}\\
		(\CE,\iota_\CE,\rho_\CE) \ar@{|->}[r]  &  (\CE,\CE,\iota_\CE(\pi),\rho_\CE).
	}
\]
Indeed, $\wt\delta_\CN$ sends $(\CE,\iota_\CE,\rho_\CE)$ to the point $(\CE \times \ov\CE, \iota_\CE \times \iota_{\ov\CE}, -(\lambda_\CE \times \lambda_{\ov\CE}), \rho_\CE \times \rho_{\ov\CE})$ on $\CP_2$; and the asserted value in $(\CM_{\Gamma_0(\varpi)})_{\Spf O_{\breve F}}$ identifies with the point $(\CE\times\CE, \iota, -2(\lambda_\CE \times \lambda_\CE), \psi_0\circ(\rho_\CE \times \rho_\CE))$ on $\CP_2$, where $\iota$ is defined by
\[
   \iota(\pi) =
	\begin{bmatrix}
		  &  \iota_\CE(\pi)\\
		\iota_\CE(\pi)
	\end{bmatrix},
\]
where $\psi_0$ is the isomorphism $\BX \isoarrow \wt\BX_2$ in \eqref{psi_0}, and where $\rho_\CE \times \rho_\CE$ is a framing to $\BX$.  It is obvious from the definition of $\psi_0$ that the isomorphism in the special fiber given by the framings lifts to an isomorphism between these two quadruples, which proves the claim.  Combining this with Example \ref{CP_n diagram n=2}, it follows that both morphisms $\delta_\CN^\pm$ identify with the embedding of $\CN_1$ in $\CM_{\Spf O_{\breve F}}$ sending $(\CE,\iota_\CE,\rho_\CE) \mapsto (\CE,\rho_\CE)$. Furthermore, the composition
\[
   \CN_1 \xra{\delta_\CN^\pm} \CM_{\Spf O_{\breve F}} \to \CM
\]
identifies with the canonical divisor in $\CM$ associated to the embedding $\iota_\BE\colon F \inj D$ in the sense of \cite{G} and \cite[Def.~1.2]{Wewers}.  Finally, note that $\Delta_\CN^\pm = \delta_\CN^\pm$ under the identification $\CN_{1,2}\cong \CN_2$.
\end{example}

Now fix an almost $\pi$-modular lattice
\[
   \Lambda_0^\flat \subset W_0^\flat;
\]
recall that when $F/F_0$ is ramified, such lattices exist in any odd-dimensional hermitian space and are all conjugate under the corresponding unitary group.
Then
\begin{equation}\label{Lambda^nat}
   \Lambda_0^\nat := \Lambda_0^\flat \oplus O_Fu_0
\end{equation}
is a vertex lattice of type $n-2$ in $W_0$, i.e.~$\Lambda_0^\nat \subset^{n-2} (\Lambda_0^\nat)^\vee \subset^2 \pi^{-1}\Lambda_0^\nat$.  Since $W_0$ is split, it follows that there are two $\pi$-modular lattices $\Lambda_0^+,\Lambda_0^-$ contained in $\Lambda_0^\nat$ (corresponding to the two isotropic lines in $\Lambda_0^\nat/\pi (\Lambda_0^\nat)^\vee$).  Let
\[
   K_0^\flat \subset H_0(F_0)
\]
denote the stabilizer of $\Lambda_0^\flat$, and let
\[
   K_0^\pm \subset G_0(F_0)
	\quad\text{and}\quad
	\fkk_0^\pm \subset \fkg_0(F_0)
\]
denote the respective stabilizers of $\Lambda_0^\pm$. 
Then $K_0^\flat$ is a maximal compact subgroup containing a special maximal parahoric subgroup with index $2$, and $K_0^\pm$ are special maximal parahoric subgroups, cf.~\cite[\S4.a]{PR-TLG}. Note that the two lattices $\Lambda_0^\pm$ are $K_0^\flat$-conjugate under the embedding $K_0^\flat \subset G_0(F_0)$, since, for example, $K_0^\flat$ contains elements of nontrivial Kottwitz invariant.

\begin{remark}\label{group theoretic}
The fact that we take $K_0^\flat$ to be the stabilizer of an almost $\pi$-modular lattice in $W_0^\flat$ reflects that we have almost $\pi$-modular polarizations in the moduli problem for $\CN_{n-1}$; and similarly, that $K_0^\pm$ are stabilizers of $\pi$-modular lattices reflects that the polarizations in the moduli problem for $\CN_n$ are $\pi$-modular.  The fact that $K_0^\flat \not\subset K_0^\pm$ under the embedding $H_0(F_0) \subset G_0(F_0)$ can then be regarded as a group-theoretic reflection of why $\delta_\CN$ cannot be defined as simply in the even ramified case as in the cases we have encountered previously.
\end{remark}

We normalize the Haar measure on $H_0(F_0)$ (and hence the product measure on $H_{W_0}(F_0) = H_0(F_0) \times H_0(F_0)$) by assigning $K_0^\flat$ volume $1$.  As before we normalize the Haar measures on $H'(F_0)$, $H'_1 ({F_0})$, and $H_2'(F_0)$ as in \S\ref{s:FL}.  For $g \in G_{W_1}(F_0)_\rs$ or $g \in G_1(F_0)_\rs$, we define the intersection number $\Int(g)$ with respect to the above cycle $\Delta$ as in \S\ref{s:AFL}.  For $x \in \fkg_1(F_0)_\rs$, we define (abusing notation in the usual obvious way)
\[
   \Delta_x := \biggl\{\, (Y,X) \in \CN_{n-1,n} \biggm| 
	   \begin{varwidth}{\linewidth}
		\centering
		Zariski-locally on the base, $x\colon \BX_n \rightarrow \BX_n$\\
		lifts to a homomorphism $\delta_\CN^+(Y) \rightarrow X$ or $\delta_\CN^-(Y) \rightarrow X$
		\end{varwidth} \,\biggr\}.
\]
As in the case of the AFL, this definition makes sense for any quasi-endomorphism $x$ of $\BX_n$, and for $g \in G_1(F_0)$ we have $\Delta_g = (1\times g)\Delta$. As usual, when $\Delta \cap \Delta_x$ is an artinian scheme, we define
\[
   \lInt(x) := \length(\Delta \cap \Delta_x).
\]

\begin{conjecture}[Arithmetic transfer conjecture]\label{conjram v3}\hfill
\begin{altenumerate}
\renewcommand{\theenumi}{\alph{enumi}}
\item\label{conjram v3 homog}
\textup{(Homogeneous version)} 
Let $f' \in C^\infty_c(G')$ be a function transferring to the pair of functions $(\mathbf{1}_{K_0^\flat \times K_0^+} + \mathbf{1}_{K_0^\flat \times K_0^-}, 0) \in C_c^\infty(G_{W_0}) \times C_c^\infty(G_{W_1})$.  Then there exists $f'_\corr \in C^{\infty}_c (G')$ such that, for any $\gamma \in G'(F_0)_\rs$ matching an element $g \in G_{W_1}(F_0)_\rs$,
\begin{equation*}
   \omega_{G'}(\gamma)\del(\gamma, f') 
	   = - \Int(g)\cdot\log q + \omega(\gamma) \Orb(\gamma, f'_\corr).
\end{equation*}
Furthermore, there exist such functions $f'$ for which $f'_\corr = 0$.
\item\label{conjram v3 inhomog}
\textup{(Inhomogeneous version)}
Let $f' \in C^{\infty}_c (S)$ be a function transferring to the pair of functions $(\mathbf{1}_{K_0^\flat K_0^+}+\mathbf{1}_{K_0^\flat K_0^-}, 0) \in C_c^\infty(G_0) \times C_c^\infty(G_1)$. Then there exists  $f'_\corr \in C^\infty_c(S)$ such that, for any $\gamma \in S(F_0)_\rs$ matching an element $g \in G_1(F_0)_\rs$,
\begin{equation*}
   2 \omega_S (\gamma)\del(\gamma, f') = - \Int(g)\cdot \log q + \omega (\gamma)\Orb(\gamma, f'_\corr).
\end{equation*}
Furthermore, there exist such functions $f'$ for which $f'_\corr = 0$.
\item\label{conjram v3 lie}
\textup{(Lie algebra version)}
Let $\phi' \in C^{\infty}_c (\fks)$ be a function transferring to the pair of functions $(\mathbf{1}_{\fkk_0^+}+\mathbf{1}_{\fkk_0^-},0) \in C_c^\infty(\fkg_0) \times C_c^\infty(\fkg_1)$. Then there exists $\phi'_\corr \in C^{\infty}_c (\fks)$ such that, for any $y \in \fks(F_0)_\rs$ matching an element $x \in \fkg_1(F_0)_\rs$ for which the intersection $\Delta \cap \Delta_x$ is an artinian scheme,
\begin{equation*}
   2\omega_\fks (y)\del(y, \phi') = - \lInt(x)\cdot \log q + \omega (y)\Orb(y, \phi'_\corr).
\end{equation*}
Furthermore, there exist such functions $\phi'$ for which $\phi'_\corr = 0$.
\end{altenumerate}
\end{conjecture}

\begin{remark}\label{remark S12}
Note that in the inhomogeneous version, we have 
\[
   \mathbf{1}_{K_0^\flat K_0^+}+\mathbf{1}_{K_0^\flat K_0^-}=\mathbf{1}_{ K_0^+ K_0^\flat }+\mathbf{1}_{K_0^- K_0^\flat}.
\]
Moreover, we may replace the test function  $\mathbf{1}_{K_0^\flat K_0^+}+\mathbf{1}_{K_0^\flat K_0^-}$ by $2\cdot \mathbf{1}_{K_0^\flat K_0^+}$ or $2 \cdot\mathbf{1}_{K_0^\flat K_0^-}$, since they all have the same orbital integrals (the two groups $K_0^\pm$ are conjugate under $H_0(F_0)$). Similarly, in the Lie algebra version, we may replace the test function $\mathbf{1}_{\fkk_0^+}+\mathbf{1}_{\fkk_0^-}$ by $2\cdot \mathbf{1}_{\fkk_0^+}$ or $2\cdot \mathbf{1}_{\fkk_0^-}$.
\end{remark}
The proof of Conjecture \ref{conjram v3} in the case $F_0 = \BQ_p$ and $n=2$ is contained in \S\ref{even ram proof}.

\section{Arithmetic transfer theorems for $\wt\CN_2^{(0)}$ and $\wt\CN_2^{(1)}$}\label{s:ATC ram n=2}

In this final section of Part \ref{conjectures}, we state AT theorems attached to the spaces $\smash[t]{\wt\CN_2^{(0)}}$ and $\wt\CN_2^{(1)}$ defined in \S\ref{ram self-dual n=2} over $\Spf O_{\breve F_0}$.  To formulate the statements, we will define an embedding of $\CN_1$ into each of them, and then proceed in a way similar to before, albeit with a few differences.  In particular, we will only obtain statements in the inhomogeneous group and Lie algebra cases; and in order to intersect cycles in a regular space, it is crucial that we work over $O_{\breve F_0}$ instead of $O_{\breve F}$.  Let $F/F_0$ be ramified.

\subsection{The case of $\wt\CN_2^{(0)}$}\label{ATC wtCN_2^(0)}
Identify $F$ with its image in $D$ via $\iota_\BE$, and fix $\zeta \in O_D^\times$ such that $\zeta \pi = - \pi \zeta$.  Then the elements $\pi$ and $\zeta$ generate $O_D$ over $O_{F_0}$ since $p \neq 2$.  Furthermore, we have
\[
   \ep := \zeta^2 \in O_{F_0}^\times \smallsetminus \RN F^\times.
\]
We take the special vectors $u_i \in W_i$ to have norm $-\ep$.  Let
\[
   \bigl(\ov\BX_1^{(0)}, \iota_{\ov\BX_1^{(0)}}, \lambda_{\ov\BX_1^{(0)}}\bigr) := \bigl(\ov\BE, \iota_{\ov\BE}, -\ep\lambda_{\ov\BE}\bigr),
\]
which is the conjugate of the framing object $\BX_1^{(0)}$ for $\CN_1^{(0)}$, cf.\ \eqref{BX_1} and \eqref{BX_n^{(0)}}.  Taking the product with the conjugate of the corresponding universal object over $\CN_1^{(0)}$ then defines a closed embedding
\begin{equation*}
   \wt\delta^{(0)}_\CN \colon
	\xymatrix@R=0ex{
	   \CN_1 \ar[r]  &  (\wt\CN_2^{(0)})_{\Spf O_{\breve F}}\\
		(\CE, \iota_\CE, -\lambda_\CE, \rho_\CE) \ar@{|->}[r]  
		   &  \bigl(\CE \times \ov\CE, \iota_\CE \times \iota_{\ov\CE}, 
			       -(\lambda_\CE \times \ep\lambda_{\ov\CE}), \rho_\CE \times \rho_{\ov\CE} \bigr).
	}
\end{equation*}
Here the last entry $\rho_\CE \times \rho_{\ov\CE}$ is a framing to the constant object over the special fiber defined by
\begin{equation}\label{wtBX_2^(0)}
   \wt\BX^{(0)}_2 := \BX_1 \times \ov\BX_1^{(0)} = \bigl(\BE \times \ov\BE, \iota_\BE \times \iota_{\ov\BE}, 
			       -(\lambda_\BE \times \ep\lambda_{\ov\BE}) \bigr).
\end{equation}
Note that since $\BV(\BX_1) = \BV(\BX_1^{(1)})$ is non-split, $\BV(\wt\BX_2^{(0)})$ is indeed split by \eqref{chi decomp formula}.
We compose with the projection to the regular space $\wt\CN_2^{(0)}$ to obtain
\begin{equation*}
   \wt\Delta^{(0)}_\CN \colon \CN_1 \xra{\wt\delta_\CN^{(0)}} (\wt\CN_2^{(0)})_{\Spf O_{\breve F}} \to \wt\CN_2^{(0)}.
\end{equation*}
By inspection $\wt\Delta_\CN^{(0)}$ is again a closed embedding, and we denote its image by
\[
   \wt\Delta^{(0)} \subset \wt\CN_2^{(0)}.
\]
The canonical vector
\[
   u := (0,\id_{\ov\BE}) \in \BV\bigl(\wt\BX_2^{(0)}\bigr) = \BV(\BX_1) \oplus \BV\bigl(\ov\BX_1^{(0)}\bigr)
\]
has norm $-\ep$, and therefore we may identify the triple $(\BV(\wt\BX_2^{(0)}),u,\BV(\BX_1))$ with $(W_0,u_0,W_0^\flat)$. In this way $\wt\Delta_\CN^{(0)}$ is equivariant for the embedding $H_0(F_0) \inj G_0(F_0)$.

To express the above discussion in terms of the natural identifications we have previously noted for the spaces in play, recall from Example \ref{CN_1} that $\CN_1 = \Spf O_{\breve F}$ is the universal deformation space for $\BE$ as a formal $O_F$-module, with universal object $(\CE,\iota_\CE,\rho_\CE)$.  Further recall from \eqref{ramified Drinfeld isom} that $\wt\CN_2^{(0)}$ identifies with the Drinfeld space $(\wh\Omega_{F_0}^2)_{\Spf O_{\breve F_0}}$.  In fact, we will now use a slightly different description for the moduli problem for $(\wh\Omega_{F_0}^2)_{\Spf O_{\breve F_0}}$ than the one in \cite{KR-alt}, by changing the framing object.  Let
\[
   \BX_\Dr := \BE \times \ov\BE,
\]
endowed with its natural $O_F$-action $\iota_{\BX_\Dr}$ on the right-hand side, and extend this to an $O_D$-action by defining
\begin{equation}\label{iota(zeta) formula}
   \iota_{\BX_\Dr}(\zeta) :=
	\begin{bmatrix}
		  &  \ep\\
		1
	\end{bmatrix}.
\end{equation}
Then $\iota_{\BX_\Dr}(\zeta)\iota_{\BX_\Dr}(\pi) = - \iota_{\BX_\Dr}(\pi)\iota_{\BX_\Dr}(\zeta)$, as required, and we take $\BX_\Dr$ as the framing object for the Drinfeld moduli problem.  Now let
\[
   \lambda_{\BX_\Dr} := -(\lambda_\BE \times \ep \lambda_{\ov\BE}).
\]
It is easy to verify that
\[
   \Ros_{\lambda_{\BX_\Dr}}\bigl(\iota_{\BX_\Dr}(\pi)\bigr) = - \iota_{\BX_\Dr}(\pi)
	\quad\text{and}\quad
	\Ros_{\lambda_{\BX_\Dr}}\bigl(\iota_{\BX_\Dr}(\zeta)\bigr) = \iota_{\BX_\Dr}(\zeta).
\]
Hence by \cite[Th.~1.2]{KR-alt}, for any object $(X,\iota,\rho)$ of the Drinfeld moduli problem, there is a unique principal polarization $\lambda$ on $X$ lifting $\lambda_{\BX_\Dr}$, and $(X,\iota,\rho) \mapsto (X,\iota|_{O_F},\lambda,\rho)$ defines an isomorphism $(\wh\Omega_{F_0}^2)_{\Spf O_{\breve F_0}} \isoarrow \wt\CN_2^{(0)}$.  In terms of this identification, the map $\wt\Delta_\CN^{(0)}$ is therefore given by
\[
   \xymatrix@R=0ex{
	   \Spf O_{\breve F} \ar[r]  &  (\wh\Omega_{F_0}^2)_{\Spf O_{\breve F_0}}\\
		(\CE,\iota_\CE,\rho_\CE) \ar@{|->}[r]  &  (\CE \times \ov\CE, \iota, \rho_\CE \times \rho_{\ov\CE}),
	}
\]
where $\iota$ restricts to the natural $O_F$-action on $\CE \times \ov\CE$, and $\iota(\zeta)$ is defined as in \eqref{iota(zeta) formula}.

\begin{remark}\label{specialDr} Again, as in the case of Remark \ref{reltospecialunr}, one can introduce the special cycle $\CZ(u)$ in $\wt\CN_2^{(0)}$. 
However, according to Sankaran \cite{San-mail}, $\CZ(u)$ is not a divisor, but has embedded components. The divisor $\wt\Delta^{(0)}_\CN (\CN_1)$ is the \emph{purification} of $\CZ(u)$. 
\end{remark}

Now consider the one-dimensional space $W_1^\flat$, and let
\[
   \Lambda_1^\flat \subset W_1^\flat
\]
be the unique self-dual $O_F$-lattice. Let
\[
   \Lambda_1 := \Lambda_1^\flat \oplus O_Fu_1 \subset W_1,
\]
which is again self-dual.  Let 
\[
   \wt K_1 \subset G_1(F_0)
	\quad\text{and}\quad
	\wt \fkk_1 \subset \fkg_1(F_0)
\]
denote the respective stabilizers of $\Lambda_1$.  Then $\wt K_1$ contains the connected stabilizer of $\Lambda_1$ (which is the unique parahoric subgroup\footnote{After extension of scalars, this becomes an Iwahori subgroup.}) with index $2$. We normalize the Haar measure on $H_1(F_0) = F^1$ by assigning it volume $1$.  As usual we normalize the Haar measure on $H'(F_0)$ as in \S\ref{s:FL}.  For $g \in G_0(F_0)_\rs$, we define the intersection number $\Int(g)$ with respect to the cycle $\wt\Delta^{(0)} \subset \wt\CN_2^{(0)}$ in analogy with before,
\[
   \Int(g) := \bigl\la \wt\Delta^{(0)}, g\wt\Delta^{(0)} \bigr\ra_{\wt\CN_2^{(0)}}.
\]
Since $\wt\Delta^{(0)}\cap g\wt\Delta^{(0)}$ is an artinian scheme, this intersection number is simply given by the length 
\[
   \Int(g) =\length\bigl(\wt\Delta^{(0)} \cap g\wt\Delta^{(0)}\bigr) ,
\] 
cf.\ \cite[Prop.\ 8.10]{RSZ}.
For $x \in \fkg_0(F_0)_\rs$, we define (abusing notation as usual)
\[
   \wt\Delta_x^{(0)} := \bigl\{\, X \in \wt\CN_2^{(0)} \bigm| x\colon \wt\BX_2^{(0)} \rightarrow \wt\BX_2^{(0)} \text{ lifts to a homomorphism } \CE \times \ov\CE \rightarrow X \,\bigr\}.
\]
Then $\wt\Delta^{(0)} \cap \wt\Delta_x^{(0)}$ is an artinian scheme, and we define
\[
   \lInt(x) := \length(\wt\Delta^{(0)} \cap \wt\Delta_x^{(0)}).
\] 

\begin{theorem}[Arithmetic transfer theorem]\label{conjram i=0}
Let $F_0 = \BQ_p$.
\begin{altenumerate}
\renewcommand{\theenumi}{\alph{enumi}}
\item\label{conjram i=0 inhomog}
\textup{(Inhomogeneous version)}
Let $f' \in C^\infty_c (S)$ be a function transferring to the pair $(0,\mathbf{1}_{\wt K_1}) \in C_c^\infty(G_0) \times C_c^\infty(G_1)$. Then there exists a function $f'_\corr \in C^\infty_c(S)$ such that, for any $\gamma \in S(F_0)_\rs$ matching an element $g \in G_0(F_0)_\rs$,
\begin{equation*}
   2 \omega_S (\gamma)\del(\gamma, f') = - \Int(g)\cdot \log q + \omega (\gamma)\Orb(\gamma, f'_\corr).
\end{equation*}
Furthermore, there exist such functions $f'$ for which $f'_\corr = 0$.
\item\label{conjram i=0 lie}
\textup{(Lie algebra version)}
Let $\phi' \in C^{\infty}_c (\fks)$ be a function transferring to the pair $(0, \mathbf{1}_{\wt\fkk_1}) \in C_c^\infty(\fkg_0) \times C_c^\infty(\fkg_1)$. Then there exists a function $\phi'_\corr \in C^{\infty}_c (\fks)$ such that, for any $y \in \fks(F_0)_\rs$ matching an element $x \in \fkg_0(F_0)_\rs$,
\begin{equation*}
   2\omega_\fks (y)\del(y, \phi') = - \lInt(x)\cdot \log q + \omega (y)\Orb(y, \phi'_\corr).
\end{equation*}
Furthermore, there exist such functions $\phi'$ for which $\phi'_\corr = 0$.
\end{altenumerate}
\end{theorem}

We will prove Theorem \ref{conjram i=0} in \S\ref{proof wtCN_2^(0)}.

\subsection{The case of $\wt\CN_2^{(1)}$}\label{atc wtCN_2^(1)}
Now we take the special vectors $u_i \in W_i$ to have norm $-1$.
To obtain an analogous embedding of $\CN_1$ into $\wt\CN_2^{(1)}$, we simply use the morphism $\wt\delta_\CN$ defined in \eqref{wtdelta_CN}, followed by the natural projection to $\wt\CN_2^{(1)}$,
\[
   \wt\Delta_\CN^{(1)}\colon \CN_1 \xra{\wt\delta_\CN} \CP_2 \simeq (\wt\CN_2^{(1)})_{\Spf O_{\breve F}} \to \wt\CN_2^{(1)},
\]
where the isomorphism in the middle is the one in Example \ref{CP_n diagram n=2}.
Recall from Example \ref{CN_2 ATC eg} that in concrete terms we have $\CN_1 \cong \Spf O_{\breve F}$ and $\wt\CN_2^{(1)} \cong \CM_{\Gamma_0(\varpi)}$, and that $\wt\Delta_\CN^{(1)}$ then identifies with
\[
   \xymatrix@R=0ex{
	   \Spf O_{\breve F} \ar[r]  &  \CM_{\Gamma_0(\varpi)}\\
		(\CE,\iota_\CE,\rho_\CE) \ar@{|->}[r]  &  (\CE,\CE,\iota_\CE(\pi),\rho_\CE).
	}
\]

Now recall the framing object $\BX$ for $\wt\CN_2^{(1)}$, which is isomorphic to the framing object for $\CP_2$,
\begin{equation*}
   \wt\BX_2^{(1)} := \wt\BX_2 = \BX_1 \times \ov\BX_1 = \bigl(\BE \times \ov\BE, \iota_\BE \times \iota_{\ov\BE}, -(\lambda_\BE \times \lambda_{\ov\BE})\bigr),
\end{equation*}
via \eqref{psi_0}.  This gives us a canonical vector
\begin{equation}
   u := (0,\id_{\ov\BE}) \in \BV(\BX) \simeq \BV\bigl(\wt\BX_2^{(1)}\bigr) = \BV(\BX_1) \oplus \BV\bigl(\ov\BX_1\bigr)
\end{equation}
of norm $-1$, and therefore we may identify the triple $(\BV(\wt\BX_2^{(1)}),u,\BV(\BX_1))$ with $(W_1,u_1,W_1^\flat)$. In this way $\wt\Delta_\CN^{(1)}$ is equivariant for the embedding $H_1(F_0) \inj G_1(F_0)$.

\begin{remark} Again, as in the AFL case, one can introduce the special cycle $\CZ(u)$ in $\wt\CN_2^{(0)}$, cf.\ Remark \ref{reltospecialunr}. 
It seems likely that $\CZ(u)$ is not a divisor, but has embedded components \cite{KRbarc}, comp.\ Remark \ref{specialDr}. The divisor $\wt\Delta^{(0)}_\CN (\CN_1)$ is the \emph{purification} of $\CZ(u)$. 
\end{remark}

Now consider the one-dimensional space $W_0^\flat$, and let
\[
   \Lambda_0^\flat \subset W_0^\flat
\]
be the unique self-dual lattice.  Let
\[
   \Lambda_0 := \Lambda_0^\flat \oplus O_Fu_0 \subset W_0,
\]
which is again self-dual. Let 
\[
   \wt K_0 \subset G_0(F_0)
	\quad\text{and}\quad
	\wt\fkk_0 \subset \fkg_0(F_0)
\]
denote the respective stabilizers of $\Lambda_0$.  Then $\wt K_0$ contains an Iwahori subgroup with index $2$, cf.~\cite[\S4.a]{PR-TLG}. We normalize the Haar measure on $H_0(F_0) = F^1$ by assigning it volume $1$.  The other normalizations and transfer factors for $S$ and \fks are all as in the previous subsection.
Finally, as in the previous subsection, we define $\Int(g)$ and $\lInt(x)$ for $g\in G_1(F_0)_\rs$, resp.\ $x\in\fkg_{1}(F_0)_\rs$ (again, both are given by a length). 

\begin{theorem}[Arithmetic transfer theorem]\label{conjram i=1}
Let $F_0 = \BQ_p$.
\begin{altenumerate}
\renewcommand{\theenumi}{\alph{enumi}}
\item\label{conjram i=1 inhomog}
\textup{(Inhomogeneous version)}
Let $f' \in C^{\infty}_c (S)$ be a function transferring to the pair $(\mathbf{1}_{\wt K_0}, 0) \in C_c^\infty(G_0) \times C_c^\infty(G_1)$. Then there exists a function $f'_\corr \in C^\infty_c(S)$ such that, for any $\gamma \in S(F_0)_\rs$ matching an element $g \in G_1(F_0)_\rs$,
\begin{equation*}
   2 \omega_S (\gamma)\del(\gamma, f') = - \Int(g)\cdot \log q + \omega (\gamma)\Orb(\gamma, f'_\corr).
\end{equation*}
Furthermore, there exist such functions $f'$ for which $f'_\corr = 0$.
\item\label{conjram i=1 lie}
\textup{(Lie algebra version)}
Let $\phi' \in C^{\infty}_c (\fks)$ be a function transferring to the pair $(\mathbf{1}_{\wt\fkk_0},0) \in C_c^\infty(\fkg_0) \times C_c^\infty(\fkg_1)$. Then there exists a function $\phi'_\corr \in C^{\infty}_c (\fks)$ such that, for any $y \in \fks(F_0)_\rs$ matching an element $x \in \fkg_1(F_0)_\rs$,
\begin{equation*}
   2\omega_\fks (y)\del(y, \phi') = - \lInt(x)\cdot \log q + \omega (y)\Orb(y, \phi'_\corr).
\end{equation*}
Furthermore, there exist such functions $\phi'$ for which $\phi'_\corr = 0$.
\end{altenumerate}
\end{theorem}

We will prove Theorem \ref{conjram i=1} in \S\ref{proof wtCN_2^(1)}.

\part{Results}
In this part of the paper we prove results on the conjectures formulated in Part \ref{conjectures} (except Conjecture \ref{ATC ram odd}, which is proved for $n = 3$ in \cite{RSZ}).  We continue to take $S = S_n$ as in \eqref{S_n}.

\section{On the ATC for $F/F_0$ unramified, almost self-dual type}
\label{unram proof}
We resume the setup of \S\ref{s:ATunr}, with $F/F_0$ unramified and $n \geq 2$. Recall that we have an orthogonal decomposition
\[
   W_i = W_i^\flat\oplus F u_i,\quad (u_i,u_i)=\varpi, \quad i \in \{0,1\},
\]
where $W_0$ is the split hermitian space of dimension $n$, and $W_1$ is the non-split hermitian space of dimension $n$; and, by \eqref{chi decomp formula}, $W_0^\flat$ is the non-split space of dimension $n-1$, and $W_1^\flat$ is the split space of dimension $n-1$.

\subsection{On the FL Conjecture \ref{FL almost self-dual}}

In this subsection we prove the following theorem.

\begin{theorem}
\begin{altenumerate}
\item\label{fl equiv} 
The Lie algebra FL Conjectures \ref{FLconj}(\ref{FLconj lie}) and \ref{FL almost self-dual}(\ref{FL almost self-dual lie}) are equivalent to each other.
\item\label{fl equiv 2} 
The homogeneous and inhomogeneous FL Conjectures \ref{FL almost self-dual}(\ref{FL almost self-dual homog}) and \ref{FL almost self-dual}(\ref{FL almost self-dual inhomog}) are equivalent to each other.
\item\label{fl imply}
Assume $q \geq n$.  Then Conjecture \ref{FL almost self-dual}(\ref{FL almost self-dual lie}) implies Conjecture \ref{FL almost self-dual}(\ref{FL almost self-dual homog})(\ref{FL almost self-dual inhomog}).
\end{altenumerate}\label{unram variant FL}
\end{theorem}

In order to make the relation between the conjectures in part (\ref{fl equiv}), we first introduce some auxiliary spaces. 
Let
\[
   \wt W_i = \wt W_i^{\flat}\oplus F \wt u_i,\quad (\wt u_i, \wt u_i)=1, \quad i \in \{0,1\},
\]
where $\wt W_0$ is a split hermitian space and $\wt W_1$ is non-split. Thus $\wt W_i$ and $W_i$ are isometric, but we are choosing different special vectors in these spaces.  Let
\[
   \wt G_i := \U\bigl(\wt W_i\bigr),
	\quad
	\wt H_i := \U\bigl(\wt W_i^\flat\bigr),
	\quad\text{and}\quad
	\wt\fkg_i := \Lie \wt G.
\]
Recall from \S\ref{s:ATunr} the self-dual lattice $\Lambda_1^\flat \subset W_1^\flat$, the almost self-dual lattice $\Lambda_1 = \Lambda_1^\flat \oplus F u_1 \subset W_1$, and the Lie algebra stabilizer $\fkk_1 \subset \fkg_1(F_0)$ of $\Lambda_1$. We similarly fix an almost self-dual lattice $\Lambda_0^\flat \subset W_0^\flat$; we set $\Lambda_0 :=\Lambda_0^\flat\oplus O_F u_0$ (a vertex lattice of type $2$ in $W_0$); and we denote by $\fkk_0 \subset \fkg_0(F_0)$ the stabilizer of $\Lambda_0$.  We choose an identification $\wt W_i^\flat=W_{1-i}^\flat$, and we take
\[
   \wt \Lambda_i^\flat := \Lambda_{1-i}^\flat.
\]
We set $\wt\Lambda_i := \wt\Lambda_i^{\flat}\oplus O_F \wt u_i$, and we denote by $\wt\fkk_i \subset \wt\fkg_i(F_0)$ the stabilizer of $\wt\Lambda_i$. 

We next choose an $O_F$-basis of $\Lambda_i^\flat$ (and hence of $\wt\Lambda_{1-i}^\flat$) and extend it by adding $u_i$ (resp.~$\wt u_{1-i}$) to obtain an $F$-basis of $W_i$ (resp.~$\wt W_{1-i}$). In this way, all the groups under consideration identify with subgroups of $\GL_n(F)$, and all the Lie algebras identify with $F_0$-subspaces of $\RM_n(F)$. We define a $\GL_{n-1}(F)$-equivariant map
\[
  \theta\colon 
  \xymatrix@R=0ex{
     \RM_n(F) \ar[r]  &  \RM_n(F)\\
	  {\begin{bmatrix}
	 		 A &  b  \\
	       c  & d
	  \end{bmatrix}}
	  \ar@{|->}[r]  &
	  {\begin{bmatrix}
	  	  A &  \varpi^{-1} b  \\
	     c  & d
	  \end{bmatrix} },
  } 
\]
where as usual the diagonal blocks are of sizes $(n-1) \times (n-1)$ and $1 \times 1$, respectively.  Finally, we recall the compact open subgroup $\fkk' \subset \fks(O_{F_0})$ defined in \eqref{eqn fkk'}.

\begin{lemma}
\begin{altenumerate}
\item\label{lemtheta i} The restriction of $\theta$ to $\fks(F_0)$ gives a $\GL_{n-1}(F_0)$-equivariant bijection
\[
   \theta\colon \fks(F_0)\isoarrow \fks(F_0)
\]
such that
\[
   \theta (\fkk')=\fks(O_{F_0}).
\]
Furthermore, $\theta$ restricts to a bijection $\fks(F_0)_\rs \isoarrow \fks(F_0)_\rs$, and
\begin{align*}\label{eqn orb s}
   \omega_\fks (y) = (-1)^{n-1}\omega_\fks\bigl(\theta(y)\bigr) \quad\text{and}\quad
	\Orb(y, \mathbf{1}_{\fkk'}, s) = \Orb\bigl(\theta(y), \mathbf{1}_{\fks(O_{F_0})}, s\bigr), \quad y \in \fks(F_0)_\rs.
\end{align*}
\item\label{lemtheta ii} 
For $i \in \{0,1\}$, the restriction of $\theta$ to $\fkg_i(F_0)$ gives an $H_i(F_0) = \wt H_{1-i}(F_0)$-equivariant bijection
\[
   \theta\colon \fkg_i(F_0) \isoarrow \wt\fkg_{1-i}(F_0)
\]
such that
$$
\theta(\fkk_i)=\wt\fkk_{1-i}.
$$
In addition, $\theta$ further restricts to a bijection $\fkg_i(F_0)_\rs \isoarrow \wt\fkg_{1-i}(F_0)_\rs$, and
\begin{equation*}
   \Orb(x, \mathbf{1}_{\fkk_i}) = \Orb\bigl(\theta(x),\mathbf{1}_{\wt\fkk_{1-i}}\bigr), \quad x\in \fkg_i(F_0)_\rs .
\end{equation*} 
\end{altenumerate}
\label{lemtheta}
\end{lemma}

Here the final asserted equality of unitary orbital integrals when $i = 0$ is with respect to any fixed Haar measure on $H_0(F_0) = \wt H_1(F_0)$; the particular choice will not be important for us.

\begin{proof} 
We just prove the claim for the transfer factor; all of the other assertions are straightforward.  Let $y=\bigl[\begin{smallmatrix} A &  b \\ c  & d \end{smallmatrix}\bigr]\in \fks(F_0)_\rs$.  By \eqref{sign}, $\omega_\fks(y) = (-1)^{v(\det(y^i e)_{0 \leq i \leq n-1})} = (-1)^{v(\det(A^i b)_{0 \leq i \leq n-2})}$.  The claim now follows since applying $\theta$ to $y$ replaces $b$ by $\varpi\i b$.
\end{proof}

\begin{proof}[Proof of Theorem \ref{unram variant FL}(\ref{fl equiv})]
Noting that the lattice $\wt\Lambda_0 \subset \wt W_0$ defined above is self-dual, Conjecture \ref{FLconj}\eqref{FLconj lie} asserts that $\mathbf{1}_{\fks(O_{F_0})}\in  C_c^\infty(\fks)$ transfers to  $(\mathbf{1}_{\wt\fkk_0},0)\in C_c^\infty(\wt\fkg_0) \times C_c^\infty(\wt\fkg_1)$. The equivalence of this with Conjecture \ref{FL almost self-dual}(\ref{FL almost self-dual lie}) now follows immediately from Lemma \ref{lemtheta}.
\end{proof}
  
The key tool in proving part (\ref{fl imply}) of Theorem \ref{unram variant FL} will be the \emph{Cayley transform}.
For any $\xi\in F^1$ and $y \in \RM_n(F)$ with $\det(1-y) \neq 0$, the Cayley map $\fkc_\xi$ is defined by
\[
   \fkc_{\xi}(y) =  \xi \dfrac{1+y}{1-y} \in \RM_n(F).
\]
Its inverse is given by the formula
$$\fkc_\xi^{-1}(\gamma)=\frac{\gamma-\xi}{\gamma+\xi}.
$$
Note that both of these maps are equivariant for the conjugation action by $\GL_n(F)$. We introduce the following terminology.

\begin{definition}
\begin{altenumerate}
\item An element $\gamma\in S(F_0)$ is \emph{integral} if its characteristic polynomial has coefficients in $O_F$.  We make the same definition for $y\in \fks(F_0)$, for $g\in G_i(F_0)$, and for $x\in \fkg_i(F_0)$. 
\item An element $y\in \fks(F_0)$  is \emph{strongly integral} if it is integral and $\det(1-y)\in O_F^\times$. We make the same definition for $x\in \fkg_i(F_0)$.  We denote the subsets of strongly integral elements by $\fks(F_0)^\circ \subset \fks(F_0)$ and $\fkg_i(F_0)^\circ \subset \fkg_i(F_0)$.
\item Let $\xi\in F^1$. An element $\gamma\in S(F_0)$  is \emph{$\xi$-strongly integral} if it is integral and $\det(\gamma+\xi)\in O_F^\times$.  We make the same definition for $g\in G_i(F_0)$.  We denote the subsets of $\xi$-strongly integral elements by $S(F_0)_\xi^\circ \subset S(F_0)$ and $G_i(F_0)_\xi^\circ \subset G_i(F_0)$.
\end{altenumerate}
\end{definition}

Also recall the compact open subsets $K' \subset S(F_0)$ and $K_1 \subset G_1(F_0)$ from \S\ref{s:ATunr}, and define
\[
   \fkk^{\prime\circ} := \fkk' \cap \fks(F_0)^\circ,
	\quad
   K_\xi^{\prime \circ} := K' \cap S(F_0)_\xi^\circ,
	\quad
	\fkk_1^\circ := \fkk_1 \cap \fkg_1(F_0)^\circ,
	\quad
	K_{1,\xi}^\circ := K_1 \cap G_1(F_0)_\xi^\circ.
\]

\begin{lemma}\label{cayley}
Let $\xi \in F^1$.
\begin{altenumerate}
\item\label{cayley i} The Cayley transform $\fkc_\xi$ induces bijections
\[
   \fks(F_0)^\circ \isoarrow S(F_0)_\xi^\circ,
	\quad
	\fkg_0(F_0)^\circ \isoarrow G_0(F_0)_\xi^\circ,
	\quad\text{and}\quad
   \fkg_1(F_0)^\circ \isoarrow G_1(F_0)_\xi^\circ
\]
which are equivariant for the respective actions of $H'(F_0)$, $H_0(F_0)$, and $H_1(F_0)$.  These bijections respect the regular semi-simple sets on both sides, and they have the property that regular semi-simple elements $y \in \fks(F_0)^\circ$ and $x \in \fkg_i(F_0)^\circ$ match if and only if $\fkc_\xi(y) \in S(F_0)_\xi^\circ$ and $\fkc_\xi(x)\in G_i(F_0)_\xi^\circ$ match.
\item\label{cayley ii} Similarly, $\fkc_\xi$ induces bijections
\[
   \fkk^{\prime\circ} \isoarrow K_\xi^{\prime\circ}
	\quad\text{and}\quad
	\fkk_1^\circ \isoarrow K_{1,\xi}^\circ.
\]
\end{altenumerate}
\end{lemma}

\begin{proof}
The only claim that possibly requires proof is that the bijections in (\ref{cayley i}) respect the sets of regular semi-simple elements on both sides.  This follows by using that the Cayley--Hamilton theorem gives an equality of $F$-algebras $F[y] = F[\fkc_\xi(y)]$, and that an arbitrary $y \in \RM_n(F)$ is regular semi-simple if and only if the sets $\{y^i e\}_{i=0}^{n-1}$ and $\{\tensor[^t]{e}{}y^i\}_{i=0}^{n-1}$ are linearly independent; cf.\ the proofs of \cite[Lems.\ 8.7, 10.6]{RSZ}.
\end{proof}

\begin{lemma}\label{cayley orb}
Let $\xi \in F^1$.  Then for any regular semi-simple $y \in \fks(F_0)^\circ$,
\[
   \omega_\fks(y) = \omega_S\bigl(\fkc_\xi (y)\bigr)
	\quad\text{and}\quad
	\Orb(y, \mathbf{1}_{\fkk'}, s) = \Orb\bigl(\fkc_\xi(y), \mathbf{1}_{K'}, s\bigr) .
\]
Similarly, for any regular semi-simple $x \in \fkg_1(F_0)^\circ$,
\[
   \Orb(x, \mathbf{1}_{\fkk_1}) = \Orb\bigl(\fkc_\xi(x), \mathbf{1}_{K_1}\bigr).
\]
\end{lemma}

\begin{proof}
Since we take $y$ and $x$ to be strongly integral, the equalities for the orbital integrals follow from equivariance of the Cayley maps and from Lemma \ref{cayley} (especially part \eqref{cayley ii}).  The equality for the transfer factor is proved as for \cite[Lem.\ 11.9]{RSZ}. (Note that loc.~cit.\ considers the case that $F/F_0$ is ramified and $n$ is odd; when $F/F_0$ is unramified, the proof simplifies and is valid for any $n$.)
\end{proof}

Now we are ready to prove part (\ref{fl imply}) of the theorem.

\begin{proof}[Proof of Theorem \ref{unram variant FL}(\ref{fl imply})]
We show that part (\ref{FL almost self-dual lie}) of Conjecture \ref{FL almost self-dual} implies part (\ref{FL almost self-dual inhomog}); that this also implies (\ref{FL almost self-dual homog}) will follow from Theorem \ref{unram variant FL}(\ref{fl equiv 2}).  We must show that for all matching $\gamma \in S(F_0)_\rs$ and $g \in G_i(F_0)_\rs$,
\begin{align*}
   \omega_S(\gamma) \Orb\bigl(\gamma, (-1)^{n-1}\mathbf{1}_{K'}\bigr)=\begin{cases}0,&  g\in G_0(F_0)_\rs;\\
\Orb(g, \mathbf{1}_{K_1}),& g\in G_1(F_0)_\rs.
\end{cases}
\end{align*}

First assume that $\gamma$ is not integral. Then $\Orb(\gamma, \mathbf{1}_{K'}) = 0$, and we have to show that $\Orb(g,\mathbf{1}_{K_1}) = 0$ if $g \in G_1(F_0)$.  But since $\gamma$ and $g$ have identical characteristic polynomials, $g$ is also not integral, and hence $\Orb(g,\mathbf{1}_{K_1})$ vanishes.

Now assume that $\gamma$ is integral. Since we assume that $q\geq n$, there exist at least $n+1$ elements $\xi_0,\xi_1,\dotsc,\xi_n\in F^1$ with pairwise distinct residues mod $\varpi$ (since the kernel of the norm map $\BF_{q^2}^\times \to \BF_q^\times$ has $q+1$ elements). Hence there exists some $\xi = \xi_i$ such that $\det(\gamma+\xi)\in O_F^\times$, i.e.\ $\gamma$ is $\xi$-strongly integral. Since $\gamma$ and $g$ have the same characteristic polynomial, $g$ is $\xi$-strongly integral too.  Let $y := \fkc_\xi^{-1}(\gamma) \in \fks(F_0)^\circ$ and $x := \fkc_\xi^{-1}(g) \in \fkg_i(F_0)^\circ$.  Then $y$ and $x$ are matching regular semi-simple elements, and by Lemma \ref{cayley orb} and Conjecture \ref{FL almost self-dual}(\ref{FL almost self-dual lie}),
\begin{multline*}
   \omega_S(\gamma)\Orb\bigl(\gamma, (-1)^{n-1}\mathbf{1}_{K'}\bigr) = \omega_\fks(y)\Orb\bigl(y, (-1)^{n-1}\mathbf{1}_{\fkk'}\bigr)\\
	= 
	\begin{cases}
		0,  &  g \in G_0(F_0)_\rs;\\
		\Orb(x, \mathbf{1}_{\fkk_1}) = \Orb(g,\mathbf{1}_{K_1}),  &  g \in G_1(F_0)_\rs,
	\end{cases}
\end{multline*}
as desired.
\end{proof}

\begin{remark}
The idea to reduce the group statement to the Lie algebra statement via the Cayley transform also appears in \cite[Lem.\ 8.4, Lem.\ 11.1, Prop.\ 11.14]{RSZ}, and  in fact our situation is simpler than in loc.~cit.
If we assume that $q\geq n+2$, then we may show the converse to (\ref{fl imply}), i.e.\ the group version also implies the Lie algebra version, cf.\ \cite[Prop.\ 2.4]{M-AFL}.
\end{remark}

It remains to prove Theorem \ref{unram variant FL}(\ref{fl equiv 2}), which we will do by relating functions and orbital integrals on $G'(F_0)$ and $S(F_0)$ as in \cite[\S2.1]{Z14} and \cite[\S5.2]{RSZ}. 
We recall the map $r(g) = g \ov g\i$ from \eqref{Res GL_n -> S}, and we continue to normalize the Haar measures as in \S\ref{s:FL}.  Also recall that since we are in the unramified case, $\wt\eta$ is the natural extension $\wt\eta(x) = (-1)^{v(x)}$ of $\eta$ to $F^\times$.  For $g \in \GL_n(F)$, we write $\wt\eta(g) := \wt\eta(\det g)$.

\begin{lemma}\label{homog to inhomog}
Let $f'\in  C^\infty_c(G')$, and define the function $\wt f'$ on $S(F_0)$ by, for $g \in \GL_n(F)$,
\begin{equation*}
   \wt f'\bigl(r(g)\bigr) := \int_{\GL_{n-1}(F)\times \GL_n(F_0)}f'(h_1\i,h_1\i g h_2) \wt\eta(gh_2)^{n-1}\, dh_1\,dh_2.
\end{equation*}
\begin{altenumerate}
\item\label{all}
$\wt f'\in C^\infty_c(S)$, and every element in $C^\infty_c(S)$ arises in this way.  
\item\label{new}
We have
\[
   \frac{q^n-1}{q-1} \cdot \wt{\mathbf{1}}_{\GL_{n-1}(O_F) \times K_0(\varpi)} = \mathbf{1}_{K'}.
\]
\item\label{old}
For all $\gamma = (\gamma_1, \gamma_2)\in G'(F_0)_\rs$,
\begin{equation*}
   \Orb\bigl(r(\gamma_1\i\gamma_2),\wt f'\bigr) = \wt\eta(\gamma_1^{-1}\gamma_2)^{n-1}\Orb(\gamma,f').
\end{equation*}
\item\label{old 2}
If $f'$ transfers to $(0,f_1)$ for some $f_1\in  C^\infty_c(G_{W_1})$, then there exists a function $\phi_\corr' \in C_c^\infty(G')$ such that for any $\gamma$ matching an element in $G_{W_0}(F_0)_{\rs}$,
\[
   2 \del\bigl(r(\gamma_1\i\gamma_2),\wt f'\bigr) = \wt\eta(\gamma_1^{-1}\gamma_2)^{n-1} \bigl(\del(\gamma,f') +  \Orb(\gamma,\phi'_{\corr})\bigr).
\]
If furthermore the support of $f'$ is contained in the set $\{(\gamma_1,\gamma_2)\in G'(F_0) \mid \lv\det \gamma_1\rv = 1 \}$, then we may take $\phi_\corr' = 0$.
% Suppose that $f'$ transfers to $(0,f_1)$ for some $f_1\in  C^\infty_c(G_{W_1})$, and that the support of $f'$ is contained in $\{(\gamma_1,\gamma_2)\in G'(F_0) \mid \lv\det \gamma_1\rv = 1 \}$. Then for any $\gamma$ matching an element in $G_{W_0}(F_0)_{\rs}$,
% \begin{equation*}
%    \del\bigl(r(\gamma_1\i\gamma_2),\wt f'\bigr) = \wt\eta(\gamma_1^{-1}\gamma_2)^{n-1} \del(\gamma,f').
% \end{equation*}
\end{altenumerate}
\end{lemma}

\begin{proof}
Part (\ref{all}) is clear, and parts (\ref{old}) and (\ref{old 2}) are proved as in \cite[Lem.\ 5.7]{RSZ}.\footnote{Note that in loc.\ cit.\ $n$ is odd, which in the case of our $\wt\eta$ implies that $\wt\eta(\gamma_1^{-1}\gamma_2)^{n-1} = 1$; however the proof for arbitrary $n$ is analogous. The last statement in (\ref{old 2}) is immediate from the specific choice of $\phi_\corr'$ given in (5.13) in loc.\ cit.}  We prove (\ref{new}).  For $g \in \GL_n(F)$, we have
\begin{align}
   &\wt{\mathbf{1}}_{\GL_{n-1}(O_F) \times K_0(\varpi)}\bigl(r(g)\bigr)\notag\\ 
	   &\qquad\qquad\qquad = \int_{\GL_{n-1}(F)\times \GL_n(F_0)} \mathbf{1}_{\GL_{n-1}(O_F) \times K_0(\varpi)}(h_1\i,h_1\i g h_2) \wt\eta(gh_2)^{n-1}\, dh_1\,dh_2 \notag\\
		&\qquad\qquad\qquad= \int_{\GL_{n-1}(O_F)\times \GL_n(F_0)} \mathbf{1}_{K_0(\varpi)} (h_1\i g h_2) \wt\eta(gh_2)^{n-1}\, dh_1\,dh_2\notag\\
		&\qquad\qquad\qquad= \int_{\GL_n(F_0)} \mathbf{1}_{K_0(\varpi)} (g h_2) \wt\eta(gh_2)^{n-1}\, dh_2,\label{last}
\end{align}
where the last equality holds because $\GL_{n-1}(O_F) \subset K_0(\varpi)$.  We claim that
\begin{equation}\label{blah}
   r(g) \in K' \iff g \GL_n(F_0) \cap K_0(\varpi) \neq \emptyset.
\end{equation}
Before establishing this claim, let us show that it implies the conclusion of (\ref{new}).  If $r(g) \notin K'$, then $\wt{\mathbf{1}}_{\GL_{n-1}(O_F) \times K_0(\varpi)}(r(g)) = 0$ because, by (\ref{blah}), the integrand in (\ref{last}) is identically zero.  If $r(g) \in K'$, then by (\ref{blah}) we may assume that $g \in K_0(\varpi)$.  Since $\wt\eta$ is identically one on $K_0(\varpi)$, we conclude from (\ref{last}) that
\begin{multline*}
   \wt{\mathbf{1}}_{\GL_{n-1}(O_F) \times K_0(\varpi)}\bigl(r(g)\bigr) 
	   = \vol\bigl(\GL_n(F_0) \cap K_0(\varpi)\bigr)\\
	   = \frac 1{[\GL_n(O_{F_0}): \GL_n(O_{F_0})\cap K_0(\varpi)]}
		= \frac 1{\frac{q^n-1}{q-1}},
\end{multline*}
where in the last equality we use that the index in question equals $\#\BP_k^{n-1}(k)$ (note that $\GL_n(O_{F_0})$ acts transitively on the lines in $k^n$, and $\GL_n(O_{F_0})\cap K_0(\varpi)$ is the stabilizer of a line).  This proves (\ref{new}).

It remains to establish the equivalence (\ref{blah}).  The reverse implication is trivial.  To prove the forward implication, let $\Lambda := O_F^n$ and $\Lambda' := O_F^{n-1} \oplus \varpi\i O_F$.  Then $K_0(\varpi)$ is the stabilizer in $\GL_n(F)$ of the lattice chain $\Lambda \subset \Lambda'$.  Since $r(g) \in K'$, we have $\ov g\i \Lambda = g\i \Lambda$ and $\ov g\i \Lambda'= g\i \Lambda'$.  Hence these are Galois-stable lattices in $F^n$, so that they come from an $O_{F_0}$-lattice chain in $F_0$. Hence there exists $h \in \GL_n(F_0)$ such that $h \cdot (\Lambda \subset \Lambda') = g\i \cdot (\Lambda \subset \Lambda')$.  Hence $gh \in K_0(\varpi)$, as desired.
\end{proof}

\begin{proof}[Proof of Theorem \ref{unram variant FL}(\ref{fl equiv 2})]
By Lemma \ref{homog to inhomog}(\ref{new})(\ref{old}) and the definition of the transfer factors in \S\ref{trans factor},
\[
   \omega_{G'}(\gamma) \Orb\biggl(\gamma,\frac{q^n-1}{q-1}\cdot\mathbf{1}_{\GL_{n-1}(O_F)\times K_0(\varpi)}\biggr) = \omega_S\bigl(r(\gamma_1\i\gamma_2)\bigr) \Orb\bigl(r(\gamma_1\i\gamma_2), \mathbf{1}_{K'} \bigr)
\]
for all $\gamma = (\gamma_1,\gamma_2) \in G'(F_0)_\rs$.  On the other hand, it is easy to verify that
\[
   \Orb(g,\mathbf{1}_{K_1^\flat \times K_1}) = \Orb(g_1\i g_2,\mathbf{1}_{K_1})
\]
for all $g = (g_1,g_2) \in G_{W_1}(F_0)_\rs$.  Since the maps
\begin{equation}\label{maps}
   \begin{gathered}
		\xymatrix@R=0ex{
		   G'(F_0)_\rs \ar[r]  &  S(F_0)_\rs\\
			(\gamma_1,\gamma_2) \ar@{|->}[r]  &  r(\gamma_1\i \gamma_2)
		}
	\end{gathered}
	\qquad\text{and}\qquad
   \begin{gathered}
		\xymatrix@R=0ex{
		   G_{W_i}(F_0)_\rs \ar[r]  &  G_i(F_0)_\rs\\
			(g_1,g_2) \ar@{|->}[r]  &  g_1\i g_2
		}
	\end{gathered}
\end{equation}
are compatible with the matching relation, the theorem follows.
\end{proof}

Let us finally note that, in analogy with \cite[Lem.\ 2.5]{Z12} in the self-dual case, the transfer relations asserted in the FL Conjecture \ref{FL almost self-dual} are easy to verify in the case that the unitary matching elements are
attached to the split hermitian space.  For any $y \in \RM_n(F)$, define
\[
   \tensor*[^*]{y}{} := \diag(1,\dotsc,1,\varpi\i) \tensor[^t]{y}{} \diag(1,\dotsc,1,\varpi).
\]

\begin{lemma}
\begin{altenumerate}
\item\label{matching 0 homog}
Suppose that $f' \in C_c^\infty(G')$ satisfies $f'(\gamma_1,\gamma_2) = f'(\tensor*[^t]{\ov\gamma}{_1^{-1}},\tensor*[^*]{\ov\gamma}{_2^{-1}})$.  Then for all $\gamma \in G'(F_0)_\rs$ matching an element $g \in G_{W_0}(F_0)_\rs$,
\[
   \Orb(\gamma,f') = 0.
\]
In particular, $\Orb(\gamma,\mathbf{1}_{\GL_{n-1}(O_F) \times K_0(\varpi)}) = 0$ for such $\gamma$.
\item\label{matching 0 inhomog}
Suppose that $f' \in C_c^\infty(S)$ satisfies $f'(\gamma) = f'(\tensor*[^*]{\gamma}{})$.  Then for all $\gamma \in S(F_0)_\rs$ matching an element $g \in G_0(F_0)_\rs$,
\[
   \Orb(\gamma,f') = 0.
\]
In particular, $\Orb(\gamma,\mathbf{1}_{K'}) = 0$ for such $\gamma$.
\item\label{matching 0 lie}
Suppose that $\phi' \in C_c^\infty(\fks)$ satisfies $\phi'(y) = \phi'(\tensor*[^*]{y}{})$.  Then for all $y \in \fks(F_0)_\rs$ matching an element $x \in \fkg_0(F_0)_\rs$,
\[
   \Orb(y,\phi') = 0.
\]
In particular, $\Orb(\gamma,\mathbf{1}_{\fkk'}) = 0$ for such $y$.
\end{altenumerate}\label{matching 0}
\end{lemma}

\begin{proof}
The proofs of (\ref{matching 0 inhomog}) and (\ref{matching 0 lie}) are virtually identical to that of \cite[Lem.\ 2.5]{Z12}.  Part (\ref{matching 0 homog}) then follows from (\ref{matching 0 inhomog}) by the easily verified fact that if $f'$ satisfies the hypothesis in (\ref{matching 0 homog}), then $\wt f'$ satisfies the hypothesis in (\ref{matching 0 inhomog}); by the fact that the maps in \eqref{maps} respect the matching relation; and by Lemma \ref{homog to inhomog}(\ref{old}).
\end{proof}

\subsection{On the AT and AFL Conjecture \ref{conjunramified}}\label{results atc afl}
In this subsection we prove results related to Conjecture \ref{conjunramified}.  We retain the notation of the previous subsection.

\begin{theorem}\label{AFLtoAT1} 
The AFL identity in Conjecture \ref{conjunramified}(\ref{conjunramified lie})
holds true for an element $x\in \fkg_0(F_0)_\rs$ if and only if Conjecture \ref{AFLconj}(\ref{AFLconj lie}) holds true for $\theta(x)\in\fkg_1(F_0)_\rs$.
\end{theorem}

\begin{proof} Forgetting the polarizations, the framing objects $\BX_n$ of $\CN_n$ and $\wt\BX_n$ of $\wt\CN_n$ are identical, as are the objects $\ov\CE$ and $\ov\CE'$ used to define the respective embeddings $\delta_\CN\colon \CN_{n-1} \to \CN_n$ and $\wt\delta_\CN\colon \CN_{n-1} \to \wt\CN_n$. Hence for any $x\in\End_{O_F}^\circ(\wt\BX_n) = \End_{O_F}^\circ (\BX_n)$, we have an equality of closed formal subschemes of $\CN_{n-1}$, 
\[
   \wt\Delta\cap\wt\Delta_x=\Delta\cap\Delta_x . 
\]

Now let $x\in \fkg_0(F_0)$. We claim that 
\begin{equation*}
   \Delta\cap\Delta_x =\Delta\cap\Delta_{\theta(x)} .  
\end{equation*}
Indeed, with respect to the product decomposition $\wt\BX_n = \BX_n = \BX_{n-1} \times \ov \BE$, write
\[
   x = 
	\begin{bmatrix}
		 A &  b  \\
      c  & d
   \end{bmatrix},
\]  
where $A\in \End^\circ_{O_F}(\BX_{n-1})$, $b\in \Hom^\circ_{O_F}(\ov \BE,\BX_{n-1})$, $c \in \Hom_{O_F}^\circ(\BX_{n-1}, \ov\BE)$, and $d\in \End^\circ_{O_F}(\ov \BE)$.  Since $x + \Ros_{\lambda_{\wt\BX_n}}(x) = 0$, we have $b = - \varpi \tensor[^t]{c}{}$, where $\tensor[^t]{c}{} := \lambda_{\BX_{n-1}}\i \circ c^\vee \circ \lambda_{\ov\BE}$.  Hence
\[
   \theta(x) =
	\begin{bmatrix}
		A  &  -\tensor[^t]{c}{}\\
		c  &  d
	\end{bmatrix}.
\]
Given a point $Y$ on $\CN_{n-1}$, we see by inspection that $x$ lifts to an endomorphism of $Y \times \ov\CE$ if and only if $\theta(x)$ does, which proves the claim.
  
Of course, it follows from the claim that $\wt\Delta \cap \wt\Delta_x$ is artinian if and only if $\Delta \cap \Delta_{\theta(x)}$ is, and when they are,
$$
   \lInt(x)=\lInt\bigl(\theta(x)\bigr).
$$
On the orbital integral side, it follows from Lemma \ref{lemtheta}(\ref{lemtheta i}) that
\[
   \omega_\fks(y) \del\bigl(y, (-1)^{n-1}\mathbf{1}_{\fkk'}\bigr) = \omega_\fks\bigl(\theta(y)\bigr)\del\bigl(\theta(y), \mathbf{1}_{\fks(O_{F_0})}\bigr),\quad  y\in \fks(F_0)_\rs .
\]
This completes the proof.
\end{proof}

\begin{theorem} Assume $q\geq n$.
Then the AFL identities in Conjecture \ref{conjunramified}(\ref{conjunramified homog})(\ref{conjunramified inhomog}) hold true for all $g\in G_{W_0}(F_0)_\rs$, resp.\ $g\in G_0(F_0)_\rs$, for which the intersection is non-degenerate, provided that the AFL identity in
Conjecture \ref{conjunramified}(\ref{conjunramified lie})
holds true.
\end{theorem}

\begin{proof}
First note that homogeneous AFL identity reduces to the inhomogeneous one via Lemma \ref{homog to inhomog}(\ref{new})(\ref{old 2}), Lemma \ref{matching 0}(\ref{matching 0 homog}), and the easy fact that $\Int(g) = \Int(g_1\i g_2)$ for any $g = (g_1,g_2) \in G_{W_0}(F_0)_\rs$.  Thus we show that the inhomogeneous identity holds.	
The proof again uses Cayley maps
(cf.\ also the proof of \cite[Lem.\ 2.2]{M-AFL}). Let $g\in G_0(F_0)_\rs$ be such that the intersection is non-degenerate. Let $\gamma\in S(F_0)_\rs$ be a matching element. We need to show that
\begin{equation}\label{desired}
   \omega_S(\gamma) \del\bigl(\gamma, (-1)^{n-1}\mathbf{1}_{K'}\bigr)=-\Int(g)\cdot \log q.
\end{equation}

As in the proof of Theorem \ref{unram variant FL}(\ref{fl imply}), we consider cases based on the integrality of $\gamma$. If $\gamma$ is not integral, then the left-hand side of \eqref{desired} is obviously zero. But then $g$ is also not integral, and we claim this forces the intersection $\wt\Delta\cap (1\times g)\wt\Delta$ to be empty, so that the right-hand side of \eqref{desired} is also zero.  Indeed, if instead this intersection were nonempty, then it would contain a $\ov k$-point, and $g$ would stabilize the corresponding Dieudonn\'e module.  Hence the characteristic polynomial of $g$ would lie in $O_{\breve F}[T]$, contrary to the non-integrality of $g$.

Now assume that $\gamma$ is integral. Since $q\geq n$,  $\gamma$ is $\xi$-strongly integral for some $\xi\in F^1$, and hence so is the matching element $g$. Let $y:=\fkc_\xi^{-1}(\gamma) \in \fks(F_0)^\circ$ and $x:= \fkc_\xi\i(g) \in \fkg_0(F_0)^\circ$. Then $y$ and $x$ also match. By Lemma \ref{cayley orb},
$$
 \omega_\fks(y)\del(y, \mathbf{1}_{\fkk'})=\omega_S(\gamma)\del(\gamma, \mathbf{1}_{K'}) .
$$
To complete the proof, since we assume the AFL identity in Conjecture \ref{conjunramified}(\ref{conjunramified lie}), it suffices to show that $\wt\Delta \cap \wt\Delta_x$ is artinian if and only if $\wt\Delta \cap (1\times g)\wt\Delta = \wt\Delta \cap \wt\Delta_g$ is, and when they are, that
\begin{align}\label{idcayint}
   \lInt(x)=\Int(g).
\end{align}
Both of these statements follow from the equality 
of subschemes of $\wt\Delta$, for any strongly integral $x$ and $g=\fkc_\xi(x)$,
\begin{equation}\label{eqn int cayley u}
   \wt\Delta\cap\wt \Delta_x = \wt\Delta\cap\wt \Delta_g.
\end{equation}
Indeed, this implies \eqref{idcayint} because, when $\wt\Delta\cap\wt \Delta_g$ is artinian, 
its length is equal to $\Int(g)$, cf.\ \cite[Prop.\ 8.10]{RSZ}. Now, the two sides \eqref{eqn int cayley u} are the loci of points $Y$ in $\CN_{n-1}$ where, respectively, $x\colon \wt\BX_n \rightarrow \wt\BX_n$ and $g\colon \wt\BX_n \to \wt\BX_n$ lift to a homomorphism $Y \times \ov\CE \rightarrow Y\times \ov\CE$.  But by the Cayley--Hamilton theorem, for strongly integral $x$ we have $O_F[x] = O_F[\fkc_\xi(x)]$ as $O_F$-subalgebras of $\End_{O_F}^\circ(\wt\BX_n)$. Hence $x$ lifts if and only if $g$ does. This establishes \eqref{eqn int cayley u} and completes the proof.
\end{proof}

\begin{corollary}
Let $F_0 = \BQ_p$. Then Conjecture \ref{conjunramified} holds when   $n=2$ or $3$.
\end{corollary}
\begin{proof}
 Let us first consider the AFL identities in  Conjecture \ref{conjunramified}. In these cases the intersection is automatically non-degenerate.  Furthermore, $q\geq 3\geq n$. Thus by the two preceding theorems, these identities follow from the Lie algebra AFL Conjecture \ref{AFLconj}(\ref{AFLconj lie}) for $n = 2$ and $3$. This is proved in \cite{M-AFL}.
(When $q\geq 5$, the Lie algebra AFL for $n=3$ can also be deduced from the group version proved in \cite{Z12}; see \cite[Lem.\ 2.2]{M-AFL}.)

Now we consider the rest of Conjecture \ref{conjunramified}. By the same proof as for \cite[Lem.\ 5.18]{RSZ}, this follows from the density principle for our weighted orbital integrals, i.e.\ the analog in the unramified case of \cite[Conj.\ 5.16]{RSZ}. This conjectural density principle is known to hold for $n=2$ or $3$, cf.\ \cite[Th.\ 1.1]{Z12b} and \cite[Rem.\ 5.17]{RSZ}. 
\end{proof}

\begin{remark}
The AFL in the self-dual case (Conjecture \ref{AFLconj}) is known for arbitrary $n$ and $F_0 = \BQ_p$, provided that $g$ is minuscule in the sense of \cite{RTZ}, and $p\geq n/2+1$.
If $p \geq n+2$, then we may deduce from this the AFL identity in Conjecture \ref{conjunramified}(\ref{conjunramified inhomog}) in the almost self-dual case for certain $g$, by applying \cite[Th.\ 2.5]{M-AFL} and then passing through the Lie algebras as above.
\end{remark}

\section{On the ATC for $F/F_0$ ramified, $n=2$}\label{even ram proof}

In this section we prove Conjecture \ref{conjram v3} when $n=2$.

\subsection{The groups}\label{the groups}
On the symmetric space $S(F_0) = S_2(F_0)$, we write an element as
$$\gamma=
\begin{bmatrix}
		a &  b \\
         c  &  d
      \end{bmatrix}\in S(F_0).
$$
Then $\gamma$ is regular semi-simple if and only if $bc\neq 0$, in which case we may write $\gamma$ as
\begin{equation}\label{gamma(a,b)}
\begin{aligned}
\gamma=\gamma(a,b) &:=\begin{bmatrix}
		a &  b \\
         (1-\RN a)/\ov b &  -\ov ab/\ov b
      \end{bmatrix}\\
		&\phantom{:}=\begin{bmatrix}
		1&   \\
          &  -b/\ov b
      \end{bmatrix} \begin{bmatrix}
		a &  b \\
        - (1-\RN a)/ b &  \ov a
      \end{bmatrix} \in S(F_0)_\rs,
		\quad a \in F \smallsetminus F^1,\ b \in F^\times.
\end{aligned}
\end{equation}
We define the set of semi-simple but irregular elements
\[
A_{S} :=
	   \biggl\{\, 
		\begin{bmatrix}
		   a  &    \\
           &  d
      \end{bmatrix} 
		\in S(F_0) \biggm|  a,d \in F^1 \,\biggr\}.
\]
Similarly, in the ``Lie algebra'' $\fks(F_0) = \fks_2(F_0)$ we write an element as
\begin{equation}\label{y(a,b,c,d)}
   y=y(a, b, c, d)=\begin{bmatrix}
		   a  & b   \\
        c   &  d
      \end{bmatrix} \in \fks(F_0), \quad a,b,c,d\in \pi F_0.
\end{equation}
The set of semi-simple but irregular elements is 
\begin{align*}
\fka_{\fks} :=
	   \biggl\{\,
		\begin{bmatrix}
		   a  &    \\
		      &  d
		\end{bmatrix}
		\in \fks(F_0) \biggm|  a,d \in \pi F_0 \,\biggr\}.
\end{align*}

On the unitary side, we first look at objects attached to the split hermitian space.  Recall from \S\ref{s:ATCeven} that we have
\[
   W_0=W^\flat_0\oplus F u_0,\quad (u_0,u_0)=-1.
\]
By \eqref{chi decomp formula} the space $W_0^\flat$ is also split, and we choose a basis vector $u^\flat \in W_0^\flat$ with $(u^\flat, u^\flat) = 1$.  This choice determines a special embedding $G_0(F_0) = \U(W_0)(F_0) \inj \GL_2(F)$ as in \S\ref{setup inhomog}, which realizes $G_0$ as the unitary group associated to the diagonal hermitian matrix $\diag(1,-1)$. In this way, the set of irregular elements in $G_0(F_0)$ is 
\begin{equation}\label{eqn U_2 A}
A_{G_0} :=
	   \biggl\{\, 
		\begin{bmatrix}
		   a  &    \\
           &  d
      \end{bmatrix}
		\in \RM_2(F) \biggm| a,d \in F^1  \,\biggr\}.
\end{equation}
These elements are all semi-simple since the group $H_0 = \U(W_0^\flat)$ is anisotropic. Similarly, we embed the Lie algebra $\fkg_0(F_0) \inj \RM_2(F)$, and we denote its set of irregular semi-simple elements by 
\begin{align*}
\fka_{\fkg_0} :=
	   \biggl\{\, 
		\begin{bmatrix}
		   a  &    \\
		      &  d
		\end{bmatrix}
		\in \RM_2(F) \biggm|  a,d \in \pi F_0  \,\biggr\}.
\end{align*}

Now recall the lattice
$\Lambda_0^\nat=\Lambda_0^\flat\oplus O_F u_0$ from \eqref{Lambda^nat}, which is self-dual since $n = 2$. The two $\pi$-modular lattices lying between $\pi\Lambda_0^\nat$ and $\Lambda_0^\nat$ are given by
$$
\Lambda_0^\pm=\pi\Lambda_0^\nat+O_F\, (u_0\pm u^\flat).
$$
They are both stable under elements of $K_0^\flat = F^1$ with reduction $1 \bmod \pi$, and they are permuted by elements with reduction $-1 \bmod \pi$. Now let $\Lambda_0$ be either of them, and let $K_0 \subset G_0(F_0)$ and $\fkk_0 \subset \fkg_0(F_0)$ be the respective stabilizers of $\Lambda_0$. A subtle point is that 
$$
A_{G_0}\cap K_0=  \biggl\{\,  \begin{bmatrix}
		   a  &    \\
           &  d
      \end{bmatrix}\in G_0(F_0)\biggm|   a\equiv 1\bmod \pi   \,\biggr\}
$$
is a subgroup of $A_{G_0}$ of index $2$.  However, we do have 
\begin{equation}\label{A_G_0 intersect}
   A_{G_0}\cap K_0^\flat K_0= A_{G_0}.
\end{equation}
In the Lie algebra $\fkg_0(F_0)$, we have 
\begin{equation}\label{a_g_0 intersect}
   \fka_{\fkg_0} \cap \fkk_0= \biggl\{\, 
	\begin{bmatrix}
	   a  &    \\
	      &  d
	\end{bmatrix}
	\in \fkg_0(F_0) \biggm| a,d \in \pi O_{F_0}  \,\biggr\}.
\end{equation}

We now look at the non-quasi-split unitary group $G_1$. Recall from \eqref{u} that the hermitian space
\[
   W_1=\BV(\BX_2) \cong \BV\bigr(\wt\BX_2\bigl) 
	               = \Hom_{O_F}^\circ\bigl(\ov\BE,\BX_1 \times \ov\BX_1\bigr)
						= \BV(\BX_1)\oplus \BV\bigl(\ov \BX_1\bigr)
\]
has a special vector $u=(0,\id_{\ov \BE})$ with norm $-1$. Note that the rightmost space in the display is canonically $D=D^-\oplus F$ (as an $F_0$-vector space) with hermitian norm given by $(v,v)=-\RN v$ for $v\in D$, and the special vector $u$ corresponds to $1\in D$.
We identify 
$$
   \End_{O_{F_0}}^\circ \bigl(\wt\BX_2\bigr) = \End_{O_{F_0}}^\circ\bigl(\BE^2\bigr)=\RM_2(D),
$$ 
and the $O_F$-action is given by $\pi\mapsto \diag(\pi,-\pi)\in \RM_2(D)$. Then the $O_F$-linear quasi-endomorphism algebra is 
$$
\End_{O_F}^\circ\bigl(\wt\BX_2\bigr)= \biggl\{\,  \begin{bmatrix}
		   a  & b    \\
       c    &  d
      \end{bmatrix} \in \RM_2(D) \biggm|   a,d\in F,\ b,c\in D^-  \,\biggr\} .
$$
The Rosati involution is given by
$$
   \Ros_{\lambda_{\wt\BX_2}}(x)=x^\dagger= \tensor[^t]{\ov x}{},
$$
where $x\mapsto \ov x$ is the entry-wise main involution on $\RM_2(D)$.  We then have the unitary group 
$$
   G_1(F_0) = \bigl\{\, x\in \End_{O_F}^\circ\bigl(\wt\BX_2\bigr) \bigm|  x^\dagger x=1\,\bigr\} ,
$$
and its Lie algebra 
$$
\fkg_1(F_0)=\bigl\{\, x\in \End_{O_F}^\circ\bigl(\wt\BX_2\bigr) \bigm|  x^\dagger + x=0 \,\bigr\} .
$$
They may be explicitly presented as
\begin{equation}\label{G_1 pres}
   G_1(F_0)=\biggl\{\,  \begin{bmatrix}
		   1 &     \\
         &  \alpha
      \end{bmatrix}    \begin{bmatrix}
		   a  & b    \\
       b   &  a
      \end{bmatrix} \in \RM_2(D)  \biggm|  a\in F,\ b\in D^-,\ \RN a+\RN b=1,\ \alpha \in F^1 \,\biggr\}
\end{equation}
and 
\begin{equation}\label{fkg_1 pres}
\fkg_1(F_0)=\biggl\{\, \begin{bmatrix}
		   a  & b    \\
       b  &  d
      \end{bmatrix} \in \RM_2(D)  \biggm|  a,d\in \pi F_0,\ b\in D^-   \,\biggr\} .
\end{equation}
If we fix a basis element $\zeta \in D^-$, then we may also express these presentations in terms of special embeddings into $\RM_2(F)$,
\[
   \begin{bmatrix}
		   1 &     \\
         &  \alpha
   \end{bmatrix}
	\begin{bmatrix}
		   a  & b    \\
       b   &  a
   \end{bmatrix} \in G_1(F_0) \subset \RM_2(D)
	\quad\text{\emph{identifies with}}\quad
	\begin{bmatrix}
		   a  & b \zeta\i    \\
       \ov{\alpha b \zeta}  &  \ov{\alpha a}
   \end{bmatrix} \in \RM_2(F),
\]
and
\[
	\begin{bmatrix}
		 a  &  b    \\
       b  &  d
   \end{bmatrix} \in \fkg_1(F_0) \subset \RM_2(D)
	\quad\text{\emph{identifies with}}\quad
	\begin{bmatrix}
		 a  &  b \zeta\i   \\
       \ov{b \zeta}  &  -d
   \end{bmatrix} \in \RM_2(F).
\]
Similarly to before, the irregular semi-simple elements in $G_1(F_0)$ are given by the diagonal elements
\[
   A_{G_1} :=  \biggl\{\, 
		\begin{bmatrix}
		   a  &    \\
           &  d
      \end{bmatrix}
		\in \RM_2(D) \biggm| a,d \in F^1  \,\biggr\}
	=
	\biggl\{\, 
	\begin{bmatrix}
	   a  &    \\
        &  \ov d
   \end{bmatrix}
	\in \RM_2(F) \biggm| a,d \in F^1  \,\biggr\},
\]
and the irregular semi-simple elements in $\fkg_1(F_0)$ are given by
\[
   \fka_{\fkg_1} := \biggl\{\,
	   \begin{bmatrix}
		   a  &     \\
            &  d
      \end{bmatrix} \in \RM_2(D)  \biggm|  a,d\in \pi F_0  \,\biggr\} 
		= \biggl\{\,
	   \begin{bmatrix}
		   a  &     \\
            &  - d
      \end{bmatrix} \in \RM_2(F)  \biggm|  a,d\in \pi F_0  \,\biggr\}.
\]

\subsection{Orbit matching}\label{orbit matching}
Let us now indicate how regular semi-simple orbits match in terms of the presentations just given.  Recall from \cite[Prop.\ 6.2]{RS} that in the case $n = 2$, two regular semi-simple elements in $\RM_2(F)$ (cf.\ \S\ref{setup inhomog}) are $\GL_1(F)$-conjugate if and only if their diagonal entries are the same and the products of their off-diagonal entries are the same.  In the group setting, it is easy to verify from this that elements $\gamma\in S(F_0)_\rs$ and $g\in G_i(F_0)_{\rs}$ match if and only if, after regarding $g$ as an element of $\GL_2(F)$ via a special embedding, 
$
\det\gamma = \det g
$ and the upper-left entries of these matrices are equal.  Let $S_{\rs,i}$ denote the subset of $S(F_0)_\rs$ of elements matching with elements in $G_i(F_0)_\rs$.  It is easy to see that
\begin{equation}\label{S_rs,i cond}
   \gamma(a,b)\in S_{\rs,i} \iff \eta( \RN a - 1)=(-1)^i.
\end{equation}
Furthermore, in terms of the presentation \eqref{G_1 pres},
$$
   \gamma(a,b)\in S_{\rs, 1} \quad \text{\emph{matches}} \quad g=\begin{bmatrix}
		   1 &     \\
         &  \alpha
      \end{bmatrix}    \begin{bmatrix}
		   a'  & b'    \\
       b'   &  a'
      \end{bmatrix}\in G_1(F_0)_\rs
$$
if and only if $a=a'$ and $-b/\ov b=\det \gamma(a,b)=\det g = \ov\alpha$.

In the Lie algebra setting, we analogously denote by $\fks_{\rs,i}$ the subset of $\fks(F_0)_\rs$ of elements matching with elements in $\fkg_i(F_0)_\rs$. It is easy to see that
\[
   y(a,b,c,d)\in \fks_{\rs,i} \iff bc \neq 0 \text{ and } \eta(bc)=(-1)^i.
\]
In terms of the presentation \eqref{fkg_1 pres},
$$
   y(a,b,c,d) \in \fks_{\rs, 1}  \quad \text{\emph{matches}} \quad x= \begin{bmatrix}
		   a'  & b'    \\
       b'   &  d'
      \end{bmatrix}\in \fkg_1(F_0)_\rs
$$
if and only if  $a=a'$, $d= -d'$, and $bc=-\RN b'$.

\subsection{Harmonic analysis}\label{harm}
The harmonic analysis on $S(F_0)$ is done in Mihatsch's article \cite{M-ATC}. We briefly recall the results in loc.~cit.

The orbital integral for $\gamma = \gamma(a,b) \in S(F_0)_\rs$ and a function $f'\in C_c^\infty(S)$ is given by
$$
\Orb(\gamma,f',s)=\bigintssss_{F_0^\times} f'\biggl(\begin{bmatrix}
		   a  & b/x   \\
      x(1-\RN a)/\ov b &  -\ov ab/\ov b
      \end{bmatrix}\biggr) \eta(x)|x|^s \,dx.
$$
The transfer factor takes the form
\begin{equation}\label{omega_S n=2}
   \omega_S(\gamma)=\wt\eta\bigl(\ov b\bigr) = \wt\eta(b)\i,
\end{equation}
cf.\ \S\ref{trans factor}.  Similarly, the orbital integral for $f_i\in C_c^\infty(G_i)$, $i\in\{0,1\}$, is given by
$$
\Orb(g,f_i)=\int_{H_i(F_0)} f_i(h^{-1}gh ) \, dh.
$$
This is well-defined for all $g$, since $H_i(F_0) = F^1$ is compact. Recall that the Haar measure on $H_0(F_0)$ is chosen to give $K_0^\flat = H_0(F_0)$ volume one.  We similarly assign $H_1(F_0)$ volume one.
We also recall the orbital integrals for functions on the Lie algebras from \S\ref{setup lie}.

\begin{theorem}
\begin{altenumerate}
\item\label{germ i} Let $f'\in C_c^\infty(S)$ transfer to $(f_0,f_1)\in C_c^\infty(G_0)\times C_c^\infty(G_1)$, and let $\gamma_0=\diag(a_0,d_0)\in A_{S}$. If  $f_i=0$ for some $i\in\{0,1\}$, then there is the following germ expansion for $\gamma=\gamma(a,b)\in S_{\rs,i}$ in a neighborhood of $\gamma_0$,
$$
\omega_S(\gamma)\del(\gamma,f')=\frac{1}{2}\Orb\bigl(\diag(a_0,d_0),f_{1-i}\bigr)\log\lv 1-\RN a \rv + C,
$$
where $C$ is a constant depending on 
$\gamma_0$, $f'$, and $i$,
but not on $\gamma$. 
\item\label{germ ii} Let $f' \in C_c^\infty(\fks)$ transfer to $(f_0,f_1)\in C_c^\infty(\fkg_0)\times C_c^\infty(\fkg_1)$, and let $y_0=\diag (a_0,d_0)\in \fka_{\fks}$.  If $f_i=0$ for some $i\in\{0,1\}$, then there is the following germ expansion for $y=y(a, b, c, d) \in \fks(F_0)_{\rs,i}$ in a neighborhood of $y_0$,
$$
\omega_\fks(y)\del(y,f')=\frac{1}{2}\Orb\bigl(\diag(a_0,d_0),f_{1-i}\bigr)\log\lv bc \rv + C,
$$
where $C$ is again a constant which does not depend on $y$.
\end{altenumerate}\label{del germ s}
\end{theorem}

\begin{proof}
(\ref{germ i}) By \cite[Th.\ 3.5, Cor.\ 3.8]{M-ATC} and \eqref{omega_S n=2}, there is a germ expansion for regular semi-simple $\gamma = \gamma(a,b)$ in a neighborhood of $\gamma_0$,
\[
   \omega_S(\gamma)\Orb(\gamma,f',s) = 
	   \phi_+(\gamma_0,s)|b|^s 
		   +\phi_-(\gamma_0,s)\wt\eta\bigl(-\ov b\vphantom{b}^2\bigr) \eta(\RN a-1)\bigl\lv(1-\RN a)/\ov b\bigr\rv^{-s},
\]
where $\phi_\pm(\gamma_0,s)$ are Laurent polynomials in $q^{s}$ depending on $\gamma_0$ and $f'$.  Here $\lv\phantom{a}\rv$ denotes the natural extension of the normalized absolute value on $F_0$ to $F$.  For all $\gamma$ sufficiently near $\gamma_0$, the factor $\wt\eta(-\ov b\vphantom{b}^2) = \wt\eta(-\ov b/b)$ is constant-valued.  Hence, after possibly shrinking the neighborhood around $\gamma_0$ and replacing $\phi_-$ by a constant multiple of itself, we obtain
\begin{equation}\label{eqn germ s}
   \omega_S(\gamma)\Orb(\gamma,f',s) = \phi_+(\gamma_0,s)|b|^s +\phi_-(\gamma_0,s)\eta(\RN a-1)\bigl\lv(1-\RN a)/\ov b\bigr\rv^{-s}.
\end{equation}
Evaluation at $s=0$ yields
\begin{equation}\label{eqn germ s=0}
   \omega_S(\gamma)\Orb(\gamma,f') = \phi_+(\gamma_0)  +\phi_-(\gamma_0)\eta(\RN a-1),
\end{equation}
where $\phi_\pm(\gamma_0):=\phi_\pm(\gamma_0,0)$.

Similarly (and more simply), on the unitary side, the function $g \mapsto \Orb(g,f_i)$ is locally constant (and compactly supported) on $G_i(F_0)$.  In particular, for $g_0:=\diag(a_0, d_0)\in A_{G_i}$, there is a neighborhood of $g_0$ in $G_i(F_0)$ on which for all elements $g$,
$$
\Orb(g,f_i)=\Orb(g_0, f_i).
$$
Since $f'$ transfers to $(f_0,f_1)$, we obtain from this, \eqref{S_rs,i cond}, and \eqref{eqn germ s=0} that
\begin{equation}\label{eqn nil orb}
\begin{aligned}
   \Orb\bigl(\diag(a_0,d_0),f_0\bigr) &=\phi_+(\gamma_0)+\phi_-(\gamma_0),\\
	\Orb\bigl(\diag(a_0,d_0),f_1\bigr) &=\phi_+(\gamma_0)-\phi_-(\gamma_0).
\end{aligned}
\end{equation}

Now assume that $i=1$ in the statement of the theorem. Then by \eqref{eqn nil orb},
\begin{align}\label{eqn germ s u}
   \phi_+(\gamma_0) = \phi_-(\gamma_0) = \frac{1}{2}\Orb\bigl(\diag(a_0,d_0),f_0\bigr).
\end{align}
When $\gamma\in S_{\rs,1}$, i.e.\  $\eta(\RN a-1)=-1$, we may rewrite \eqref{eqn germ s} as
\[
   \omega_S(\gamma)\Orb(\gamma,f',s) =
	   |b|^s \bigl(\phi_+(\gamma_0,s)-\phi_-(\gamma_0,s)\bigr)+|b|^s\phi_-(\gamma_0,s)\bigl(1-|1-\RN a|^{-s}\bigr).
\]
Taking the derivative, we obtain
$$
\omega_S(\gamma)\del(\gamma,f')=C+\phi_-(\gamma_0)\log\lv 1-\RN a \rv,
$$
where $C:=\frac{d}{ds}\big|_{s=0}  (\phi_+(\gamma_0,s)-\phi_-(\gamma_0,s)) $ is a constant. The desired result now follows from \eqref{eqn germ s u}. The case $i=0$ is analogous.

(\ref{germ ii}) The proof is analogous to (\ref{germ i}).  Note that Th.\ 3.5 and Cor.\ 3.8 in \cite{M-ATC} are only stated in the group setting.  In the Lie algebra setting, one analogously proves that for all regular semi-simple $y = y(a,b,c,d)$ in a neighborhood of $y_0$,
\[
   \Orb(y,f',s) = \phi_+(y_0,s)\wt\eta(b)|b|^s + \phi_-(y_0,s)\wt\eta(c)\i |c|^{-s}
\]  
for some Laurent polynomials $\phi_\pm(y_0,s)$ in $q^s$ depending on $y_0$ and $f'$.  Also note that the Lie algebra transfer factor is given by $\omega_\fks(y) = \wt\eta(-b) = \wt\eta\bigl(\ov b\bigr)$.  With these remarks the proof of (\ref{germ i}) carries over in a straightforward way.
\end{proof}

For the specific test function $f_0=\mathbf{1}_{K_0^\flat K_0} \in C_c^\infty(G_0)$, the irregular orbital integrals are as follows. 

\begin{lemma}\label{lem nil orb int u}
The irregular orbital integrals  of $\mathbf{1}_{K_0^\flat K_0}$ are given by
$$
\Orb\bigl(\diag(a,d), \mathbf{1}_{K_0^\flat K_0}\bigr) = 1,\quad a,d\in F^1.
$$
The irregular orbital integrals  of $\mathbf{1}_{\fkk_0}$ are given by
$$
\Orb\bigl(\diag(a,d), \mathbf{1}_{\fkk_0}\bigr) = \begin{cases} 1,& a,d\in \pi O_{F_0},\\
0, &\mbox{otherwise}.
\end{cases}
$$
\end{lemma}

\begin{proof}
This is immediate from the fact that the conjugation action of $H_0(F_0)$ on diagonal matrices is trivial, from the normalization $\vol H_0(F_0) = 1$, and from \eqref{A_G_0 intersect} and \eqref{a_g_0 intersect}.
\end{proof}

\subsection{Intersection numbers}

We now record the values of the intersection numbers appearing in parts (\ref{conjram v3 inhomog}) and (\ref{conjram v3 lie}) of Conjecture \ref{conjram v3}.  Recall that $v$ denotes the normalized valuation on $F_0$.

\begin{proposition}\label{prop int(g) int(x)}
Let $g= \begin{bmatrix}
		   1 &     \\
         &  \alpha
      \end{bmatrix} 
\begin{bmatrix}
		   a  & b   \\
           b&  a
      \end{bmatrix}
\in G_1(F_0)_\rs$, expressed in the presentation (\ref{G_1 pres}). Then 
\[
   \Int(g)= 2v(\RN b) + 2.
\]
Similarly, for $x= 
\begin{bmatrix}
		   a  & b   \\
           b&  d
      \end{bmatrix}
\in \fkg_1(F_0)_\rs$ expressed in the presentation (\ref{fkg_1 pres}), 
\[
   \lInt(x)=
	\begin{cases}
      2v(\RN b) + 2,  &  a,d\in \pi O_{F_0} \text{ and } b\in O_D \cap D^-;\\
      0,  &  \text{otherwise}.
   \end{cases}
\]
\end{proposition}

\begin{proof}
We first consider the group version.  Note that under the parametrization \eqref{G_1 pres}, $a + b$ is a norm one element in $D$, and hence $a \in O_F$ and $b \in O_D \cap D^-$ are both integral.

The coordinates \eqref{G_1 pres} express $g$ naturally as a quasi-endomorphism of $\wt\BX_2$; the quasi-endo\-morphism of $\BX_2$ attached to $g$ is the conjugate $\phi_0\i g \phi_0$, where $\phi_0\colon \BX_2 \to \wt\BX_2$ is the isogeny \eqref{phi_0}.  With respect to the $O_{F_0}$-linear decompositions $\BX_2 = O_F \otimes_{O_{F_0}} \BE = (1 \otimes \BE) \oplus (\pi \otimes \BE)$ and $\wt\BX_2 = \BE^2$, the element $\phi_0 \in \Hom_{O_F}^\circ (\BX_2,\wt\BX_2) \subset \End_{O_{F_0}}^\circ(\BE^2) = \RM_2(D)$ identifies with the matrix
\[
   \begin{bmatrix}
		1  &  \pi\\
		1  &  -\pi
	\end{bmatrix}.
\]
Hence the conjugate $\phi_0\i g \phi_0$ identifies with the matrix
\begin{equation}\label{g conj}
   \begin{bmatrix}
		1  &  \pi\\
		1  &  -\pi
	\end{bmatrix}\i
	\begin{bmatrix}
		a  &  b\\
		\alpha b  &  \alpha a
	\end{bmatrix}
   \begin{bmatrix}
		1  &  \pi\\
		1  &  -\pi
	\end{bmatrix}
	=
   \frac 1 2
	\begin{bmatrix}
		(1 + \alpha)(a+b)  &  \pi(1-\alpha)(a+b)\\
		\pi\i(1-\alpha)(a+b)  &  (1+\alpha)(a+b)
	\end{bmatrix}.
\end{equation}

Now recall from Example \ref{CN_2 ATC eg} that $\CN_{1,2} \cong \CN_2 = \CN_2^+ \amalg \CN_2^-$, where each of the summands $\CN_2^\pm$ identifies with $\CM_{\Spf O_{\breve F}}$, and the cycle $\Delta$ identifies with the canonical divisor attached to the embedding $\iota_\BE\colon F \inj D$ in each copy of $\CM_{\Spf O_{\breve F}}$.  More precisely, under these identifications, the universal object over $\Delta \cap \CN_2^+$ is $\delta_\CN^+(\CE) = (O_F \otimes_{O_{F_0}} \CE, O_F \otimes \rho_\CE)$, and the universal object over $\Delta \cap \CN_2^-$ is $\delta_\CN^-(\CE) = (O_F \otimes_{O_{F_0}} \CE, \kappa_0 \circ(O_F \otimes \rho_\CE))$; here we have suppressed auxiliary structure in the obvious way, and
\[
   \kappa_0 :=
	\begin{bmatrix}
		  &  \pi\\
		\pi\i
	\end{bmatrix}
	\in \Aut_{O_F}^\circ (\BX_2) \subset \RM_2(D)
\]
is a self-quasi-isogeny of $\BX_2$ which lies in the unitary group $G_1(F_0)$ and has Kottwitz invariant $-1$ (so that the action of $\kappa_0$ on $\CN_2$ interchanges the components $\CN_2^\pm$).

Now we compute $\Int(g)$. First note that, as in \cite[Prop.\ 8.10]{RSZ}, there are no higher Tor terms in the calculation.  Thus $\Int(g)$ is simply the length of $\Delta \cap (1 \times g)\Delta$, which in turn is twice the length of $\Delta \cap (1\times g)\Delta \cap \CN_2^+$.  This last intersection is the locus in $\Spf O_{\breve F}$ where $\phi_0\i g \phi_0$ lifts to a homomorphism $\delta_\CN^+(\CE) \to \delta_\CN^+(\CE)$ or $\delta_\CN^-(\CE) \to \delta_\CN^+(\CE)$, or equivalently, where $\phi_0\i g \phi_0$ or $\phi_0\i g \phi_0 \kappa_0$ lifts to a homomorphism $\delta_\CN^+(\CE) \to \delta_\CN^+(\CE)$.  In other words, this is the locus where all the entries of the right-hand side of \eqref{g conj}, or all the entries of
\begin{equation}\label{g conj kappa}
   \phi_0\i g \phi_0 \kappa_0 = \frac 1 2
	\begin{bmatrix}
		(1-\alpha)(a-b)  &  \pi(1 + \alpha)(a-b)\\
		\pi\i(1+\alpha)(a-b)  &   (1-\alpha)(a-b)
	\end{bmatrix},
\end{equation}
lift to endomorphisms of the canonical lifting \CE.  In the group setting, only one of these matrices will have all entries integral to begin with; this is governed by the Kottwitz invariant of $g$, that is, by the residue class of the norm one element $\alpha$ mod $\pi$.  Indeed, suppose that $\alpha \equiv 1 \bmod \pi$.  Then $1+\alpha \in O_F^\times$ since $p \neq 2$.  Since $\RN(a - b) = \RN a + \RN b = 1$, we conclude that the lower left entry in \eqref{g conj kappa} is non-integral.  Continuing to assume $\alpha \equiv 1$, we evidently have that $\pi \mid 1 -\alpha$, and therefore the locus where the entries of \eqref{g conj} lift is the locus where $a + b$ lifts.  By Gross's formula \cite[Th.\ 2.1]{V1}, the length of this locus is $\ell + 1$, where $\ell$ is the nonnegative integer such that $a + b \in (O_F + \pi^\ell O_D)\smallsetminus (O_F + \pi^{\ell + 1}O_D)$;  comp.\ also \cite[Prop.\ 9.1]{RSZ}.  Since $a \in O_F$ and $b \in O_D \cap D^-$, the asserted formula for $\Int(g)$ follows.

If $\alpha \equiv -1 \bmod \pi$, then one similarly finds that the lower left entry of \eqref{g conj} is non-integral, and that the locus where the entries of \eqref{g conj kappa} lift is the locus where $a - b$ lifts, which has the same length as before.  This completes the proof in the group case.

The Lie algebra case is quite similar, and we only briefly outline the differences.  The matrices \eqref{g conj} and \eqref{g conj kappa} are replaced by
\[
   \frac 1 2
   \begin{bmatrix}
		a + d + 2b  &  \pi(a-d)\\
		\pi\i(a-d)  &  a + d + 2b
	\end{bmatrix}
	\quad\text{and}\quad
	\frac 1 2
	\begin{bmatrix}
		a-d  &  \pi(a+d-2b)\\
		\pi\i(a+d-2b)  &  a-d
	\end{bmatrix},
\]
respectively, and $\lInt(x)$ is twice the length of the locus where all the entries of at least one of these matrices lift to endomorphisms of \CE.  If any one of $a$, $b$, or $d$ is non-integral, then it is easy to see (again using that $p \neq 2$) that both of these matrices have a non-integral entry, and hence $\lInt(x) = 0$.  If all of them are integral, then the locus where the entries in the second matrix lift is contained in the locus where the entries in the first matrix lift, and the length of this latter locus is again given by Gross's formula as $v(\RN b) + 1$.
\end{proof}

\subsection{Proof of Conjecture \ref{conjram v3} for $F_0 = \BQ_p$ and $n=2$}\label{proof ram ATC n=2}
The homogeneous group version (\ref{conjram v3 homog}) reduces to the inhomogeneous version (\ref{conjram v3 inhomog}) as in \cite[Lem.\ 5.8]{RSZ} (using Lemma \ref{homog to inhomog}(\ref{all})(\ref{old})(\ref{old 2}), which is valid for arbitrary $n$ and arbitrary $F/F_0$, in place of \cite[Lem.\ 5.7]{RSZ} in the proof). Thus we show that (\ref{conjram v3 inhomog}) holds.  To prove the first assertion in (\ref{conjram v3 inhomog}), fix $\gamma_0=\diag(a_0,d_0) \in A_{S}=A_{G_0}$. It suffices to show that the sum
\begin{equation}\label{the sum}
   2\omega_S(\gamma)\del(\gamma,f')+\Int(g)\log q
\end{equation}
is constant when $\gamma\in S_{\rs,1}$ is in a neighborhood of $\gamma_0$ (and $g \in G_1(F_0)_\rs$ is any match of $\gamma$). It then follows from, for example, \cite[Cor.~3.8]{M-ATC} that \eqref{the sum} is an orbital integral function, which is what we have to show.

We have for $f_0=\mathbf{1}_{K_0^\flat K_0^+}+ \mathbf{1}_{K_0^\flat K_0^-}$ and $g \in G_0(F_0)$,
$$
\Orb(g,f_0)=2\Orb(g,\mathbf{1}_{K_0^\flat K_0^+})=2\Orb(g,\mathbf{1}_{K_0^\flat K_0^-}).
$$
We may therefore replace $f_0$ by $2\cdot \mathbf{1}_{K_0^\flat K_0}$, cf.\ Remark \ref{remark S12}. 
By Theorem \ref{del germ s} (in the case $i=1$) and Lemma \ref{lem nil orb int u}, there exists a neighborhood of $\gamma_0$ on which for all $\gamma = \gamma(a,b) \in S_{\rs,1}$,
$$
2\omega_S(\gamma)\del(\gamma,f')=2\log\lv 1-\RN a\rv +C=-2 v(1-\RN a)\log q+C
$$
for some constant $C$.  Proposition \ref{prop int(g) int(x)} then yields immediately that the sum 
\eqref{the sum}
is a constant for all such $\gamma$, as desired.  The second assertion in (\ref{conjram v3 inhomog}), namely that there exists an $f'$ transferring to $(\mathbf{1}_{K_0^\flat K_0^+}+ \mathbf{1}_{K_0^\flat K_0^-}, 0)$ for which $f_\corr' = 0$, now follows by the same argument as in \cite[Prop.\ 5.14]{RSZ}.

The Lie algebra version \eqref{conjram v3 lie} for $n=2$ follows similarly, by the corresponding statements in Theorem \ref{del germ s}, Lemma \ref{lem nil orb int u}, and  Proposition \ref{prop int(g) int(x)}.\qed

\section{Proofs of Theorems \ref{conjram i=0} and \ref{conjram i=1}}\label{s:ramifiedn=2}

\subsection{Proof of Theorem \ref{conjram i=0}}\label{proof wtCN_2^(0)}
We resume the setup of \S\ref{ATC wtCN_2^(0)}.  Let us recall the groups and their Lie algebras. 
The quasi-split unitary group $G_0$ is associated to the framing object $\wt\BX_2^{(0)}$ of $\wt\CN_2^{(0)}$, cf.\ \eqref{wtBX_2^(0)}. We identify 
\[
   \End_{O_{F_0}}^\circ\bigl(\wt\BX_2^{(0)}\bigr) = \End_{O_{F_0}}^\circ\bigl(\BE^2\bigr)=\RM_2(D),
\] 
and the $O_F$-action is given by $\pi\mapsto \diag(\pi,-\pi)\in \RM_2(D)$. Then the $O_F$-linear quasi-endomorphism algebra is 
$$
\End_{O_F}^\circ\bigl(\wt\BX_2^{(0)}\bigr)= \biggl\{\,  \begin{bmatrix}
		   a  & b    \\
       c    &  d
      \end{bmatrix} \in \RM_2(D) \biggm|   a,d\in F,\ b,c\in D^-  \,\biggr\} .
$$
The Rosati involution is given by 
$$
\Ros_{\lambda_{\wt\BX_2^{(0)}}}(x)=x^\dagger=  \begin{bmatrix}
		   1 &     \\
         &  \ep
      \end{bmatrix}  ^{-1}   \tensor[^t]{\ov x}{}   \begin{bmatrix}
		   1 &     \\
         &  \ep
      \end{bmatrix},
$$
where $x\mapsto \ov x$ is the entry-wise main involution on $\RM_2(D)$.  Recall that $\ep \in O_{F_0}^\times$ is a non-norm. We then have the unitary group 
$$
G_0(F_0)=\bigl\{\, x\in \End_{O_F}^\circ\bigl(\wt\BX_2^{(0)}\bigr) \bigm|  x^\dagger x=1 \,\bigr\},
$$
and its Lie algebra 
$$
\fkg_0(F_0)=\bigl\{\, x\in \End_{O_F}^\circ\bigl(\wt\BX_2^{(0)}\bigr) \bigm|  x^\dagger+ x=0\,\bigr\} .
$$
They may be explicitly presented as
$$G_0(F_0)=\biggl\{ \, \begin{bmatrix}
		   1 &     \\
         &  \alpha
      \end{bmatrix}    \begin{bmatrix}
		   a  & b    \\
          \ep^{-1}  b  &  a
      \end{bmatrix} \in \RM_2(D)  \biggm|  a\in F,\ b\in D^-,\ \RN a+\ep^{-1}\RN b=1,\ \alpha \in F^1\,\biggr\} ,
$$
and $$
\fkg_0(F_0)=\biggl\{\, \begin{bmatrix}
		   a  & b    \\
       \ep^{-1} b  &  d
      \end{bmatrix} \in \RM_2(D)  \biggm|  a,d\in \pi F_0,\ b\in D^-  \, \biggr\} .
$$

The matching of orbits is analogous to \S\ref{orbit matching}: in terms of the parametrization \eqref{gamma(a,b)} for elements in $S(F_0)_\rs$, we have
$$
   \gamma(a,b)\in S_{\rs, 0}  \quad \text{\emph{matches}} \quad g=\begin{bmatrix}
		   1 &     \\
         &  \alpha
      \end{bmatrix}    \begin{bmatrix}
		   a'  & b'    \\
      \ep^{-1} b'   &  a'
      \end{bmatrix}\in G_0(F_0)_\rs
$$
if and only if $a = a'$ and $-b/\ov b=\det\gamma(a,b) = \det g = \ov\alpha$. In the Lie algebras, in terms of the parametrization \eqref{y(a,b,c,d)} for elements in $\fks(F_0)$, we have 
$$
   y(a,b,c,d) \in \fks_{\rs, 0}
	\quad\text{\emph{matches}}\quad 
	x= \begin{bmatrix}
		   a'  & b'    \\
     \ep^{-1}  b'   &  d'
      \end{bmatrix}\in \fkg_0(F_0)_\rs
$$
if and only if  $a=a'$, $d = -d'$, and $bc=-\ep^{-1}\RN b'$.

The unitary group $G_1$ is associated to to the non-split hermitian space
\[
   W_1=W_1^\flat \oplus F u_1,\quad  (u_1,u_1)=-\ep.
\]
By \eqref{chi decomp formula} the space $W_1^\flat$ is split, and we choose a basis vector $u^\flat \in W_1^\flat$ with $(u^\flat, u^\flat) = 1$.  In this way we view $G_1$ and its Lie algebra $\fkg_1$ as attached to the diagonal hermitian form $\diag(1,-\ep)$.
The irregular elements in the $F_0$-points of each then consist of the diagonal matrices contained in each.
  Now recall from \S\ref{ATC wtCN_2^(0)} the self-dual lattice $\Lambda_1 \subset W_1$ and its respective stabilizers $\wt K_1 \subset G_1(F_0)$ and $\wt\fkk_1 \subset \fkg_1(F_0)$.

\begin{lemma}\label{lem nil orb int u 0}
The irregular orbital integrals  of $\mathbf{1}_{\wt K_1}$ are given by
$$
\Orb\bigl(\diag(a,d), \mathbf{1}_{\wt K_1}\bigr)=1,\quad a,d\in F^1.
$$
The irregular orbital integrals  of $\mathbf{1}_{\wt\fkk_1}$ are given by
$$
\Orb\bigl(\diag(a,d), \mathbf{1}_{\wt\fkk_1}\bigr)=\begin{cases} 1,& a,d\in \pi O_{F_0},\\
0, &\mbox{otherwise}.
\end{cases}
$$
\end{lemma}

\begin{proof}
This is the same as the proof of Lemma \ref{lem nil orb int u}, noting that $\wt K_1$ contains the irregular group elements, and that the intersection of $\wt\fkk_1$ with the irregular Lie algebra elements is the set $\{\diag(a,d) \mid a,d \in \pi O_{F_0}\}$.
\end{proof}

\begin{proposition}\label{prop int(g) int(x) 0}
Let $g= \begin{bmatrix}
		   1 &     \\
         &  \alpha
      \end{bmatrix} 
\begin{bmatrix}
		   a  & b   \\
             \ep^{-1}  b&  a
      \end{bmatrix}
$ be an element in $G_0(F_0)_\rs$. Then 
\[
   \Int(g)=
	\begin{cases}
      v(\RN b) + 1,  &  a\in O_F \text{ and } b\in O_D \cap D^-;\\
      0,  &  \text{otherwise}.
   \end{cases}
\]
Similarly,  for an element $x= 
\begin{bmatrix}
		   a  & b   \\
              \ep^{-1} b&  d
      \end{bmatrix}
$  in $\fkg_0(F_0)_\rs$, 
\[
   \lInt(x)=
	\begin{cases}
      v(\RN b) + 1,  &  a,d\in \pi O_{F_0} \text{ and } b\in O_D \cap D^-;\\
      0,  &  \text{otherwise}.
   \end{cases}
\]
\end{proposition}
\begin{proof}
The method of proof is the same as that of Proposition \ref{prop int(g) int(x)}, except this time the details are much simpler.  In the group case, $\Int(g)$ is the length of the locus in $\Spf O_{\breve F}$ where $\bigl[\begin{smallmatrix} a  &  b\\ \alpha \ep\i b &  \alpha a \end{smallmatrix}\bigr]$ lifts from an endomorphism of $\BE \times \ov\BE$ to an endomorphism of $\CE \times \ov\CE$, or in other words, where each of the entries lifts to an endomorphism of $\CE$.  If $a$ or $b$ is non-integral, then this locus is empty and we obtain $\Int(g) = 0$.  If $a$ and $b$ are integral,  then $a,\alpha,\ep\i \in O_F$ lift without constraint, and the length in question is again given by Gross's formula as $v(\RN b) + 1$.  The Lie algebra case is analogous.  We remind the reader that the factor of $2$ in Proposition \ref{prop int(g) int(x)} is due to the fact that there are two copies of $\CM$ involved there.
\end{proof}

Now we complete the proof of the theorem.

\begin{proof}[Proof of Theorem \ref{conjram i=0}] 
We first prove the group version (\ref{conjram i=0 inhomog}); the argument is formally the same as the proof of Conjecture \ref{conjram v3}(\ref{conjram v3 inhomog}) for $n=2$ in \S\ref{proof ram ATC n=2}. To prove the first assertion in (\ref{conjram i=0 inhomog}),
fix $\gamma_0=\diag(a_0,d_0)\in A_S$, which we also regard as an irregular element in $G_1(F_0)$. It again suffices to show that the sum
\begin{equation}\label{sum eq0}
   2\omega_S(\gamma)\del(\gamma,f')+\Int(g)\log q
\end{equation}
is  constant when $\gamma\in S_{\rs,0}$ is in a neighborhood of $\gamma_0$ (and $g \in G_0(F_0)_\rs$ is any match of $\gamma$). 
By Theorem \ref{del germ s} (in the case $i=0$) and Lemma \ref{lem nil orb int u 0}, there exists such a neighborhood on which for all $\gamma = \gamma(a,b) \in S_{\rs,0}$,
$$
2\omega_S(\gamma)\del(\gamma,f')=\log\lv 1-\RN a\rv +C=-v(1-\RN a)\log q+C
$$
for some constant $C$.
Proposition \ref{prop int(g) int(x) 0} then yields that \eqref{sum eq0}
is constant for such $\gamma$, as desired.  The second assertion in (\ref{conjram i=0 inhomog}) again follows by the argument in \cite[Prop.\ 5.14]{RSZ}.

The Lie algebra version (\ref{conjram i=0 lie}) follows similarly, by the corresponding statements in Theorem \ref{del germ s}, Lemma \ref{lem nil orb int u 0}, and  Proposition \ref{prop int(g) int(x) 0}.
\end{proof}

\subsection{Proof of Theorem \ref{conjram i=1}}\label{proof wtCN_2^(1)}
Now we resume the setup of \S\ref{atc wtCN_2^(1)}.  This case is very similar to the $\pi$-modular case in \S\ref{even ram proof}.
The groups are the same as in \S\ref{even ram proof}. 
In particular, the irregular elements in $G_0(F_0)$ are the subset $A_{G_0}$ \eqref{eqn U_2 A}.
Recall that the compact open subgroups $\wt K_0 \subset G_0(F_0)$ and $\wt\fkk_0 \subset \fkg_0(F_0)$ are the respective stabilizers of the self-dual lattice $\Lambda_0 \subset W_0$.  By the same argument as in the proofs of Lemmas \ref{lem nil orb int u} and \ref{lem nil orb int u 0}, we obtain the following.

\begin{lemma}\label{lem nil orb int u 1}
The irregular orbital integrals  of $\mathbf{1}_{\wt K_0}$ are given by
$$
\Orb\bigl(\diag(a,d), \mathbf{1}_{\wt K_0}\bigr)=1,\quad a,d\in F^1.
$$The irregular orbital integrals  of $\mathbf{1}_{\wt\fkk_0}$ are given by
\begin{flalign*}
	\phantom{\qed} & &
	\Orb\bigl(\diag(a,d), \mathbf{1}_{\wt\fkk_0}\bigr)=\begin{cases} 1,& a,d\in \pi O_{F_0},\\
	0, &\text{otherwise}.
	\end{cases}
	& & \qed
\end{flalign*}
\end{lemma}

\begin{proposition}\label{prop int(g) int(x) 1}
Let $g= \begin{bmatrix}
		   1 &     \\
         &  \alpha
      \end{bmatrix} 
\begin{bmatrix}
		   a  & b   \\
           b&  a
      \end{bmatrix}
\in G_1(F_0)_\rs$, expressed in the presentation (\ref{G_1 pres}). Then 
\[
   \Int(g)= v(\RN b) + 1.
\]
Similarly, for $x= 
\begin{bmatrix}
		   a  & b   \\
           b&  d
      \end{bmatrix}
\in \fkg_1(F_0)_\rs$ expressed in the presentation (\ref{fkg_1 pres}),
\[
   \lInt(x)=
	\begin{cases}
      v(\RN b) + 1,  &  a,d\in \pi O_{F_0} \text{ and } b\in O_D \cap D^-;\\
      0,  &  \text{otherwise}.
\end{cases}
\]
\end{proposition}
\begin{proof}
The proof is virtually identical to that of Proposition \ref{prop int(g) int(x) 0}.  Note that in the group case, as in Proposition \ref{prop int(g) int(x)}, the entries $a$ and $b$ are automatically integral.
\end{proof}

Now we complete the proof of the theorem.

\begin{proof}[Proof of Theorem \ref{conjram i=1}]
The proof is essentially the same as the proofs of Conjecture \ref{conjram v3}(\ref{conjram v3 inhomog})(\ref{conjram v3 lie}) when $n=2$ and Theorem \ref{conjram i=0}. We first prove the group version (\ref{conjram i=1 inhomog}).
For the first assertion, fix $\gamma_0=\diag(a_0,d_0)\in A_{S}$, which we also regard as an irregular element in $G_0(F_0)$. As before, it suffices to show that the sum
\begin{equation}\label{sum eq}
2\omega_S(\gamma)\del(\gamma,f')+\Int(g)\log q
\end{equation}
is  constant when $\gamma\in S_{\rs,1}$ is in a neighborhood of $\gamma_0$ (and $g$ matches $\gamma$). 
By Theorem \ref{del germ s} (in the case $i = 1$) and Lemma \ref{lem nil orb int u 1}, there exists such a neighborhood on which for all $\gamma = \gamma(a,b) \in S_{\rs,1}$,
$$
2\omega_S(\gamma)\del(\gamma,f')=\log\lv 1-\RN a\rv+C=-v(1-\RN a)\log q+C
$$
for some constant $C$. Proposition \ref{prop int(g) int(x) 1} then yields that \eqref{sum eq}
is constant for such $\gamma$, as desired.  The second assertion in (\ref{conjram i=1 inhomog}) again follows by the argument in \cite[Prop.\ 5.14]{RSZ}.

The Lie algebra version (\ref{conjram i=1 lie}) follows similarly, by the corresponding statements in Theorem \ref{del germ s}, Lemma \ref{lem nil orb int u 1}, and  Proposition \ref{prop int(g) int(x) 1}.
\end{proof}

\end{document}